%% file: culminants7.tex
\documentclass[10pt,a4paper,reqno]{amsart}

\usepackage{latexsym}
\usepackage{verbatim}
\usepackage{epsfig}
\usepackage{rotating}
\usepackage{amssymb}
\usepackage[T1]{fontenc}
\usepackage{afterpage}
\usepackage{color}
\usepackage{dsfont}

\usepackage{amssymb,amsfonts,amsmath,stmaryrd}

\catcode`\@=11
\def\section{\@startsection{section}{1}%
 \z@{.7\linespacing\@plus\linespacing}{.5\linespacing}%
 {\normalfont\bfseries\scshape\centering}}
\def\subsection{\@startsection{subsection}{2}%
	\z@{.5\linespacing\@plus\linespacing}{.5\linespacing}%
	{\normalfont\bfseries\scshape}}

\def\subsubsection{\@startsection{subsubsection}{3}%
 \z@{.5\linespacing\@plus\linespacing}{-.5em}%{.5\linespacing}%
	{\normalfont\bfseries\itshape}}
\catcode`\@=12

\newcommand\mir[1] {\overset{\leftarrow}{#1}}

\newtheorem{theorem}{Theorem}[section]

\newtheorem{proposition}[theorem]{Proposition}
\newtheorem{definition}[theorem]{Definition}
\newtheorem{corollary}[theorem]{Corollary}
\newtheorem{property}[theorem]{Property}

\def\qed{$\hfill{\vrule height 3pt width 5pt depth 2pt}$}

\newfont{\bbold}{msbm10 scaled \magstep1}
\newfont{\bbolds}{msbm7 scaled \magstep1}

\newcommand{\ns}{\mbox{\bbold N}}

\newcommand{\zs}{\mbox{\bbold Z}}
\newcommand{\zss}{\mbox{\bbolds Z}}
\newcommand{\qs}{\mbox{\bbold Q}}

\newcommand{\cs}{\mbox{\bbold C}}

\newcommand{\Sn}{\mathfrak S}
\newcommand{\mb}{\overline m}

\newcommand{\cC}{\mathcal C}
\newcommand{\D}{\mathcal D}
\newcommand{\E}{\mathcal E}

\newcommand{\I}{\mathcal I}
\newcommand{\La}{\mathcal L}

\newcommand{\V}{\mathcal V}
\newcommand{\W}{\mathcal W}
\newcommand{\M}{\mathcal M}

\newcommand{\cH}{\mathcal H}
\newcommand{\K}{\mathcal K}
\newcommand{\cL}{\mathcal L}
\newcommand{\cR}{\mathcal R}

\newcommand{\PP}{\mathcal P}
\newcommand{\cP}{\mathcal P}
\newcommand{\X}{\mathcal X}
\newcommand{\U}{\mathcal U}
\newcommand{\C}{\mathcal C}
\newcommand{\G}{\mathcal G}

\newcommand{\beq}{\begin{equation}}
\newcommand{\eeq}{\end{equation}}
\newcommand{\gf}{generating function}
\newcommand{\gfs}{generating functions}

\def\emm#1,{{\em #1}}

\newcommand{\si}{\sigma}
\newcommand{\la}{\lambda}

\begin{document}
\title[Culminating paths]
{Culminating paths}

\author{Mireille Bousquet-M\'elou\and Yann Ponty
}

\address{CNRS, LaBRI, Universit\'e Bordeaux 1, 351 cours de la Lib\'eration,
	33405 Talence Cedex, France
\and LRI, Bât 490 Université Paris-Sud
91405 Orsay Cedex France }
\email{mireille.bousquet@labri.fr  \and yann.ponty@lri.fr}
\thanks{MBM was  partially supported by the French ``Agence Nationale
de la Recherche'', project SADA ANR-05-BLAN-0372. YB was partially supported by the action "Aspects mathématiques et algorithmiques des réseaux
	biochimiques et évolutifs ($\pi$-vert)" of ACI Nouvelles Interfaces des
	Mathématiques, French Ministry of Research.
	}

%*********************** Abstract ******************************
\begin{abstract}
Let $a$ and $b$ be two positive integers.
A  culminating path is a path of $\zss^2$ that starts from
$(0,0)$, consists  of steps $(1,a)$
and $(1,-b)$, stays above the $x$-axis and ends at the highest ordinate
it ever reaches.  These paths were first encountered
in bioinformatics, in the analysis of
 similarity search algorithms.
They are also related to
certain models of Lorentzian gravity in theoretical physics.

We first show that the language on a two letter alphabet that
naturally encodes culminating paths is not context-free.

Then, we focus on the enumeration of culminating paths.
A step by step approach, combined with the kernel method, provides a closed form expression for the
\gf\  of culminating paths ending at a (generic) height $k$.  In the case $a=b$, we derive from
this expression the asymptotic behaviour of the number of culminating paths of length $n$. When
$a>b$, we obtain the asymptotic behaviour
 by a simpler argument. When  $a<b$,
we only determine the exponential growth  of the number of
culminating paths. 

Finally, we study  the uniform random generation of culminating
paths via various methods.
The rejection approach, coupled with a symmetry argument,
gives an algorithm that is linear when $a\ge b$, with no
precomputation stage nor non-linear storage required.
The choice of the best algorithm is not as clear when $a<b$. An
elementary  recursive approach  yields a linear algorithm after a
precomputation stage involving ${O}(n^3)$ arithmetic operations, but
we also present some alternatives  that may be more efficient in practice.
\end{abstract}
\maketitle
%\subjclass{Primary: subject; Secondary: subject}
\date{\today}

%%%%%%%%%%%%%%%%%%%%%%%%%%%%%%%%%%%%%%%%%%%%%%%%%%%%%%%%%%%%

\section{Introduction}
%%%%%%%%%%%%%%%%%%%%%%%%%%%%%%%%%%%%%%%%%%%%%%%%%%%%%%%%%%%%
One-dimensional lattice walks on $\zs$ %(or lattice paths)
have been extensively studied over
the past 50 years. These walks usually start from the point $0$, and
take their steps in a prescribed
finite set $\mathcal S \subset \zs$. A large number of results are
now known on the enumeration
of sub-families of these walks, and can be obtained in a systematic
way once the set $\mathcal S$ is given. This includes the enumeration
of \emm bridges, (walks ending at 0), \emm meanders, (walks that always
remain at a non-negative level), \emm excursions, (meanders ending at
level 0), excursions of bounded height, and so
on. In particular, the nature of the associated \gfs\ is well
understood: these series are always algebraic, and even rational for
bounded
walks~\cite{AZ2,BaFl02,mbm-excursions,bousquet-petkovsek-recurrences,Duchon98,gessel-factorization,labelle,labelle-yeh,stanley-vol2}.
These algebraicity properties  actually reflect the fact that the
languages on the alphabet $\mathcal S$ that naturally encode these
families of walks are \emm context-free,, and even \emm regular, in the
bounded case.  In many papers, these
one-dimensional walks are actually described as \emm directed,\,
two-dimensional (2D) walks, upon replacing the starting point $0$ by $(0, 0)$
and every step $s$ by $(1,s)$. This explains why
excursions are often called \emm generalized Dyck paths, (the
authentic Dyck paths correspond to the case $\mathcal
S=\{1,-1\}$). This two-dimensional setting allows for a further
generalisation, with steps
of the form $(i,j)$, with $i>0$ and $j \in \zs$, but this does not
affect the nature of the associated languages and \gfs.
The uniform random generation of these walks has also been
investigated, through a recursive
approach~\cite{wilf77,flajoletcalculusINRIA,fullboltz} 
% mbm added ref. to flajolet-zimmermann-van-cutsem and Boltzmann
or using an anticipated rejection~\cite{BaPiSp92,lou97}.

This paper deals with
a new class of walks which has recently occurred in two independent
contexts, and
seems to have a more complicated structure than the above mentioned
classes: \emm culminating
walks,. A 2D directed walk is said to be culminating if each step ends
at a positive level, and the
final step ends at the highest level ever reached by the walk
(Figure~\ref{fig1}). 
%Alternatively, a culminating walk  can be seen as a walk lying in a
%strip whose height is only revealed at the final step. 
We focus here on the case where the 
%only allowed
 steps are $(1,a)$ and $(1,-b)$, with $a$ and
$b$ positive, hoping that this encapsulates all the possible typical behaviours.

In the case $a=b=1$, culminating walks have recently been shown to be
in bijection with certain
\emm  Lorentzian triangulations,~\cite{DiFr00}, a class of
combinatorial objects studied in
theoretical physics as a model of discrete two-dimensional Lorentzian
gravity. Using a transfer
matrix approach, the authors derived the generating function for this
case. We give two shorter
proofs of their result. Also, while it is not clear how  the
method used in~\cite{DiFr00}
could be extended to the general  $(a,b)$-case, one of our approaches
works for arbitrary values of $a$ and $b$.

The general $(a,b)$-case appears in bioinformatics in the study of the
sensitivity of heuristic homology search algorithms, such as BLAST,
FASTA or FLASH~\cite{blast90,fasta88,flash93}. These algorithms aim at
finding the
 most conserved regions
(\emm similarities,\,)
between two genomic sequences (DNA, RNA, proteins...) while allowing certain
alterations in the entries of the sequences. In order to avoid the supposedly intrinsic quadratic
complexity of the deterministic algorithms, these heuristic algorithms first consider identical
regions of bounded size and extend them in both directions, updating the score with a bonus for a
match or a penalty for an alteration, until the score drops below a certain threshold. The
evolution of the score all the way through the final alignment turns out to be encoded by a
culminating walk.

In \cite{KuNoPoBIBE04}, we first studied the probability of a
culminating walk to contain certain 
patterns called \emm seeds,, as some recent algorithms make use of them to relax the mandatory
conservation of small \emph{anchoring} portions.
Then, we proposed
a variant of the recursive approach for the random generation of these walks.
Finally, we observed
that the naive rejection-based  algorithm, which consists in drawing
uniformly at random up and 
down steps  and rejecting the resulting walk if is not culminating,
seemed to be linear (resp. 
exponential) when $a>b$ (resp. $a<b$). This observation, which is
closely related to the asymptotic 
enumeration of culminating walks, is confirmed below in Section~\ref{sec:rejection-culminating}.

\begin{figure}[tb]
	\begin{center}
	\scalebox{0.65}{\input{PathToWord.pstex_t}}
	\end{center}
\caption{A culminating path (for $a=5$ and $b=3$) and the corresponding word.}
\label{fig1}
\end{figure}

\medskip
To conclude this introduction, let us fix the notation and summarize
the contents of this paper. Let $a$ and $b$ be two positive integers.
A walk (or path) of length $n$ is a sequence
$(0,\eta_0),\ldots,(n,\eta_{n})$  such that $\eta_0=0$ and
$\eta_{i+1}-\eta_i \in
\{a,-b\}$ for all $i$. The \emm height, of the walk is the largest
of the $\eta_i$'s, while the \emm final height, is $\eta_n$.
The walk is \emm culminating, if the two following conditions hold:
$$
\forall i \in [1,n],  \quad \eta_i>0
\quad \mbox{(Positivity),}
$$
$$
\forall i \in [0,n-1],  \quad \eta_i<\eta_n
\quad  \mbox{(Final record)}.
$$
See Figures~\ref{fig1} and~\ref{fig2} for examples  and counter-examples.
We encode every walk by a word on the alphabet $\{m, \overline m\}$ in
a standard way: each ascending step $(1,a)$ is replaced by a letter
$m$ and each descending step $(1,-b)$ is replaced by a letter
$\overline{m}$. We denote by $\{m, \mb\}^*$ the set of words on the
alphabet $\{m, \mb\}$.
From now on, we identify a path and the corresponding word.  Since
these objects are essentially one-dimensional, we will often use a 1D
vocabulary, saying, for instance, that our paths take steps $+a$ and
$-b$ (rather than $(1,a)$ and $(1,-b)$). We hope that this will not
cause any confusion. Without loss of generality, we restrict our study
to the case where  $a$ and $b$ are coprime. 

For any word $w$, we denote by $|w|_{m}$ (resp.~$|w|_{\overline m}$)
the number of occurrences of the letter $m$ (resp.~$\overline m$) that
it contains. We denote by $|w|$ the length of $w$.   The function
$\phi_{a,b}:\{m, \mb\}^*\to \mathbb{N}$   maps a
word to the final height of the corresponding walk. That is,
$\phi_{a,b}(w) = a|w|_{m}-b|w|_{\overline{m}}$.
The  culmination properties can be translated into  the following
language-theoretic definition:
\begin{definition}
The language of \emm culminating words, is the set
$\C^{a,b}\subset\{m,\overline{m}\}^*$ of words $w$ such that, for 
%all
every non-empty prefix $w'$ of $w$:
$$
\phi_{a,b}(w')>0
\quad \mbox{(Positivity)},
$$
and, for 
%all
every proper prefix $w'$ of $w$:
$$
\phi_{a,b}(w')<\phi_{a,b}(w) \quad \mbox{(Final record)}.
$$
\end{definition}

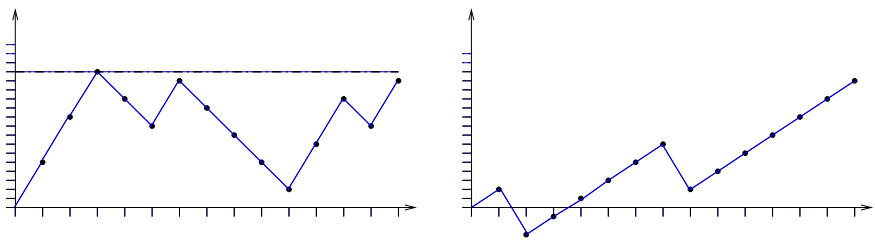
\begin{figure}[tb]
\begin{center}
	\scalebox{0.6}{\input{BadPath.pstex_t}}
\end{center}
\caption{Two walks that are not culminating, violating the final
  record condition 	(left) or the positivity condition (right).} 
\label{fig2}
\end{figure}

The main result of Section~\ref{sec:languages} is that the language
$\C^{a,b}$ is not context-free.
In Section~\ref{sec:exact}, we obtain a closed form expression for
the \gf\ of culminating walks.
This expression is complicated, but we believe this only
reflects the complexity of this class of walks.
This enumerative section is
closely related to the recent work~\cite{mbm-excursions}, devoted to a
general study of excursions confined in a strip. In particular,
\emm symmetric functions, play a slightly surprising role in the proof
and statement of our results.
We then derive in Section~\ref{sec:asympt} the asymptotic number of
culminating walks, in the case $a\ge b$.
Our result implies  that, asymptotically, a positive fraction
of (general) $(a,b)$-walks are culminating if $a>b$. We prove that this
fraction tends to 0 exponentially fast if $a<b$.
More precisely, we determine  the exponential growth of the number of
culminating walks.
This asymptotic section uses the results obtained
in~\cite{BaFl02} on the exact and asymptotic enumeration of excursions
and meanders.
Finally, in Section~\ref{sec:random}, we present several algorithms for
generating uniformly at random culminating walks of a given
length. Our best algorithms are linear when $a\ge b$. When $a<b$, the
choice of the best algorithm is not obvious.
An elementary recursive approach yields a quasi-linear generating stage
but requires the precomputation and storage of $O(n^3)$ numbers.
We exploit in this section several 
%more or less classical 
generation schemes, like the recursive
method~\cite{wilf77,flajoletcalculusINRIA}, the rejection
method~\cite{Den96a} and Boltzmann samplers~\cite{fullboltz}.
Moreover, we address in Section~\ref{sec:random-meander} the random
generation of positive walks,
which is a preliminary step in some of our algorithms
generating culminating walks.
We have implemented our algorithms in Java, and we invite the reader
to generate his/her own paths at the address {\tt
  http://www.lri.fr/$\sim$ponty/walks}. 
Figure~\ref{fig:paths} shows random
culminating paths of length 1000 generated with our software, 
for various values of $a$ and $b$. 
%%%%%%%%%%%%%%%%%%%

 \begin{figure}[tb]
\begin{center}
\scalebox{0.65}{\input{Paths1-1.tex}}  
\scalebox{0.65}{\input{Paths2-1.tex}}  
%\scalebox{0.65}{\input{Paths1-2.tex}}
\scalebox{0.65}{\input{Paths1-2-1.tex}}
%\scalebox{0.65}{\input{Paths1-2-2.tex}}
%\scalebox{0.65}{\input{Paths1-2-3.tex}}
%\scalebox{0.65}{\input{Paths1-2-4.tex}}
\end{center}
\caption{Random culminating paths of size $1000$, when $(a,b)=(1,1)$,
  $(a,b)=(2,1)$,  $(a,b)=(1,2)$. 
In the first two cases, four paths
  are displayed, while for the sake of clarity, only one path is shown
in the third case.} 
\label{fig:paths}
\end{figure}
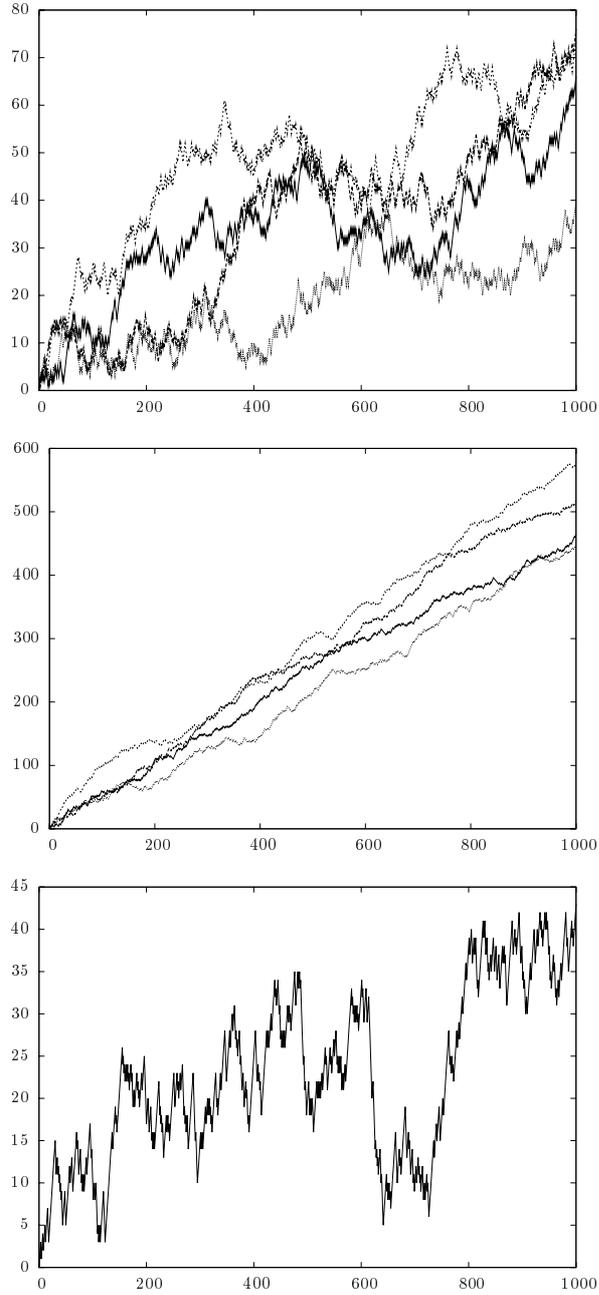

%%%%%%%%%%%%%%%%%%%%%%%%%%%%%%%%%%%%%%%%%%%%%%%%%%%%%%%%%%%%%%%
\section{Language theoretic properties} \label{sec:languages}
%%%%%%%%%%%%%%%%%%%%%%%%%%%%%%%%%%%%%%%%%%%%%%%%%%%%%%%%%%%%%%%%%%%

We denote by $\C^{a,b\Rightarrow k}$  the subset of $\C^{a,b}$ that
%contains
consists of  the walks (words) ending at  height $k$. It will be easily
seen that this language (for a fixed $k$) is regular. However, we
shall prove that the full language  $\C^{a,b}$ is not context-free. We
refer to~\cite{hopcroft} for definitions on languages.

%%%%%%%%%%%%%%%%%%%%%%%%%%%%%%%%%%%%%%%%%%%%%%%%%%%%%%%%
\subsection{Culminating walks of bounded height}
\label{sec:automaton}
\begin{proposition} For all $ a,b,k\in \mathbb{N}$, the language
$\C^{a,b\Rightarrow k}$ of culminating words  ending at height $k$ is regular.
\end{proposition}
\begin{proof}
The culminating paths of final height $k$  move inside
a \emm bounded, space.  This allows us to construct  a (deterministic)
finite-state automaton that recognizes these paths. The  states of
this automaton are the accessible heights (that is, $0,1, \ldots ,
k$), plus a \emph{garbage} state $\perp$. The initial state is $0$,
the final state is $k$, and the transition function $\delta$ is given, for
$0\le q <k$, by:
$$
\begin{array}{rclrcl}
\delta(q,m)&= & \left\{ \begin{array}{cl}
q+a & \mbox{if }q\le k-a,\\
\perp & \mbox{otherwise}
\end{array} \right.,&
 \delta(q,\mb) & =& \left\{ \begin{array}{cl}
q-b   &  \mbox{if }q>b,\\
\perp & \mbox{otherwise}
\end{array} \right.,\\
\end{array}$$
while
$$
\delta(k,\cdot) = \delta(\perp,\cdot)= \perp .
$$

Clearly, this automaton sends any word attempting to walk below  $0$
 (resp. above $k$) in the garbage $\perp$, where it will stay forever
 and therefore be rejected. Moreover, it  only accepts those words ending in
 the state $k$. Hence this automaton  recognizes exactly
 $\C^{a,b\Rightarrow k}$. Since the state space is finite,
 $\C^{a,b\Rightarrow k}$ is a regular language.
\end{proof}

%%%%%%%%%%%%%%%%%%%%%%%%%%%%%%%%%%%%%%%%%%%%%%%%%%%%%%%%%
\subsection{Unbounded culminating walks}
%%%%%%%%%%%%%%%%%%%%%%%%%%%%%%%%%%%%%%%%%%%%%%%%%

\begin{proposition} For all $a,b\in \mathbb{N}$, the language
$ \C^{a,b}$ of culminating walks is not context-free.
\end{proposition}
\begin{proof}
Recall that the intersection of a context-free language and a regular
language is context-free~\cite{hopcroft}. Let  $\La $ be the following
regular language:
$\La =m^*.{\overline{m}}^*.m^*$.
It can be seen as the language of ``zig-zag'' paths.
 Let $\K=\C^{a,b}\cap\ \La$. It is easy to   see that
$$
\K = \{m^i.{\overline{m}}^j.m^k|\, i>0, \, bj<ai \hbox{ and }
bj<ak\}.
$$
Assume that  $\C^{a,b}$  is context-free. Then so is $\K $, and, by the
pumping lemma for context-free languages~\cite[Theorem 4.7]{hopcroft},
there exists $n \in \ns$ such that any  word $w \in \K $ of length at
least $n$
admits a factorisation $w=x.u.y.v.z$ satisfying  the following properties:
\begin{enumerate}
	\item[$(i)$] $|u.v|\ge 1$,
	\item[$(ii)$] $|u.y.v|\le n$,
	\item[$(iii)$] $\forall \ell \ge 0,
w_\ell  :=x. u^\ell.y. v^\ell.z\in \K $.
\end{enumerate}
Since $a$ and $b$ are coprime, there exist $i>n$ and $j>n$ such
that $ia-jb=1$ (this is the Bachet-Bezout theorem).
 Hence the word $w=m^{i}{\overline{m}}^{j}m^{i}$    belongs to $ \K $.
In the rest of the proof, we will refer to the first sequence of
ascending steps of $w$ as $A$, to the descending sequence as $B$ and
to the second ascending sequence    as $C$.
\begin{table}[bth]
\begin{center}
\begin{tabular}{|c|c|c|l|}\hline
Where is the factor $u.y.v$?   & $\ell$   & $w_\ell$  & Failing
condition \\
\hline
%%%%%%%%%%%%%%%%%%%%%%%%%%%
%      CASE A
%%%%%%%%%%%%%%%%%%%%%%%%%%%%
$A$             & $0$ & $m^{i-h}.{\overline{m}}^j.m^i$   &
		Pos.: $\phi(m^{i-h}.{\overline{m}}^j)=1-ah\leq0$\\
\hline
%%%%%%%%%%%%%%%%%%%%%%%%%%%
%      CASE B
%%%%%%%%%%%%%%%%%%%%%%%%%%%%
$B$             & $2$ & $m^{i}.{\overline{m}}^{j+h}.m^i$ &
		Pos.: $\phi(m^{i}.{\overline{m}}^{j+h})=1-bh\leq0$\\
\hline
%%%%%%%%%%%%%%%%%%%%%%%%%%%
%      CASE C
%%%%%%%%%%%%%%%%%%%%%%%%%%%%

$C$             & $0$ & $m^{i}.{\overline{m}}^{j}.m^{i-h}$ &
	 Fin. rec.: $\phi(w_\ell) =\phi(m^i)-ah \leq \phi(m^i)$\\
\hline
%%%%%%%%%%%%%%%%%%%%%%%%%%%
%      CASE A union B
%%%%%%%%%%%%%%%%%%%%%%%%%%%%

$A\cup B$ & & & \\
\ $|u|_{\overline{m}}.|u|_m+|v|_{\overline{m}}.|v|_m\neq 0$ & $2$ &
	 $m^{p}.{\overline{m}}^{k}.m^{k'}.{\overline{m}}^{p'}.m^{i}$ &
$w_\ell \notin \La$ (Too many {peaks})\\

&&&\\
				\ $u=m^k, \ v={\overline{m}}^{k'}$    & $2$ &
				$m^{i+k}.{\overline{m}}^{k'+j}.m^{i}$ &
 Final record: \\
 &&&     $\phi(w_\ell)=\phi(m^{i+k})+1 -bk'  \leq\phi(m^{i+k})$ \\
 \hline

%%%%%%%%%%%%%%%%%%%%%%%%%%%
%      CASE B union C
%%%%%%%%%%%%%%%%%%%%%%%%%%%%

$B\cup C$ & & & \\
\ $|u|_{\overline{m}}.|u|_m+|v|_{\overline{m}}.|v|_m\neq 0$ & $2$ &
	 $m^{i}.{\overline{m}}^{p}.m^{k}.{\overline{m}}^{k'}.m^{p'}$ &
$w_\ell \notin \La$ (Too many {valleys})\\

&&&\\
\ $u={\overline{m}}^k, \ v=m^{k'}$    & $2$ & $m^{i}.{\overline{m}}^{j+k}.m^{k'+i}$ & Pos.: $\phi(m^{i}.{\overline{m}}^{j+k})=1-kb\leq 0$ \\ \hline

\end{tabular}

\medskip
\caption{Why the pumping lemma is not satisfied.}
\label{fig4}      \end{center}
\end{table}

In Table~\ref{fig4}, we consider all  eligible  factorisations of $w$
of the form $w=x.u.y.v.z$. Five cases arise, depending on which part of $w$
contains the factor $u.y.v$. Condition $(ii)$ implies  that
this  factor cannot overlap simultaneously with the parts $A$ and
$C$. Each of   the cases $A\cup B$ and $B\cup C$ is further subdivided
into two cases, depending on whether $u$ and $v$ are \emm monotone, or not.

For each factorisation, the table gives a value of $\ell$ for which the
word  $w_\ell$  does not belong to $\K $. This is justified in the
rightmost column:
either $w_\ell$ does not belong to the set $\La$ of zig-zag paths,
or the positivity  condition does not hold, or the last step of
the walk is not a record.

Once all the possible factorisations have been  investigated and found
not to satisfy  the pumping lemma, we conclude that the languages $\K
$ and $\C^{a,b}$ are not context-free.
\end{proof}

%*******************************************************
\section{Exact enumerative results} \label{sec:exact}
%*******************************************************

In this section, we give a closed form expression for the  \gf\ of
$(a,b)$-culminating walks. More precisely, we give an expression for the
series counting culminating walks of height $k$, and then sum over
$k$. This summation makes the series a bit difficult to handle, for
instance to extract the asymptotic behaviour of the coefficients
(Section~\ref{sec:asympt}). We believe that this 
complexity
 is inherent to the problem. In particular, we prove that the \gf\ of
$(1,1)$-culminating walks is not only transcendental, but also
not \emm D-finite,.
 That is, it does not satisfy any linear differential equation with polynomial
coefficients~\cite[Ch.~6]{stanley-vol2}.

\subsection{Statement of the results and discussion}
\label{sec:statement}
Let us first state our results in the (1,1)-case and 
then
explain what form they  take in the general $(a,b)$-case.

\begin{proposition}
\label{propo-exact11}
Let $a=b=1$ and 
$k\ge 1$.  The length \gf\ of culminating paths of height $k$ is
$$
C_k(t)= \frac{t^k}{F_{k-1}}=t\frac{U_1-U_2}{U_1^k-U_2^k}=
\frac {1-U^2}{1+U^2}\frac{U^k}{1-U^{2k}},
$$
where
\begin{itemize}
	\item $F_k$ is the $k$th Fibonacci polynomial, defined by
$F_0=F_1=1$ and $F_k=F_{k-1}-t^2F_{k-2}$ for $k\ge 2$,
\item $U_1$ and $U_2$ are the two roots of the polynomial $u-t(1+u^2)$:
$$
U_{1,2}= \frac{1\mp\sqrt{1-4t^2}}{2t},
$$
\item  $U$ stands for any of the $U_i$'s.
\end{itemize}
The \gf\ of culminating walks,
\beq\label{C11-expr}
C(t)= \frac {1-U^2}{1+U^2}\sum_{k\ge 1}\frac{U^k}{1-U^{2k}},
\eeq
is not D-finite.
\end{proposition}

\noindent The above expression of $C(t)$  is equivalent to the case $x=y=1$
of~\cite[Eq.(2.26)]{DiFr00}.

\medskip
The first expression of $C_k$, in terms of the Fibonacci polynomials,
is clearly rational.
As explained in Section~\ref{sec:automaton}, the language of
culminating walks of height $k$ is regular for all $a$ and $b$, so
that the series $C_k$ will always
be rational. Of course, $C_k$ is simply $0$ when
$k<a$. When $k=a$, there is only one culminating path, reduced to one
up step, so that $C_k=t$. More generally, the following property,
illustrated in Figure~\ref{fig:fewpaths}
and proved in Section~\ref{sec:property-proof}, holds.
\begin{property}\label{property:uniqueness}
 For $k\le a+b$, there is at most one culminating path of height $k$.
\end{property}

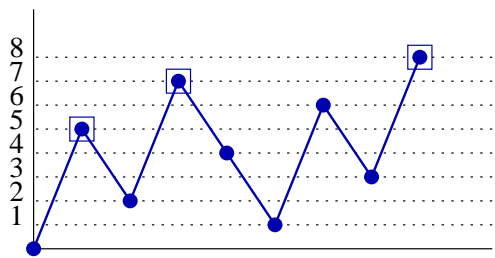
\begin{figure}[tb]
 \begin{center}
\scalebox{1}{\input{fewpaths.pstex_t}}
\end{center}
\caption{When $a=5$ and $b=3$, there is no culminating walk of
height $k$, for $k\in \llbracket 1, 8\rrbracket \setminus\{5,7,8\}$.
For $k=5,7,8$, there is exactly one culminating walk.}
\label{fig:fewpaths}
\end{figure}

As soon as $k>a$, culminating walks of height $k$ have at least two
steps. Deleting the first and last ones gives
$
C_k=t^2W_k,
$
where $W_k$ counts walks (with steps $+a, -b$) going from $a$ to $k-a$
on the segment $\llbracket 1, k-1\rrbracket$. General (and basic)
results on the enumeration of walks on a digraph
provide~\cite[Ch.~4]{stanley-vol1}:
\beq\label{ratform}
C_k=t^2W_k= t^2 \left(
(1-tA_k)^{-1}\right)_{a,k-a}
= t^2 \frac{N_k}{D_k} ,
\eeq
where $A_k=(A_{i,j})_{1\le i, j \le k-1}$ is the adjacency matrix of our
segment graph:
\beq\label{adjacency}
A_{i,j}=  \left\{
\begin{array}{ll}
1 & \hbox{if } j=i+a \hbox{ or } j=i-b, \\
0 & \hbox{otherwise,}
\end{array}
\right.
\eeq
$D_k$ is the determinant of $(1-tA_k)$ and $N_k/D_k$ is the entry $(a,k-a)$
of $(1-tA_k)^{-1}$.

We note from Proposition~\ref{propo-exact11}
 that, in the $(1,1)$-case, both $N_k$ and $D_k$ are especially
simple. Indeed, $N_k=t^{k-2}$, while $D_k=F_{k-1}$ satisfies a linear
recurrence relation (with constant coefficients) of order 2. We will
prove that, for all $a$ and $b$, both sequences $N_k$ and $D_k$
satisfy such a recurrence relation (of a larger order in general).
% --- but of order ${a+b} \choose a$ in general.
The monomial form of $N_k$ will hold as soon as $a=1$.

The second expression of $C_k$
given in Proposition~\ref{propo-exact11}
appears as a rational function of the
roots of the polynomial $u-t(1+u^2)$. Even though both series $U_1$
and $U_2$ are
algebraic (and irrational), the fact that $C_k$ is symmetric in $U_1$ and
$U_2$ explains why $C_k$ itself is rational.
 In general, we will
write $C_k$ as a \emm symmetric rational function, of the $a+b$ roots
of the polynomial $u^b-t(1+u^{a+b})$, denoted $U_1, \ldots,
U_{a+b}$.

The third expression of $C_k$ follows
from the fact that $U_1U_2=1$. In general, $t=U^b/(1+U^{a+b})$ for
 $U=U_i$, so
that it will always be possible to write $C_k$ as a rational function
of $U$. However, this expression will not  be always as simple as above.
The equivalence of  the three expressions of Proposition~\ref{propo-exact11}
follows easily from the fact that
$$
%U=t(1+U^2) \quad \hbox{and} \quad
F_k=
\frac{1-U^{2k+2}}{(1-U^2)(1+U^2)^k}.
$$
This  can be proved by solving the recurrence relation satisfied by
the $F_k$'s --- or can be checked  by induction on $k$.

\medskip

 Let us now state our generalisation of Proposition~\ref{propo-exact11} to
 $(a,b)$-culminating walks. Our first expression of $C_k$, namely  the
 rational form~\eqref{ratform},
 involves the evaluation of two determinants of size (approximately)
 $k$. Our second expression of  $C_k$ will be a \emm fixed, rational function of $U_1,
 \ldots, U_{a+b}, U_1^k, \ldots, U_{a+b}^k$, symmetric in the $U_i$,
 which  involves two determinants of constant size $a+b$. The
 existence of such smaller determinantal forms for walks confined in a strip
 has already been  recognized in~\cite[Ch.~1]{banderier-these}.
More recently, the case of excursions confined in a strip has been
simplified and worked out in greater detail~\cite{mbm-excursions}.
As in~\cite{mbm-excursions},
our results will be expressed in terms of the \emm Schur functions,
 $s_\la$, which form
 one of the most important bases of symmetric functions in $n$
 variables $x_1, \ldots, x_n$:
 for any integer partition $\la$ with at most $n$ parts, $\la=(\la_1,\ldots,
 \la_n)$ with $\la_1\ge \la_2\ge \cdots\ge  \la_n\ge 0$,
\beq\label{schur-def}
s_\la(\X)= \frac{a_{\delta+\la}}{a_\delta},
\eeq
with $\X=(x_1, \ldots, x_n)$, $\delta=(n-1,n-2, \ldots, 1, 0)$ and
$
a_\mu= \det\left( x_i^{\mu_j}\right)_{1\le i,j\le n}.
$
We refer to~\cite[Ch.~7]{stanley-vol2} for generalities on symmetric functions.

\begin{proposition}
\label{propo-exact}
Let $k>a$.  With the above notation, the length \gf\ of
$(a,b)$-culminating paths of height
$k$ admits the following expressions:
$$
C_k(t)
%= t^2W_k
= t^2 \left(
(1-tA_k)^{-1}\right)_{a,k-a}= t^2 \frac{N_k}{D_k} =
t \frac{s_\mu(\U)} {s_\la(\U)},
$$
 where $A_k$ is given by~\eqref{adjacency}, the $(a+b)$-tuple $\U=(U_1, \ldots,
 U_{a+b})$ is the collection of  roots of the  polynomial
 $u^b-t(1+u^{a+b})$, and the partitions $\la$ and $\mu $ are given by
$\la= (k-1)^a$ and $\mu=((k-1)^{a-1}, a-1) $.

 The determinant $D_k$ of $(1-tA_k)$ and the relevant
cofactor $N_k$ are respectively given by
\beq\label{Dk-Nk-Schur}
D_k=(-1)^{(a-1)(k-1)}t^{k-1}s_\la(\U) \quad \hbox{and} \quad
N_k=(-1)^{(a-1)(k-1)}t^{k-2}s_\mu(\U).
\eeq
Both sequences $N_k$ and $D_k$ satisfy a linear recurrence relation
with  coefficients in $\qs[t]$,
respectively of order ${a+b}\choose a$ and  ${a+b}\choose a-1$.
These orders are optimal.
\end{proposition}
Note that the expression of $C_k$ in terms of Schur functions still
holds for $k=a$.
Examples will be given below. For the moment, let us underline that the
case $a=1$ of this proposition takes a remarkably simple form,
which will be given a combinatorial explanation in
Section~\ref{sec:proof-coro}.
\begin{corollary}\label{coro:a=1}
When $a=1$, the  \gf\ of culminating walks of height $k\ge 1$ reads
$$
C_k(t)= \frac{t^k}{D_k} =  \frac{t} {h_{k-1}(\U)},
$$
where $h_i$ is the complete homogeneous symmetric function of degree $i$,
$D_k=1 $ for $1\le k \le b+1$ and $D_k=D_{k-1}-t^{b+1} D_{k-b-1}$
for $k> b+1$.
\end{corollary}

\noindent
{\bf Examples.} Let us illustrate Proposition~\ref{propo-exact} by
writing down explicitly the %Schur
expression of
$C_k$ for a few values of $a$ and $b$. We use the determinantal
form~\eqref{schur-def} of Schur functions.

\smallskip
\noindent{\bf Case $a=b=1$.} Here $U_1$ and $U_2$ are the two roots of the
polynomial $u-t(1+u^2)$. The partition $\mu$ is empty, so that
$s_\mu=1$, while $\la=(k-1)$. This gives
$$
C_k=t
\frac{\left|
	\begin{array}{cc}
		 U_1 & 1\\
		 U_2 & 1
	\end{array}
\right|}
{\left|
	\begin{array}{cc}
		 U_1^k & 1\\
		 U_2^k & 1
	\end{array}
\right|}
=t\frac{U_1-U_2}{U_1^k-U_2^k},
$$
as in Proposition~\ref{propo-exact11}. The recurrence relations
satisfied by the polynomials $N_k$ and $D_k$ can always be worked out from
their expressions~\eqref{Dk-Nk-Schur}, as will be explained in
Section~\ref{sec:main-proof}. In the case $a=b=1$, one finds
$$
C_k=t^2N_k/D_k \quad \hbox{ with } \quad N_k=t^{k-2} \quad \hbox{
	and } \quad  D_k=D_{k-1}-t^2D_{k-2},
$$
with initial conditions $D_1=D_2=1$.

\smallskip
\noindent{\bf Case $a=1, b=2$.} Here $U_1, U_2, U_3$ are the three roots of the
polynomial $u^2-t(1+u^3)$. Again,  $\mu$ is empty and $\la=(k-1)$
(this holds as soon as $a=1$). One obtains
$$
C_k=t \frac{
\left|
	\begin{array}{ccc}
		U_1^2 & U_1 & 1\\
		U_2^2 & U_2 & 1\\
		U_3^2 & U_3 & 1
	\end{array}
\right|
%(U_1-U_2)(U_1-U_3)(U_2-U_3)
}
{\left|
	\begin{array}{ccc}
		U_1^{k+1} & U_1 & 1\\
		U_2^{k+1} & U_2 & 1\\
		U_3^{k+1} & U_3 & 1
	\end{array}
\right|}.
$$
The rational expression of $C_k$ reads
$$C_k=t^2N_k/D_k \quad \hbox{  with} \quad
N_k=t^{k-2} \quad \hbox{ and} \quad  D_k=D_{k-1}-t^3D_{k-3},$$
with initial conditions $D_1=D_2=D_3=1$.
Note that this expression allows us to compute in a few seconds the
number $c_n $ of culminating walks for $n$ up to $500$.

\smallskip
\noindent{\bf Case $a=2, b=1$.} Here $U_1, U_2, U_3$ are the three roots of the
polynomial $u-t(1+u^3)$. One has $\mu=(k-1,1)$ and $\la=(k-1)^2$,
which gives:
$$
C_k=t \frac{\left|
	\begin{array}{ccc}
		U_1^{k+1} & U_1^2 & 1\\
		U_2^{k+1} & U_2^2 & 1\\
		U_3^{k+1} & U_3^2 & 1
	\end{array}
\right|}
{\left|
	\begin{array}{ccc}
		U_1^{k+1} & U_1^k & 1\\
		U_2^{k+1} & U_2^k  & 1\\
		U_3^{k+1} & U_3^k  & 1
	\end{array}
\right|}
=
t \frac{\left|
	\begin{array}{ccc}
		\bar U_1^{k+1} & \bar U_1^{k-1} & 1\\
		\bar U_2^{k+1} & \bar U_2^{k-1} & 1\\
		\bar U_3^{k+1} & \bar U_3^{k-1} & 1
	\end{array}
\right|}
{\left|
	\begin{array}{ccc}
		\bar U_1^{k+1} & \bar U_1 & 1\\
		\bar U_2^{k+1} & \bar U_2  & 1\\
		\bar U_3^{k+1} & \bar U_3  & 1
	\end{array}
\right|},
$$
where $\bar U_i:=1/U_i$. Note that the series $\bar U_i$ are the roots
of the polynomial $u^2-t(1+u^3)$, which occurs in the (symmetric) case
$a=1,b=2$. It is actually clear from~\eqref{ratform} that the
denominator $D_k$ is unchanged when exchanging $a$ and $b$.

The rational expression of $C_k$ reads
$$C_k=t^2N_k/D_k \quad \hbox{ with } \quad
N_k=tN_{k-2}+t^3N_{k-3} \quad \hbox{  and } \quad  D_k=D_{k-1}-t^3D_{k-3},$$
with initial conditions $N_1=0, N_2=1/t, N_3=t$ and $D_1=D_2=D_3=1$.

%%%%%%%%%%%%%%%%%%%%%%%%%%%%%%%%%%%%%%%%%%%%%%%%%%%%%%%
\subsection{Proofs}
%%%%%%%%%%%%%%%%%%%%%%%%%%%%%%%%%%%%%%%%%%%%%%%%%%%%%%%

\subsubsection{Proof of Property~\ref{property:uniqueness}}
\label{sec:property-proof}
Let us say that a path is \emm positive, if every step ends at a positive
level.  For instance, culminating walks are positive. For $n\ge0$
there exists a unique positive walk of length $n$ and height at most
$a+b$,
denoted $w_n$. Indeed, given $h \in \llbracket 0, a+b\rrbracket$,
exactly one
of the values $h+a, h-b$ lies in the interval $\llbracket 1,
a+b\rrbracket$. For the same reason, $w_i$ is a prefix of $w_j$ for
$i\le j$. Let $k \le a+b$,
and assume that there exist two distinct culminating walks of height $k$. These
walks must be $w_i$ and $w_j$, for some $i$ and $j$, with, say, $i<j$.
But then $w_i$ is a prefix of $w_j$, and ends at height $k$, which
prevents $w_j$ from being culminating.
\qed

\subsubsection{Proof of Proposition~\ref{propo-exact}}
\label{sec:main-proof}
The expression of $C_k$ in terms of the adjacency matrix $A_k$ has
been justified in Section~\ref{sec:statement}. Let us now derive the
Schur function  expression of this series. We will give actually two proofs of this
expression: the first one is based on the \emm kernel
method,~\cite{bousquet-petkovsek-recurrences,hexacephale,banderier-these},
and the second one on the \emm  Jacobi-Trudi identity., The first
proof is completely elementary. The second one  allows us to
relate the polynomials $N_k$ and $D_k$ to the
Schur functions $s_\la$ and $s_\mu$.
This derivation is very close to
what was done in ~\cite{mbm-excursions} for excursions confined in a
strip. Some of the results of~\cite{mbm-excursions} will actually be
used to shorten some arguments.

\medskip
\noindent{\bf First proof via the kernel method.}
Consider a culminating walk of height
$k>a$. Such a walk has length at least 2.
Delete its first and last steps: this gives a walk starting from level
$a$,  ending at level $k-a$, and confined  between levels
$1$ and $k-1$. 
%Translating
Shifting  this walk one step down, we obtain a
non-negative walk starting from level $a-1$ and ending at level
$k-1-a$, of height at most $k-2$. Let $G(t,u)\equiv G(u)$
denote the \gf\ of non-negative walks starting from $a-1$, of height
at most $k-2$. In this series, the variable $t$ keeps track of the
length while the variable $u$ records the final height.
Write $G(u)=\sum_{h=0}^{k-2} u^h G_h$, where $G_h$ counts walks ending
at height
$h$. The above argument implies that the \gf\ of culminating walks of
height $k$ is
\beq\label{CG}
C_k= t^2 G_{k-a-1}.
\eeq
We can construct the walks counted  by $G(u)$ step by step, starting
from height $a-1$, and adding at each time a step $+a$ (unless the
current height is $k-a-1$ or more) or $-b$ (unless the current height
is $b-1$ or less). In terms of \gfs, this gives:
$$
G(u)=u^{a-1} + t (u^a+u^{-b}) G(u)
-tu^{-b} \sum_{h=0}^{b-1} u^h G_h
- tu^a \sum_{h=k-a-1}^{k-2} u^h G_h  ,
$$
that is,
$$
\left( u^b-t (1+u^{a+b})\right) G(u) = u^{a+b-1}
-t \sum_{h=0}^{b-1} u^h G_h
-tu^{a+b} \sum_{h=k-a-1}^{k-2} u^h G_h  .
$$
The \emm kernel, of this equation, that is, the polynomial $u^b-t
(1+u^{a+b})$, has $a+b$ distinct roots, which are Puiseux series in
$t$.  We
denote them $U_1, \ldots, U_{a+b}$. Recall that $G(u)$ is a \emm
polynomial, in $u$ (of degree $k-2$). Replacing $u$ by each of the
$U_i$
gives a system of $a+b$ linear equations relating the unknown series $G_0,
\ldots, G_{b-1}$ and $ G_{k-a-1}, \ldots, G_{k-2}$. For $U=U_i$, with
$1\le i \le a+b$,
$$
 \sum_{h=0}^{b-1} U^h G_h
+ U^{a+b} \sum_{h=k-a-1}^{k-2} U^h G_h = U^{a+b-1}/t.
$$
In matrix form, we have $\M\G=\C/t$,
where $\M$ is the square matrix
of size $a+b$ given by
\beq\label{matrix-M}
\M=\left(
\begin{array}{cccccccccccccccccc}
	U_1^{a+b+k-2}  & U_1^{a+b+k-3} & \cdots & U_1^{b+k-1} &
U_1^{b-1} & U_1^{b-2} & \cdots & 1 \\
 U_2^{a+b+k-2}  &\cdots &&&&& \cdots & 1 \\
\vdots &&&&&&&\vdots \\
 U_{a+b}^{a+b+k-2}  & U_{a+b}^{a+b+k-3} & \cdots & U_{a+b}^{b+k-1} &
U_{a+b}^{b-1} & U_{a+b}^{b-2} & \cdots & 1 \\
\end{array}
\right),
\eeq
$\G$ is the column vector
$(G_{k-2},  \ldots, G_{k-a-1},
G_{b-1}, \ldots, G_0)$,
and $\C$ is the column vector
$(U_1^{a+b-1},\ldots, U_{a+b}^{a+b-1})$.
In view of the definition~\eqref{schur-def} of Schur functions,
$$
\det(\M)= s_\la(\U),
$$
with $\la= (k-1)^a$.
It has been shown in~\cite{mbm-excursions} that the \gf\ of excursions
(walks starting and ending at $0$) confined in the strip of height $k-2$ is
$$
\frac{(-1)^{a+1}}{ t}  \frac{s_{(k-2)^a}(\U)}{s_{(k-1)^a}(\U)},
$$
and that, in particular, $s_\la(\U)\not = 0$. Hence $\M$ is invertible,
and applying Cramer's rule to the above system gives
$$
G_{k-a-1}= \frac 1 t \frac {s_\mu(\U)}{s_\la(\U)},
$$
with $\la$ and $\mu$ defined as in the statement of the
proposition. Combining this with~\eqref{CG} gives the desired Schur function 
form of $C_k$.

\medskip
\noindent{\bf A second proof via symmetric functions.}
Let us now give an alternative proof of the Schur function  expression of
$C_k$. It will be based on the dual Jacobi-Trudi identity, which expresses
Schur functions as a determinant in the elementary symmetric
functions $e_i$~\cite[Cor.~7.16.2]{stanley-vol2}:  for any partition
$\nu$, % with $\nu_1\le n$,
\beq\label{JT}
s_\nu= \det\left( e_{\nu'_j+i-j}\right)_{1\le i,j\le \nu_1},
\eeq
where $\nu'$ is the conjugate partition of $\nu$.

Let us consider the identity~\eqref{ratform}, with $D_k=\det(1-tA_k)$.  It
turns out that this determinant \emm is, of the form~\eqref{JT}.
Indeed, let us define $V_i=-U_i$, for $1\le i \le a+b$. Then the
only elementary symmetric functions of the $V_i$ that do not vanish
are $e_0(\V)=1, e_a(\V)=-1/t$ and $e_{a+b}(\V)=1$ (with $\V=(V_1,
\ldots, V_{a+b})$). Let us apply~\eqref{JT}
to $\nu=\la=(k-1)^a$, with variables $V_1, \ldots, V_{a+b}$. Then
$\nu'=a^{k-1}$ and one obtains
$$
s_\la(\V)= (-t)^{-(k-1)} D_k=(-1)^{a(k-1)}s_\la(\U),
$$
since $s_\la$ is homogeneous of degree $a(k-1)$. This gives the Schur function 
expression of $D_k$.

Now, by the general inversion formula for matrices,
$N_k=(-1)^k \det((1-tA_k)^{k-a,a})$, where $(1-tA_k)^{k-a,a}$ is obtained by
deleting  row $k-a$ and column $a$ from $(1-tA_k)$. Let us apply~\eqref{JT}
to $\nu=\mu=((k-1)^{a-1},a-1)$. Then $\nu'=a^{a-1} (a-1)^{k-a}$.  The
matrix $\left( e_{\nu'_j+i-j}\right)$ has size $k-1$, and its last column contains only
one non-zero entry (equal to $e_0(\V)=1$), in row $k-a$. After
deleting this row and the last column, one obtains:
$$
s_\mu(\V)=(-1)^{a-1} (-t)^{-(k-2)} \det ((1-tA_k)^{k-a,a})=(-1)^{a-1}
t^{-(k-2)}N_k = (-1)^{k(a-1)} s_\mu(\U),
$$
as $s_\mu$ is homogeneous
of degree $k(a-1)$. This gives the desired expression of $N_k$.

\medskip
\noindent{\bf Linear recursions.}
Finally, let us prove that the sequences of polynomials $N_k$ and $D_k$ satisfy
a linear recurrence relation with coefficients in
$\qs[t]$, the ring of polynomials in $t$. Equivalently, we prove that
each of the \gfs\
$$
N(z,t):=\sum_{k\ge a} N_k z^k \quad \hbox{ and }\quad
D(z,t):=\sum_{k\ge  a} D_k z^k
$$
is actually a rational function in $z$ and $t$. The existence of a
linear recursion then easily follows by the general theory of rational
series~\cite[Ch.~4]{stanley-vol1}.

Given  the  expression~\eqref{Dk-Nk-Schur} of $N_k$, what we have to
do is to evaluate
$$
N'(z;u_1, \ldots, u_ {a+b}):=  \sum_{k\ge a}s_{(k-1)^{a-1},a-1} z^k$$
where the symmetric functions involve the $a+b$ indeterminates $u_1,
\ldots, u_ {n}$, with $n=a+b$.  We  use the definition~\eqref{schur-def} of Schur
functions to write $s_{(k-1)^{a-1},a-1}$ as a ratio of determinants of size $n$.
The determinant occurring at the denominator is the Vandermonde $V_n$
in the $u_i$'s, and is independent of
$k$. The determinant at the numerator is obtained from~\eqref{matrix-M} by
replacing the column containing $U_i^{b+k-1}$ by a column of
$U_i^{a+b-1}$ (and then each $U_i$ by the indeterminate $u_i$).  We expand it as a sum
over permutations of length $n$, and obtain:
\begin{eqnarray*}
%  \sum_{k\ge a}s_{(k-1)^{a-1},a-1} z^k
N'(z;u) &= &
\frac 1 {V_n} \sum_{k\ge a} z^k \sum_{\si \in \Sn_{n}} \varepsilon(\si) \
\si \left(
u_1^{n+k-2}\cdots u_{a-1}^{b+k}u_a^{a+b-1}u_{a+1}^{b-1}\cdots
u_{n-1}^{1}u_n^0
\right)
\\
&=& \frac 1 {V_n} \sum_{\si \in \Sn_n} \varepsilon(\si) \ \si\left(
\frac{u_1 ^{n+a-2} \cdots u_{a-1}^{a+b} u_a^{a+b-1} u_{a+1}^{b-1} \cdots u_{n-1}^1 u_n^0}
{1-zu_1\cdots u_{a-1}}
\right),\nonumber
\end{eqnarray*}
where $\si$ acts on functions of $u_1, \ldots, u_n$ by permuting the
variables:
$$
\si F(u_1, \ldots, u_n)= F(u_{\si(1)}, \ldots, u_{\si(n)}).
$$
Equivalently,
$$
N'(z;u)
% \sum_{k\ge 0}s_{(k-1)^{a-1},a-1} z^k
= \frac{P(z;u)} {Q(z;u) }
$$
where
$$
Q(z;u)=\sum_{I\subset \llbracket n\rrbracket,\
	|I|=a-1}\left(1-z\prod_{i\in I}u_i\right)
$$
and $P(z;u)$ is another polynomial in 
$z$
and the $u_i$, symmetric in
the $u_i$'s. This symmetry property shows that replacing $u_i$ by
$U_i$ transforms $N'(z;u)$
 into a   rational series in $z$ and $t$. 
The link between $N_k$ and $s_{(k-1)^{a-1},a-1}$ then gives
$$
N(z,t)
=\frac{(-1)^{a-1} P((-1)^{a-1}tz; \, \U)}{t^2\,  Q((-1)^{a-1}tz; \, \U)},$$
another rational function of $z$ and $t$.
A similar argument, given explicitly in~\cite{mbm-excursions}, yields
$$
D(z,t)= \frac{(-1)^{a-1}\tilde P((-1)^{a-1}tz; \, \U)}{t\, \tilde Q((-1)^{a-1}tz; \, \U)},$$
for two polynomials $\tilde P$ and $\tilde Q$ in $z$ and $u_1, \ldots,
u_n$. More precisely,
$$
\tilde Q(z;u)=\sum_{I\subset \llbracket n\rrbracket,\
	|I|=a}\left(1-z\prod_{i\in I}u_i\right).
$$
By looking at the degree of $\tilde Q$ and $Q$, this establishes the existence of recurrence relations of order  ${a+b}\choose a-1$ for $N_k$, and ${a+b}\choose a$ for
$D_k$. If there were recursions of a smaller order, the polynomials
		$Q(z;\, \U)$ or $\tilde Q(z;\, \U)$ would factor. It has been
shown in~\cite[Section~6]{mbm-excursions} that  $\tilde Q(z;\, \U)$ is
irreducible, and the    same argument implies that $Q(z;\, \U)$ is
irreducible as well.
\qed

\subsubsection{Two proofs of Corollary~\ref{coro:a=1}}
\label{sec:proof-coro}
Let us specialize Proposition~\ref{propo-exact} to the case $a=1$. We observe
that $\mu$ is the empty partition, so that $s_\mu=1$, while
$\la=(k-1)$, so that $s_\la=h_{k-1}$. The expressions~\eqref{Dk-Nk-Schur} of $N_k$ and
$D_k$ in terms of Schur functions give $N_k=t^{k-2}$ and $D_k=t^{k-1}
h_{k-1}(\U)$. Observe that $e_1(\U)=1/t$ and $e_{b+1}(U)=(-1)^{b+1}$.
The classical relation between elementary and complete
symmetric functions~\cite[Eq.~(7.13)]{stanley-vol2} gives, for $k\ge 1$,
$$
h_k(\U)=\frac 1 t \, h_{k-1}(\U)- h_{k-b-1}(\U),
$$
with initial conditions $h_0=1$ and $h_i=0$ for $i<0$. This gives the
desired recursion for $D_k$.

\medskip
Let us now justify combinatorially the simplicity of $N_k$ and $D_k$.
Recall that, for $k\ge 2$, one has $C_k=t^2W_k$, where $W_k$
counts walks (with steps +1, $-b$) going from 1 to $k-1$ on the
segment graph $\llbracket 1, k-1 \rrbracket$. The adjacency matrix of
this graph is $A_k$. The combinatorial description\footnote{This
description seems to have been around since, at least, the
80's~\cite{foata,viennot-heaps}. See~\cite[Thm.~2.1]{mbm-icm} for a
modern formulation.}
of the inverse of the matrix $(1-tA_k)$ tells us that $D_k$ counts
non-intersecting collections of elementary cycles on the segment $\llbracket
1, k-1 \rrbracket$, while $N_k$ counts configurations formed of a
self-avoiding path $w$ going from 1 to $k-1$ together with a
non-intersecting collection of
elementary cycles that do not meet $w$. In the polynomials $N_k$ and
$D_k$, each cycle of length $\ell$ is given a weight $(-t^\ell)$ while
the path $w$ is simply weighted $t^\ell$ if it has length $\ell$.
This gives directly
$N_k=t^{k-2}$, as the only possible path $w$ is formed of $k-2$ up
steps, and leaves no place to co-existing cycles. Now the only
elementary cycles are formed of $b$ up steps and one down step $-b$. The
recursion satisfied by $D_k$ is then obtained by discussing whether
the point $k-1$ is contained in one such cycle.

Note that this proof can be rephrased in terms of \emm heaps of
cycles, using Viennot's correspondence
between walks on a graph and certain heaps~\cite{viennot-heaps}.
 The expression $N_k/D_k$ then appears as a specialisation of the \emm
 inversion lemma,  (also found in~\cite{viennot-heaps}).
In particular,  $D_k$  is the (alternating) \gf\ of \emm trivial,
heaps of cycles.
\qed

\medskip
\noindent{\bf Remark.} For general values of $a$ and $b$, the
description of $D_k$ and $N_k$ in terms of cycles and paths on the
graph $\llbracket 1, k-1 \rrbracket$ remains perfectly valid. But the
structure of elementary cycles and self-avoiding paths becomes more
complicated. See an example in Figure~\ref{fig:cycles}.

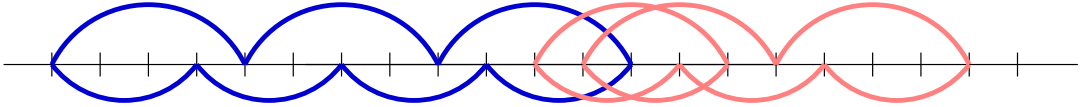
\begin{figure}[thb]
\begin{center}
	\scalebox{1}{\input{cycles.pstex_t}}
\end{center}
\caption{Two non-intersecting elementary cycles  (for $a=4$ and  $b=3$).}
\label{fig:cycles}
\end{figure}

\subsubsection{Proof of Proposition~\ref{propo-exact11}}
The expression of $C_k$ is just a specialisation of
Corollary~\ref{coro:a=1} to the case $b=1$. It remains to prove that the series $C(t)$
is not D-finite.

Let us first observe that $C(t)$ is D-finite if and only if the power
series (in $u$) $B(u):=\sum_k u^{k}/(1-u^{2k})$ is D-finite. Indeed,
one goes from $C(t)$ to $B(u)$, and vice-versa, by an algebraic
substitution of the variable, as $U$ is an algebraic function of $t$
and $t=U/(1+U^2)$. It is known that D-finite series are preserved by
algebraic substitutions~\cite[Thm.~6.4.10]{stanley-vol2}, so that we
can now focus on the series $B(u)$.

This series has integer coefficients, and radius of convergence
1. Hence it is either rational, or admits the unit circle as a natural
boundary~\cite{carlson}.  As will be recalled later~\eqref{S-asympt}, the singular
behaviour of $B(u)$ as $u$ approaches 1 involves a logarithm, which
rules out the possibility of $B(u)$ being rational. Thus $B(u)$ has a
natural boundary, and, in particular, infinitely many
singularities. But D-finite series have only finitely many
singularities, so that $B(u)$ is not D-finite.
\qed

%%%%%%%%%%%%%%%%%%%%%%%%%%%%%%%%%%%%%%%%%%%%%%%%%%%%%%%%%%%%%%
\section{Asymptotic enumerative results} \label{sec:asympt}
%%%%%%%%%%%%%%%%%%%%%%%%%%%%%%%%%%%%%%%%%%%%%%%%%%%%%%%%%%%%%%

In this section we present some results on the asymptotic enumeration
of culminating walks.
Intuitively,  three cases arise, depending on the \emm drift, of the
walks, defined as the difference   $a-b$. Indeed, an $n$-step random
walk of positive drift is known to end at level $O(n)$ and is,
intuitively, quite likely to be culminating.  On the contrary, walks
with a  negative drift have a very small probability of staying
positive. We first  work out the intermediate   case of a zero drift.

%%%%%%%%%%%%%%%%%%%%%%%%%%%%%%%%%%%%%%%%%%%%%%%%%%%%%%
\subsection{Walks with a null drift ($a=b=1$)}
\label{section-exact}
%%%%%%%%%%%%%%%%%%%%%%%%%%%%%%%%%%%%%%%%%%%%%%%%%%%

When the drift is zero, the number of \emm positive, walks (walks in
which every step ends at a positive level)  of length $n$ is known to
be asymptotically equivalent to $2^{n}/\sqrt{2\pi n}$. The average height, and the
average final level of these walks both scale like $\sqrt n$. Hence we
can expect the number of culminating walks to be of the order of
$2^n/n$. This is confirmed by the following result.
\begin{proposition}
 \label{propo-asympt11}
As $n\rightarrow \infty$, the number of $(1,1)$-culminating paths of length
$n$ is asymptotically equivalent to $2^n/(4n)$.
\end{proposition}
\begin{proof}
We start from the expression~\eqref{C11-expr} of $C(t)$, with
$U=U_1=O(t)$, and apply
the \emm singularity analysis, of~\cite{flajolet-odlyzko}.
Note that $U(t)$ is an odd function of $t$. Let us first
study the even part of $C(t)$, which counts culminating paths of even
length:
$$
C_e(t)=  \frac {1-U^2}{1+U^2}\sum_{k\ge 1} \frac{U^{2k}}{1-U^{4k}}.
$$
Let $Z\equiv Z(x)$ be such that $U(t)^2=Z(t^2)$. That is,
$$
 Z\equiv Z(x)= \frac{1-2x-\sqrt{1-4x}}{2x}.
$$
The equation $U=t(1+U^2)$ gives $Z=x(1+Z)^2$. Moreover, we have
$C_e(t) = D(t^2)$ where
$$
D(x)=  \frac {1-Z}{1+Z}\sum_{k\ge 1} \frac{Z^{k}}{1-Z^{2k}}.
$$
We  thus need  to study the asymptotic behaviour of the
coefficients of $D(x)$. We write
$$
D(x)=S(Z(x)), \quad \hbox{with}\quad
S(z)= \frac{1-z}{1+z} \sum_{k\ge 1} \frac {z^k}{1-z^{2k}}.
$$

The series  $ Z(x)$ has radius of convergence $1/4$. It is
analytic in the domain $\D=\cs\setminus[1/4, +\infty)$, with exactly
one singularity, at $x= 1/4$. One has $Z(0)=0$, and $|Z(x)|<1$ for all
$x$ in $\D$. Indeed, assume $|Z(x)|\ge 1$ for some $x$ in $\D$. By
continuity,   $Z(x)=e^{i\theta}$
for some $x$ in $\D$. From the equation $x(1+Z)^2=Z$, we conclude that
$\theta\in (-\pi, \pi)$ 
(for $\theta=\pm \pi$, we would have $Z=-1=0$), 
and that $x=1/(4\cos^2(\theta/2))$. But
this contradicts the fact that $x\in\D$.

 The series $S(z)$ has radius of convergence $1$.
Given that $|Z(x)|<1$ in $\D$, this implies that $D(x)=S(Z(x)))$ is analytic in
the domain $\D$.
It remains to understand how $D(x)$ behaves as $x$ approaches
$1/4$ in $\D$.

Take $x=(1-re^{i\theta})/4$, with $0<r<1$ and $|\theta|<\pi$. Then
$$
Z(x) =  1 - 2 \sqrt{1-4x} + O( 1-4x) = 1- 2\sqrt r e^{i\theta/2} + O(r).
$$
In particular,
$$
\arg(1-Z(x))= \theta/2 +O(\sqrt r).
$$
Choose $\alpha  \in (\pi/4, \pi/2)$.  The above identity shows that there
exists $\eta>0 $ and $\pi/2 <\phi <\pi$ such that, in the indented disk
$$
\I=\{x: \, |1-4x|<\eta \hbox{ and } |\arg(1-4x)|<\phi\},
$$
one has
\beq\label{argZ}
|\arg(1-Z(x))|<\alpha.
\eeq
Now when $z\rightarrow 1$ in such a way that $|\arg(1-z)|<\alpha $,
\beq\label{S-asympt}
\sum_{k\ge 1} \frac {z^k}{1-z^{2k}} \sim \frac 1 {2(1-z)} \log \frac 1
		{1-z}, \quad \hbox{so that} \quad S(z)  \sim \frac 1 4 \log \frac
		1 {1-z}.
\eeq
This can be obtained using a Mellin transform
or some already known results on the \gf\ of divisor
sums~\cite{flajolet-gourdon-dumas}.

Combining~\eqref{argZ} and~\eqref{S-asympt} shows that, as $x$ tends
to $1/4$ in the indented disk $\I$,
\beq\label{D-asympt}
D(x)=S(Z(x))\sim\frac 1 8 \log \frac 1 {1-4x}.
\eeq
This allows us to apply the transfer theorems
of~\cite{flajolet-odlyzko}. Indeed, the series $D(x)$ is analytic in
the following 
%``camembert'' 
domain:
$$
\Delta = \{x \not= 1/4: |4x|< 1+ \eta \ \hbox{ and } \ |\arg(1-4x)|<\phi\},
$$
with  singular behaviour near $x=1/4$ given by~\eqref{D-asympt}. From this we conclude that the
coefficient of $x^n$ in $D(x)$ is asymptotically equivalent to $4^n/(8n)$. Going back to the series $C_e(t)$, this
means that the number of culminating paths of (even) length $N=2n$ is asymptotically equivalent to $2^N/(4N)$.

The study of the odd part of $C(t)$ is similar.
 \end{proof}
\noindent

%%%%%%%%%%%%%%%%%%%%%%%%%%%%%%%%%%%%%%%%%%%%%%%%%%
\subsection{Walks with positive drift ($a>b$)}
\label{sec:positive}
%%%%%%%%%%%%%%%%%%%%%%%%%%%%%%%%%%%%%%%%%%%%%%%%%%%%%%%%%

When the drift is positive, it is known that, asymptotically, a
positive fraction of walks with steps $+a$, $-b$ is actually
\emm positive, (every step ends at a positive level).
More precisely,
as $n \rightarrow \infty$,  the number $p_n^{a,b}$ of positive walks of
length $n$ satisfies
\beq\label{positive-asympt}
 p_n^{a,b}\sim \kappa_{a,b} . 2^n
\eeq
for some  positive constant $\kappa_{a,b}$.
We will show that the culmination and final record conditions  play similar
 filtering roles in  the paths of $\{m,\overline{m}\}^*$, and prove
 the following result.
\begin{proposition}\label{conj:agtb}
For $a>b$,
 the number  $c^{a,b}_n $ of culminating walks of length $n$ satisfies
		$$
c^{a,b}_n = \kappa_{a,b}^2. 2^n + O(\rho^n),
$$
where  $\rho <2$ and $\kappa_{a,b}$ is the constant involved in the
asymptotics of positive walks.
\end{proposition}
\begin{proof}
In what follows, we consider two families of paths that are close to the
meanders and excursions defined in the introduction: the (already
defined) \emm positive, walks, and certain \emm quasi-excursions,.
The exact and  asymptotic enumeration of meanders and excursions
has been  completely   worked out in~\cite{BaFl02}, and we will
 rely heavily on this paper. For instance, the
 estimate~\eqref{positive-asympt} follows from  the results
 of~\cite{BaFl02}  by noticing that a meander factors into an
 excursion followed by a positive walk. Let us call \emm
 quasi-excursion, a walk in which every step, except the final one,
 ends at a positive level. For instance, if $a=3$ and $b=2$, the word
 $m\overline m \overline m$ is a quasi-excursion. By removing the last
 step of such a walk, we see that quasi-excursions are in bijection
 with positive walks  of final height 1, 2, \ldots, or $b$. We  denote
 the number of quasi-excursions of length $n$ by $e_n^{a,b}$. Using
 the results of~\cite{BaFl02}, it is easy to see that, when the drift
 is positive, quasi-excursions are exponentially rare among general
 walks. That is, there exists $\mu<2$ such that for $n$ large enough,
\beq\label{quasi-exc}
e_n^{a,b} < \mu^n.
\eeq

From now on,  we drop the superscripts $a$ and $b$, writing for
instance $c_n$ rather than $c_n^{a,b}$. For any word $w=w_1\cdots
w_k$, denote by $\mir w$ the mirror image of $w$, that is, $\mir w=w_k
\cdots w_1$. Let $u$ be a culminating word of length $n$, and write
$u=vw$, where the word $v$ (resp.~$w$) has length $\lfloor n/2\rfloor$
(resp.~$\lceil n/2\rceil$). Then both $v$ and $\mir w$ are positive
walks, and this proves that
\beq\label{ineq0}
c_n \le p_{\lfloor n/2\rfloor} p_{\lceil n/2\rceil}.
\eeq
Conversely, let us  bound the number of pairs $(v,w)$, where $v$
and $w$ are positive walks of respective lengths  $\lfloor
n/2\rfloor$ and $\lceil n/2\rceil$, such that the word $u=v\mir w$ is
\emm not, culminating. This means that
\begin{itemize}
\item either $u$ factors as $v_1w_1$, where $v_1$ is a
	quasi-excursion of length $i> \lfloor n/2\rfloor$,
\item or, symmetrically, $u$ factors as $v_2\mir w_2$ where  $w_2$ is
	a quasi-excursion of length $j>\lceil n/2\rceil$.
\end{itemize}
This implies that
$$
 p_{\lfloor n/2\rfloor} p_{\lceil n/2\rceil}-c_n \le 2\sum _{i=\lfloor
 n/2\rfloor} ^n e_i 2^{n-i}.
$$
In view of~\eqref{quasi-exc}, we have, for $n$ large enough:
$$
 p_{\lfloor n/2\rfloor} p_{\lceil n/2\rceil}-c_n \le 2\sum _{i=\lfloor
 n/2\rfloor} ^n \mu^i 2^{n-i} \le \frac 2{1-\mu/2}\
 2^n (\mu/2)^{\lfloor n/2\rfloor}\le \frac
 4{1-\mu/2}\ (2 \mu)^{\lfloor n/2\rfloor}.
$$
Combining this with~\eqref{ineq0} and the known asymptotics for the
numbers $p_n$ gives the expected result.
\end{proof}

\subsection{Walks with negative drift ($a<b$):  exponential decay}
When the drift is negative, it is known that positive walks are
exponentially rare among  general walks. Indeed, there exist
constants $\kappa_{a,b}>0$ and $\alpha_{a,b}\in (1,2)$, such that
$$
 p_n^{a,b} \sim \kappa_{a,b}\frac{\alpha_{a,b}^n}{n^{3/2}}.
$$
More precisely,
\beq\label{alpha-def}
 \alpha_{a,b} =
\frac{a+b}{\sqrt[a+b]{a^ab^b}}=\frac{1+q}{\sqrt[1+q]{q^q}} \equiv
\alpha(q),
\eeq
where $q=a/b<1$.
We show below that the constant $\alpha_{a,b}$ also governs the number
of culminating walks of size $n$.
\begin{proposition}\label{negative-drift}
 For $a<b$,
 the number  $c^{a,b}_n $ of culminating walks of length $n$ satisfies
\beq\label{upper}
c_{n}^{a,b}
= O\left( \frac{\alpha_{a,b}^n}{n^{3}}\right),
%\quad \hbox{with }    \alpha_{a,b}\in (1,2).
\eeq
where $\alpha_{a,b}$ is given above. Moreover,
$$
\lim_{n\rightarrow \infty  } \left(c_{n}^{a,b}\right) ^{1/n}=
\alpha_{a,b}.
$$
\end{proposition}
\begin{proof}
 The inequality~\eqref{ineq0} still holds, and gives  the
upper bound~\eqref{upper} on the number of culminating paths.

Let us now prove that the growth constant of culminating walks is
still $\alpha_{a,b}$ by constructing a  large class of such walks.
Let $\E_n$ be the set of excursions of length $n$ (from now on, we
drop the superscripts $a$ and  $b$).  Such excursions only exist
when $n$ is a multiple of $a+b$, and
the number  $e_n$ of such walks then satisfies
$$
e_n\sim \kappa \alpha_{a,b}^n n^{-3/2}
$$
for some positive constant $\kappa$.
It is known that random $(a,b)$-excursions of length $n$ converge in law to the
Brownian excursion, after normalising the length by $n$ and the height
by $\kappa' \sqrt n$, for some constant $\kappa' $ depending on $a$ and
$b$~\cite{kaigh}.  This implies that the (normalized) height of a
discrete excursion converges in law to the height of the Brownian
excursion (described by a theta distribution). In particular, the
probability $p_n$ that an excursion of $\E_n$ has height larger than
$\sqrt n$ tends to a limit $p<1$  as $n$ goes to infinity.
Take an excursion of $\E_k$ of height less than $\sqrt n$, with
$$
k=(a+b) \left\lfloor \frac{n-1-\sqrt n}{a+b}\right\rfloor
$$
and append one up step at its left, and $n-k-1$ up steps at its right:
this gives a culminating walk of length $n$, which proves that
$$
c_{n } \ge e_k (1-p_k).
$$
Taking $n$th roots gives the required lower bound on the growth of $c_n$.
\end{proof}
Hence there are exponentially few walks of size $n$ with steps $+a,
 -b$ that are culminating.
It is  likely that $c_n$ behaves like $\alpha_{a,b}^n
n^{-3-\gamma}$, for some $\gamma\ge 0$ that remains to be determined.
 Note that the final height of an $n$-step meander is known to
have a {discrete} limit law as $n\rightarrow \infty$~\cite{BaFl02}.

%%%%%%%%%%%%%%%%%%%%%%%%%%%%%%%%%%%%%%%%%%%%%
\section{Random generation of positive walks}

\label{sec:random-meander}
%%%%%%%%%%%%%%%%%%%%%%%%%%%%%%%%%%%%%%%%%%%%%%%%%%%%%%

 The random generation of positive walks will be a preliminary
 step in some of the algorithms we present in the next section for the
 generation of culminating walks.  The main ideas underlying
 the generation are the same for both classes of walks, but the class
 of positive walks is  simpler.
	We apply three  different approaches to their random generation:
\emm recursive methods, (two versions),
	\emph{anticipated   rejection}, and  \emm Boltzmann  sampling,. The
	choice of the best algorithm depends on the drift, as summarized in
	the top part of   Table~\ref{tab:res}.
We denote by $\mathcal{P}^{a,b}$ the language of positive walks, but the
	superscript $a,b$ will often be dropped.

%%%%%%%%%%%%%%%%%%%%%%%%%%%%%%%%%%%%%%%%%%%%%%%%%%%%%%%%%%%%%%
\subsection{Recursive step-by-step approach} \label{sec:recursive-meanders}
%%%%%%%%%%%%%%%%%%%%%%%%%%%%%%%%%%%%%%%%%%%%%%%%%%%%%%%%%
The first approach we present is elementary: we construct positive walks
step-by-step, choosing at each time an up or down step
\emm with the right probability,.
This is the basis of the  recursive approach introduced in~\cite{wilf77}.
Here are the three  ideas underlying the algorithm:
\begin{itemize}
\item Let $\W$ be a language, and let $\W_p$ denote the language of the
prefixes of words of $\W$. Assume that for all $w \in \W_p$ such that
$|w|\le n$, we know the number $N_w(n)$ of words of $\W$ of length $n$
beginning with $w$ (we call these word  \emm extensions, of
$w$). Then it is possible to draw  uniformly words of length $n$ in
$\mathcal{W}$ 
as follows. One starts  from the empty word, and adds steps incrementally. If
		at some point the prefix that is built is $w$, one
		adds the letter $x$ to $w$ with probability
		$N_{wx}(n)/N_w(n)$.
\item When $\W=\PP^{a,b}$, the number of extensions of length $n$
	of a prefix $w\in \W_p$ depends   only  on two parameters:
		\begin{itemize}
			\item[--] the length difference $i =n-|w|$,
			\item[--] the final height of $w$,
$j =\phi_{a,b}(w)$,
		\end{itemize}
\item Let  $p_{i,j}$ be the number of extensions of length $n$ of such a
prefix $w$. The numbers $p_{i,j}$ obey the following recurrence:
$$
\begin{array}{lllll}
p_{i,j}& = &p_{i-1,j+a} + {\mathds{1}}_{j> b}
\, p_{i-1,j-b} &
\hbox{for }  i\ge 1,\\
p_{0,j}& =& 1.
\end{array}
$$
\end{itemize}
 As the two parameters $i$ and $j$ are bounded by $n$ and $an$
 respectively, the precomputation of the numbers $p_{i,j}$ takes
 ${O}(n^2)$ arithmetic operations
 and requires to store $O(n^2)$ numbers.
 Then, the  generation of a random word of length $n$
can be performed in linear time.
However, one should take into account the cost  due to the
size of the numbers in the precomputation stage. Indeed,
 the numbers $p_{i,j}$ are exponential in $n$,
so that the actual
 time-space complexity for this  stage may  grow to  ${O}(n^3)$.
However, using a floating-point technique adapted from~\cite{DZ99R},
it should be possible to
take advantage of the numerical stability of the algorithm to reduce
the space needed to ${O}(n^{2+\varepsilon})$.

This naive recursive approach is less efficient than the one presented
below, which is based on context-free grammars. But it will be
easily adapted to the generation of culminating walks, which cannot be
generated via a grammar, as was proved in Section~\ref{sec:languages}.

%%%%%%%%%%%%%%%%%%%%%%%%%%%%%%%%%%%%%%%%%%%%%%%%%%%%%%%%%ù
\subsection{Recursive approach via context-free grammars}
\label{sec:grammars}
%%%%%%%%%%%%%%%%%%%%%%%%%%%%%%%%%%%%%%%%%%%%%%%%%%%%%%%%

It is easy to see that the language $\mathcal{P}^{a,b}\equiv
\mathcal{P}$ is recognized
by a non-deterministic push-down automaton. This implies that
$\mathcal{P}$ is {context-free}.
 The same holds for the language $\mathcal{D}^{a,b}\equiv \mathcal{D}$
of excursions.  A non-ambiguous context-free grammar
generating excursions is given explicitly in~\cite{Duchon98}. It
suffices to add one equation to obtain a non-ambiguous
grammar generating positive walks:
\beq\label{grammar}
\begin{array}{llllllll}
	\D&= &
\varepsilon + \sum_{k=1}^a \cL_k \cR_k ,& \hskip 10mm
& \cL_i&= &
{\mathds{1}}_{i=a} m \ \D + \sum_{k=i+1}^a \cL_k \cR_{k-i},\\
\cP&= &\varepsilon +
\sum_{i=1}^a \cL_i \cP ,&
& \cR_j&= &
{\mathds{1}}_{j=b} \overline m\  \D + \sum_{k=j+1}^b \cL_{k-j} \cR_k.
%\cM&=&\cE \cP.
\end{array}
\eeq
In this system, $\varepsilon $ is the empty word, $\D$ (resp. $\cP$)
is the language of
excursions (resp. positive walks) while $\cL_i$, $1\le i \le
a$ and $\cR_j$, $1\le j\le b$, are $a+b$ auxiliary languages defined
in~\cite{Duchon98}. As above, $m$ and $\mb$ are the up and down
letters in our alphabet.

From this grammar, we can apply the recursive
approach  of~\cite{flajoletcalculusINRIA}
for the uniform generation of \emm decomposable, objects,
implemented in the \texttt{combstruct} package of Maple
or in the stand-alone software \texttt{GenRGenS}~\cite{PoTeDe06}.
The generation of positive walks of size $n$  begins with the
precomputation of $O(n)$ large numbers.
These numbers count words of length $r$, for all $r\le n$, in each of
the languages involved in the grammar.
The fastest way to get them is to convert  the algebraic
system~\eqref{grammar} into a system of linear differential equations,
which, in turn, yields a system of  linear recurrence relations (with
polynomial coefficients)  defining the requested numbers. 
This step requires a linear number of arithmetic operations. 
But one has to multiply numbers whose size (number of digits) is $O(n)$, which
may result, in practice, in a quadratic time-complexity for the
precomputation stage. 
 Then, the generation of a random positive walk can be performed in
 time $O(n\log n)$. 

 Note that a careful implementation \cite{DDZ98E} of the floating point
	approach of~\cite{DZ99R} using an arbitrary-precision floating-point
	computation library yields a $O(n^{1+\varepsilon})$ complexity after
a $O(n^{1+\varepsilon})$ precomputation.

%%%%%%%%%%%%%%%%%%%%%%%%%%%%%%%%%%%%%%%%%%%%%%%%%%%%%%%%%%%%%%%
\subsection{Anticipated rejection}\label{sec:rejection-meanders}
 %%%%%%%%%%%%%%%%%%%%%%%%%%%%%%%%%%%%%%%%%%%%%%%%%%%%%%%%%%%%%%ù
The principle of this approach is to start
with  an empty walk, and then add 
successive  up and down steps by flipping an unbiased
coin until the walk  reaches the desired length $n$, or a non-positive
ordinate. In this case, the walk is rejected and the
procedure starts from the beginning.
Of course, no precomputation nor non-linear storage is required.
This principle was applied to meanders, in the case $a=b=1$,
in~\cite{BaPiSp92},  as a first  step towards the uniform random
generation of \emm directed  animals,.
The analysis of this algorithm
 yielded a linear  time-complexity, later generalized in~\cite{BaPiSp94}
to the case of coloured walks, in which up, down, and level
steps come respectively in $p$, $q$ and $r$ different colours.
There, it was  shown that the time-complexity is linear
when $p\ge q$, but exponential  when $p<q$.

Unsurprisingly, we obtain similar results for the general $(a,b)$-case.
\begin{proposition}
The anticipated rejection scheme applied to the uniform random generation of
$(a,b)$-positive walks
has a linear
time-complexity when $a\ge b$ and an exponential complexity in 
$\Theta(
(2/\alpha_{a,b})^n n\sqrt{n})
$
 when $a<b$, with $\alpha_{a,b}=\frac{a+b}{\sqrt[a+b]{a^ab^b}}<2$.
\end{proposition}
\begin{proof}
We first note that the language
$\mathcal{P}$ of positive walks is a  \emph{left-factor language}.
 That is, it is stable by taking prefixes,
and every word of $\mathcal{P}$ is the proper prefix of another  word
of $\mathcal{P}$.
It has been proved in~\cite{Den96a} that the average
 complexity $f_{\mathcal{L}}(n)$ of
the anticipated rejection scheme for a left-factor language $\mathcal{L}$
on a $k$-letter alphabet is
$$
f_{\mathcal{L}}(n)=\frac{[z^n]\frac{z}{1-z}L(z/k)}{[z^n]L(z/k)}
$$
where $L(z)$ is the length \gf\ of the words of $\mathcal{L}$.

We now exploit the results of~\cite{BaFl02}, giving the singular
behaviour of the series $M(z)$ and $E(z)$ that count respectively
meanders and excursions. As a
meander factors uniquely as an excursion followed by a positive walk,
we can derive from~\cite{BaFl02} the singular
behaviour of the series $P(z)=\sum p_n z^n$ that counts positive walks.
This series is always algebraic, so that singularity analysis  applies.
\begin{itemize}
\item[--]
For $a\ge b$, the series $P(z/2)$ has an algebraic singularity at
$z=1$ in $(1-z)^{-\nu}$ (with $\nu=1$ if $a>b$, and $\nu=1/2$ if
$a=b$). Thus $P(z/2)/(1-z)$ has a singular behaviour in $(1-z)^{-\nu-1}$.
A singularity analysis gives $f_{\mathcal{P}}(n)\sim  n/\nu$.

\item[--]For $a<b$, the series $P(z/2)$ has a square-root singularity at
	$2/\alpha_{a,b}>1$, but $P(z/2)/(1-z)$ has a smaller radius of convergence
	$z_c=1$, with a simple pole at this point. This gives
$$f_{\mathcal{P}}(n)\sim \frac{2^n\, P(1/2)}{p_n}
\sim \kappa\left(\frac2{\alpha_{a,b}}\right)^n{n\sqrt{n}}
$$
for some constant $\kappa$.

\end{itemize}

\end{proof}

%%%%%%%%%%%%%%%%%%%%%%%%%%%%%%%%%%%%%%%%%%%%%%%%%%
\subsection{Boltzmann sampling}\label{sec:meanders-boltzmann}
%%%%%%%%%%%%%%%%%%%%%%%%%%%%%%%%%%%%%%%%%%%%%%%%%%%
%

A Boltzmann generator~\cite{fullboltz} generates every object in the
class $\cC$ with a probability proportional to $x^n$, where $n$ is the
size of the object. More precisely, for every object $w$ (a walk, in
our context):
$$
\mathbb{P} (w)= \frac {x^{|w|} }{C(x)}
$$
where $C(x)$ is the  \gf\ of the objects of $\cC$. Of course, this
results in a relaxation of the size constraint, since objects of all sizes
can be generated. But, by tuning carefully the parameter $x$ (which
has to be 
smaller than or equal to
the radius of convergence of $C(x)$), and
 rejecting the too large and too small objects, one can
often achieve an approximate-size random sampling, with a \emm
tolerance, $\varepsilon$, in linear time. This means that after a linear
number of real-arithmetic operations,
and a number of attempts that is constant on average,
the algorithm will produce an object of size
$|w|\in[ (1-\varepsilon)n , (1+\varepsilon) n]$, which is uniform among the
objects of the same size.

In particular, the grammar~\eqref{grammar} shows that the class of
positive walks is \emm specifiable, in the sense
of~\cite{fullboltz}. The analysis of the \gfs\ of meanders and
excursions performed in~\cite{BaFl02} shows that the series $P(z)$
counting positive walks is always analytic in a $\Delta$-domain, with
a dominant singularity in $(1-\mu t)^{-\nu}$, where $\nu=1$ if $a>b$,
$\nu=1/2$ if $a=b$ and $\nu=-1/2$ if $a<b$. In the first two cases,
Theorem 6.3 of~\cite{fullboltz} gives an approximate sampling in
linear time (and an exact sampling in quadratic time). In the third
case, the standard deviation of the objects produced by a standard
Boltzmann sampler is much larger than their mean, which makes
rejection costly. However, we can generate instead \emm pointed,
positive walks, that is, positive walks with a distinguished step, and
forget the pointing: as guaranteed by Theorem 6.5 of~\cite{fullboltz},
this gives again an approximate sampling in linear time.

\bigskip

To conclude, the uniform random generation of $(a,b)$-positive walks
of size $n$ can be performed in linear time  when $a\ge b$ by an
anticipated rejection, and this strategy does not require any
precomputations nor storage.
When $a< b$, our best algorithm for exact sampling remains the
recursive approach based on the grammar~\eqref{grammar}. It runs in
$O(n^{1+\varepsilon})$ after a $O(n^{1+\varepsilon})$
precomputation. However, one can achieve,
in linear time and space,
an approximate-size sampling using a Boltzmann generator.

%%%%%%%%%%%%%%%%%%%%%%%%%%%%%%%%%%%%%%%%%%%%%%%%%%%%%%%%%
\section{Random generation of culminating walks}\label{sec:random}
%%%%%%%%%%%%%%%%%%%%%%%%%%%%%%%%%%%%%%%%%%%%%%%%%%%%%%%%%
\subsection{Recursive step-by-step approach}
\label{sec:recursive}
%%%%%%%%%%%%%%%%%%%%%%%%%%%%%%%%%%%%%%%%
This elementary procedure, introduced
in~\cite{KuNoPoBIBE04},  generates culminating
walks step by step, choosing every new step with the right
probability. This is again an instance of Wilf's  recursive
method. The arguments given in
Section~\ref{sec:recursive-meanders} for positive walks should now be
replaced by the following ones:
\begin{itemize}
\item For $\W=\C^{a,b}$, the number of extensions of length $n$
of a prefix $w\in \W_p$ depends   only  on three parameters:
	\begin{itemize}
		\item[--] the length difference $i =n-|w|$,
		\item[--] the final height $j =\phi_{a,b}(w)$,
		\item[--] the maximal height $h$ reached by $w$.
	\end{itemize}
\item Let  $c_{i,j,h}$ be the number of extensions of length $n$ of such a
prefix $w$.
The numbers $c_{i,j,h}$ obey the following recurrence:
$$
\begin{array}{lllll}
 c_{i,j,h}& = &c_{i-1,j+a,\max(h,j+a)} +
 {\mathds{1}}_{j>b}\, c_{i-1,j-b,h} &
\hbox{for }  i>1,\\
 c_{1,j,h}& =& {\mathds{1}}_{j+a>h}.
\end{array}
$$
\end{itemize}
As the parameters $i$, $j$ and $h$ are bounded by
$n$, $an$ and $an$ respectively, the precomputation of the numbers
$c(i,j,h)$ takes ${O}(n^3)$ arithmetic operations and requires to
store $O(n^3)$ numbers.   Then, the  generation of a random word of
length $n$ can be performed in linear time.
But again, the numbers $c_{i,j,h}$ are exponential in $n$, so that the
actual time-space complexity of the precomputation stage may grow to
${O}(n^4)$.

The above procedure is easily adapted to generate culminating walks
ending at a prescribed height $k$. The number $ c_{i,j}^{(k)} $
of $i$-step extensions of a prefix ending at height $j$ is given by
$$\begin{array}{lllll}
	c_{i,j}^{(k)}& = & {\mathds{1}}_{j+a<k}\, c_{i-1,j+a}^{(k)} +
 {\mathds{1}}_{j>b}\, c_{i-1,j-b}^{(k)} &
\hbox{for }  i>1,\\
 c_{1,j}^{(k)}& =& {\mathds{1}}_{j+a=k}.
\end{array}
$$
Now $j$ is bounded by $k$, so that we only have to compute a table of
$O(kn)$ numbers, in $O(kn)$ arithmetic operations.  The actual
time-space complexity is likely to grow to $O(kn^2)$ due to the
handling of large numbers.

However, whether the height of the walk is fixed or not, one should be
able to limit the computational overhead
 due to the size of these numbers to
${O}(n^{\varepsilon})$, using a floating-point technique adapted
from~\cite{DZ99R}.

%%%%%%%%%%%%%%%%%%%%%%%%%%%%%%%%%%%%%%%%%%%%%%%%%%%%%%%%%%%%%
\subsection{Rejection methods}\label{sec:rejection-culminating}
%%%%%%%%%%%%%%%%%%%%%%%%%%%%%%%%%%%%%%%%%%%%%%%%%%%%%%%%%%%

We presented in Section~\ref{sec:rejection-meanders} an example of the
\emm anticipated, rejection approach. The more general \emm rejection
principle, has  been applied successfully  to various
problems~\cite{devroye:rg,BaPiSp92,fullboltz}.
The principle of a  rejection algorithm for words in $\W$ is to  draw
objects uniformly in a superset $\V \supset\W$ until an object of $\W$
is found. The average-case complexity of a such a technique is then $
\zeta(n){v_n}/{w_n}$,  where $\zeta(n)$ is the  cost for the
generation
 of a word of size $n$ in $\V$, and $w_n$ and $v_n$ respectively
 denote the number of words of length $n$ in $\W$ and $\V$.

The aim is to find a superset $\V$ satisfying the following (sometimes
conflicting) requirements:

-- the words of $\V$ can be generated quickly, so that $\zeta(n)$ is
	 small,

-- the set $\V$ is not too large, so that the ratio   ${v_n}/{w_n}$
	 is small.

\noindent
Moreover, testing whether a word of $\V$ actually
belongs to $\W$ should be doable in linear time.
 This is obviously the case when $\W=\C^{a,b}$.

We investigate below
two possibilities for the superset $\V$, while fixing $\W=\C^{a,b}$.

%%%%%%%%%%%%%%%%%%%%%%%%%%%%%%%%%%%%%%%%%%
\subsubsection{Drawing from positive walks}
\label{sec:reject-meanders}
%%%%%%%%%%%%%%%%%%%%%%%%%%%%%%%%%%%%%%%%%%%%%%%%%%%%%%%%
Here, we take for $\V$ the set of positive walks.
Their random generation has been discussed in
Section~\ref{sec:random-meander}, and we refer to the last lines of
this section for our conclusions on this question.
\begin{itemize}
\item[--]
When $a<b$, the number $v_n$ of positive walks of length $n$ grows like
$\alpha_{a,b}^n n^{-3/2}$ (up to a multiplicative constant).
If $c_n^{a,b}$ grows like $\alpha_{a,b}^n n^{-3-\gamma}$ for $\gamma
\ge 0$ (see Proposition~\ref{negative-drift}), the cost will be
$O(n^{\gamma+5/2+ \varepsilon} )$, with a preprocessing stage of
${O}(n^{1+\varepsilon} )$.
However, approximate-size sampling can be performed in time
$O(n^{\gamma+5/2} )$, with no preprocessing stage. It suffices to
reject among the set of positive walks generated by a Boltzmann algorithm.

\item[--]
If $a=b$, then $v_n$ grows like $2^n n^{-1/2}$, while $c_n\sim 2^n
/n$ (Proposition~\ref{propo-asympt11}). Hence the cost here is $O(n^{3/2})$.

\item[--]
Finally, for $a>b$, the number of culminating walks grows
like $2^n$ (Proposition~\ref{conj:agtb}). This shows that the
algorithm is  linear.

\end{itemize}

\noindent{\bf Remark.} For $a>b$, culminating walks are so numerous
that we can even perform the rejection in the set of \emm general,
$(a,b)$-walks, and still obtain a linear
complexity, as discussed in the introduction. However, it seems
natural to perform an anticipated rejection, rejecting walks as soon
as they stop being positive: but this  amounts to performing
rejection in the set of positive walks, obtained themselves via an
anticipated rejection from general walks.

%%%%%%%%%%%%%%%%%%%%%%%%%%%%%%%%%%%%%%%%%%%%%%%%
\subsubsection{
Drawing from hybrid walks}
\label{section-miroir}
%%%%%%%%%%%%%%%%%%%%%%%%%%%%%%%%%%%%%%%%%%%%%%%%%

We begin with a simple, yet crucial, observation:
\begin{quote}
	\emm Let $\mir w$ denote the mirror image of the word $w$.
Then, if $w\in \C^{a,b}$,  so is $\mir w$.
\end{quote}
Graphically, taking the mirror image  amounts to a central symmetry on walks.
This remark  implies that, on average, the mid-point of a culminating walk
lies at a height which is half the final height.
This suggests
another possible superset of $\C^{a,b}$ from which we
 may draw, namely the language
$\mathcal{H}^{a,b}$ of \emm
hybrid walks,, defined by
$$
\cH\equiv \mathcal{H}^{a,b}:= \displaystyle \bigcup_{n\ge 0}
\mathcal{P}_{\lfloor{n/2}\rfloor}
\overleftarrow{\mathcal{P}_{\lceil{n/2}\rceil}},
$$
where $ \mathcal{P}$ is the language of positive walks, and
$\overleftarrow{\mathcal{P}}$ the language of
mirror images of positive walks.
 As already observed in Section~\ref{sec:asympt}, $\C^{a,b}  \subset
 \mathcal{H}^{a,b}$.

The intuition behind the choice of the superset $\mathcal{H}^{a,b}$
is that a path that violates
 the positivity (resp. final record) condition is likely to do so at
its beginning (resp. ending). Thus, ensuring positivity on the  first
half of the walk, and  the final record condition on  the second
half, should yield a lower  rejection probability than
 ensuring positivity everywhere,
as we did when drawing from positive walks.

How can one generate 
%mixed positive
hybrid walks uniformly at random? As a
%mixed positive
hybrid walk  of length $n$
is the (non-ambiguous) concatenation of a positive walk
of size $\lfloor n/2\rfloor$
and of the mirror image of another positive walk,
of size $\lceil n/2\rceil$,
it is sufficient to draw positive walks  uniformly at random.
The cost of the generation of a 
%mixed positive
hybrid walk of length $n$ will
be twice the cost of the generation of a positive walk of length
(approximately) $n/2$. We refer again to the end of
Section~\ref{sec:random-meander} for our conclusions on this cost. We
do not use below the Boltzmann sampling for positive walks, since
gluing two positive walks of approximate size $n/2$ does  not give the
same probability to all 
%mixed positive
hybrid walks of a given size.

\noindent
Let us now discuss the efficiency of the rejection approach based on
the language $\mathcal{H}$.
\begin{itemize}
\item[--]
When $a<b$,  we have $|\mathcal{H}_n|=\Theta(
	 \alpha_{a,b} ^n/n^3)$, while $m_n =\Theta( \alpha_{a,b}
	^n/n^{3/2})$, so that we gain an order $O(n^{3/2})$  in complexity
	 (comparing with the rejection of positive walks). This leads to a cost
	 $O(n^{\gamma+1+ \varepsilon})$ if $c_n$ scales like
	 $\alpha_{a,b}^n    n^{-3-\gamma}$,
with a $O(n^{1+\varepsilon})$ precomputation.

\item[--]
When $a=b=1$,
	$ |\mathcal{H}_n|=\Theta( 2^n/n)$,
while $m_{n} =\Theta(2^n/\sqrt n)$, so that the gain is of order
	$\sqrt n$.
Consequently,  the complexity of the rejection algorithm based on
$\mathcal{H}$ is \emm linear,. No precomputation nor storage is required.

\item[--]
For $a>b$, we have  $|\mathcal{H}_n|=\Theta(2^n)$, and similarly
$  m_n =\Theta( 2^n)$.
So the complexity gain (compared with the approach that generates
positive walks) can only be $\Theta(1)$. The algorithm is still linear.

\end{itemize}

%%%%%%%%%%%%%%%%%%%%%%%%%%%%%%%%%%%%%%%%%%%%%%%%%%%%%%%%%%%%%%%%%%%
\section{Conclusion and perspectives}
%%%%%%%%%%%%%%%%%%%%%%%%%%%%%%%%%%%%%%%%%%%%%%%%%%%%%%%%%%%%
We have studied  culminating paths, from the point of view of formal
languages, enumerative combinatorics and random generation. Our best
results in terms of random generation are summarized in   Table~\ref{tab:res}.
%Note that complexities marked with stars (*) are expressed in term of
%arithmetic operations,    and should be multiplied by $n$to obtain
%actual time-space complexities,as the  numbers handled grow
%exponentially with $n$. 

\begin{table}[tbh]
$\begin{array}{|c|c|c|c|c|c|}\hline
										\mbox{Paths}& \mbox{Method} &\sharp \mbox{ Attempts}   & \mbox{Precomp.} & \mbox{Cost}
\\
\hline\hline

%%%%%%%%%%%%%%%%%%%%%%%%%%%%%%%%%%%%%%%%%%%%%%%%%%%%%%%%%%%%%%%

\mathcal{P}^{a,b}  &&&&\\
	 & \mbox{Recursive method, Section~\ref{sec:grammars}:}   &&&\\
&\mbox{standard implementation} & & O(n^2) & O(n \log n)\\
& \mbox{or floating-point implementation.}
&         & O(n^{1+\varepsilon})  & {O}(n^{1+\varepsilon})
\\

 a< b   & \mbox{Approximate-size Boltzmann}, &&&\\
& \mbox{Section~\ref{sec:meanders-boltzmann}} & O(1) &\emptyset & O(n)\\

\hline
%%%%%%%%%%%%%%%%%%%%%%%%%%%%%%%%%%%%%%%%%%%%%%%%%%%%%%%%%%%%%%%

%\mathcal{P}^{a,b}  &&&&\\
 a\ge b   & \mbox{Anticipated rejection,}
&
 O(\sqrt n) \ \ (a=b)    & \emptyset  & {O}(n)
\\
	&  \mbox{Section~\ref{sec:rejection-meanders}} & O(1) \ \ (a>b)&&\\
\hline\hline

%%%%%%%%%%%%%%%%%%%%%%%%%%%%%%%%%%%%%%%%%%%%%%%%%%%%%%%%%%%%%%%%%%%%%
\mathcal{C}^{a,b\Rightarrow k} & \mbox{Recursive method,
%(rational version)
}              &         & {O}(kn^{1+\varepsilon}) & {O}(n)\\
& \mbox{Section~\ref{sec:recursive} }&&&\\
 \hline \hline
%%%%%%%%%%%%%%%%%%%%%%%%%%%%%%%%%%%%%%%%%%%%%%%%%%%%%%%%%%%%%%%

\mathcal{C}^{a,b}   &&&&\\

	& \mbox{Recursive method, \cite{KuNoPoBIBE04} and Section~\ref{sec:recursive}}   &    & {O}(n^{3+\varepsilon}) & {O}(n)
\\
 a<b    &  \mbox{or rejection from 
%mixed positive
hybrid walks,}
 &&&\\
& \mbox{Section~\ref{section-miroir}} & {O}(n^\gamma) &
 {O}(n^{1+\varepsilon}) &  {O}(n^{1+\gamma +\varepsilon}) \\

	 \hline
%%%%%%%%%%%%%%%%%%%%%%%%%%%%%%%%%%%%%%%%%%%%%%%%%%%%%%%%%%%%%%%%
% \mathcal{C}^{a,b}
a=b       & \mbox{Rejection from}  &
O(1)
 & \emptyset & O(n) \\
& \mbox{
%mixed positive
hybrid walks, Section~\ref{section-miroir}}&&&\\
\hline
%%%%%%%%%%%%%%%%%%%%%%%%%%%%%%%%%%%%%%%%%%%%%%%%%%%%%%%%%%%%%%%%%%%%%%
% \mathcal{C}^{a,b}
%
a>b & \mbox{Rejection from positive walks or}     & {O}(1) & \emptyset  & {O}(n)\\

&\mbox{
%mixed positive
hybrid walks, Sections
\ref{sec:reject-meanders} and~\ref{section-miroir}}  &&&\\
	\hline
\end{array}  $
\bigskip
\caption{The complexity of random generation of positive 
%walks
 and culminating paths. 
The cost is that of one random drawing, once the
	precomputations have been performed. 
It is assumed that  $c_n  \sim
	\alpha_{a,b}^n n^{-3-\gamma}$ if $a<b$.}
\label{tab:res}
\end{table}

An important question that is left open is to determine
the asymptotic growth of the number of culminating walks when the
drift is negative ($a<b$). One possible approach would be to exploit
the closed form expression of Proposition~\ref{propo-exact}, in the
spirit of Proposition~\ref{propo-asympt11} and~\cite{BaFl02}. The
result might have interesting consequences regarding the random
generation of culminating walks. In
particular, if
$c_n^{a,b} =\Theta((m_{n/2}^{a,b})^2\,n^{-\gamma})=
\Theta(\alpha_{a,b}^nn^{-3-\gamma})$, with $\gamma<2$,
the generation algorithm based on hybrid walks
would be faster than the recursive algorithm, at least for generating
few paths. However, our numerical data   
suggest that the ratio $c_n^{a,b}/(m_{n/2}^{a,b})^2$
decreases at least as fast as $n^{-2}$.
%

% mbm enrichi le paragraphe suivant
It would also be interesting to study how the height is  distributed
on random culminating walks of length $n$. Such a study may 
provide better algorithms for random generation,
especially  in the  $a<b$ case,
where the height is expected to be small.
How does the average height scale with $n$? 
Is there a limiting distribution for some normalized height? This is
related to a more ambitious question: is there a limiting process for
culminating walks, in the same way discrete excursions converge to the
Brownian excursion~\cite{kaigh}, or discrete meanders to the Brownian meander~\cite{iglehart}? In the case $a=b=1$, a candidate for
the limit process could be the meander conditioned (with care) to
reach its maximum at time 1. 
%he excursion stopped when it reaches its maximum, suitably renormalized.
Note that the joint law of the maximum and final position of a meander
is known~\cite{durrett}, and related to the law of the maximum and
minimum of a Brownian bridge, both in the continuous and discrete
cases~\cite{mohanty}. The case where the maximum coincides with the
final position (an event of zero probability in the continuous case) is
 closely related to our culminating walks.

%%%%%%%%%%%%%%%%%%%%%%%%%%%%%%%%%%%%%%%%%%%%%%%%%%%%%ù

Future extensions of the present work may also include
the study of culminating walks with more than two types of steps, in order to
model different kinds of matches and mismatches,
and thus capture the whole scoring scheme of the FLASH algorithm.
For instance, it is usually considered less drastic
to replace a purine base
by another  purine base (A$\leftrightarrow$G)
rather than a pyrimidine one
in DNA. It is thus natural to  penalize differently different mismatches.
This could be modelled by
introducing  down steps of different heights.

Lastly, a natural, biologically relevant perspective would be to
address the \emph{non-uniform} generation of culminating paths. Indeed,
the  matches and mismatches may not be uniform over
a biological sequence, and be subject to 
 local correlations.
This is classically modelled by a Markov chain (further
conditioned to yield  culminating paths).
Our algorithms could in principle be adapted to this more general context, but their
analysis would need to be carefully worked out. In particular, the drift
of random walks would depend on the chain and differ in general
from $a-b$. We 
naturally expect the efficiency of our algorithms to 
depend  of the model, culminating walks with positive drift being
much easier to generate than those with a negative drift.

\bigskip
\noindent{\bf Acknowledgements.} We are grateful to Jean-François
Marckert for pointing out  the reference~\cite{kaigh}, which
shortened significantly the proof of
Proposition~\ref{negative-drift}. We also thank Jean Bertoin, Philippe
Chassaing, Jean-François Le Gall and Svante Janson for discussions
on the possible limiting process of culminating paths.

%%%%%%%%%%%%%%%%%%%%%%%%%%%%%%%%%%%%%%%%%%%%%%%%%%%%%%%%%%%%%%%%%%%%%%%%

\bibliographystyle{plain}
\bibliography{biblio}
\end{document}

%% file: PathToWord.pstex_t
\begin{picture}(0,0)%
\includegraphics{PathToWord.pstex}%
\end{picture}%
\setlength{\unitlength}{4144sp}%
\begingroup\makeatletter\ifx\SetFigFont\undefined%
\gdef\SetFigFont#1#2#3#4#5{%
  \reset@font\fontsize{#1}{#2pt}%
  \fontfamily{#3}\fontseries{#4}\fontshape{#5}%
  \selectfont}%
\fi\endgroup%
\begin{picture}(4512,2515)(451,-2024)
\put(1216,-1966){\makebox(0,0)[lb]{\smash{{\SetFigFont{12}{14.4}{\rmdefault}{\mddefault}{\updefault}{\color[rgb]{0,0,0}$\overline{m}$}%
}}}}
\put(1486,-1966){\makebox(0,0)[lb]{\smash{{\SetFigFont{12}{14.4}{\rmdefault}{\mddefault}{\updefault}{\color[rgb]{0,0,0}$m$}%
}}}}
\put(1756,-1966){\makebox(0,0)[lb]{\smash{{\SetFigFont{12}{14.4}{\rmdefault}{\mddefault}{\updefault}{\color[rgb]{0,0,0}$m$}%
}}}}
\put(946,-1966){\makebox(0,0)[lb]{\smash{{\SetFigFont{12}{14.4}{\rmdefault}{\mddefault}{\updefault}{\color[rgb]{0,0,0}$m$}%
}}}}
\put(2026,-1966){\makebox(0,0)[lb]{\smash{{\SetFigFont{12}{14.4}{\rmdefault}{\mddefault}{\updefault}{\color[rgb]{0,0,0}$\overline{m}$}%
}}}}
\put(2296,-1966){\makebox(0,0)[lb]{\smash{{\SetFigFont{12}{14.4}{\rmdefault}{\mddefault}{\updefault}{\color[rgb]{0,0,0}$m$}%
}}}}
\put(2566,-1966){\makebox(0,0)[lb]{\smash{{\SetFigFont{12}{14.4}{\rmdefault}{\mddefault}{\updefault}{\color[rgb]{0,0,0}$\overline{m}$}%
}}}}
\put(2836,-1966){\makebox(0,0)[lb]{\smash{{\SetFigFont{12}{14.4}{\rmdefault}{\mddefault}{\updefault}{\color[rgb]{0,0,0}$m$}%
}}}}
\put(3106,-1966){\makebox(0,0)[lb]{\smash{{\SetFigFont{12}{14.4}{\rmdefault}{\mddefault}{\updefault}{\color[rgb]{0,0,0}$\overline{m}$}%
}}}}
\put(3376,-1966){\makebox(0,0)[lb]{\smash{{\SetFigFont{12}{14.4}{\rmdefault}{\mddefault}{\updefault}{\color[rgb]{0,0,0}$\overline{m}$}%
}}}}
\put(3646,-1966){\makebox(0,0)[lb]{\smash{{\SetFigFont{12}{14.4}{\rmdefault}{\mddefault}{\updefault}{\color[rgb]{0,0,0}$m$}%
}}}}
\put(3916,-1966){\makebox(0,0)[lb]{\smash{{\SetFigFont{12}{14.4}{\rmdefault}{\mddefault}{\updefault}{\color[rgb]{0,0,0}$\overline{m}$}%
}}}}
\put(4186,-1966){\makebox(0,0)[lb]{\smash{{\SetFigFont{12}{14.4}{\rmdefault}{\mddefault}{\updefault}{\color[rgb]{0,0,0}$m$}%
}}}}
\put(4456,-1966){\makebox(0,0)[lb]{\smash{{\SetFigFont{12}{14.4}{\rmdefault}{\mddefault}{\updefault}{\color[rgb]{0,0,0}$m$}%
}}}}
\put(451,-961){\makebox(0,0)[lb]{\smash{{\SetFigFont{12}{14.4}{\rmdefault}{\mddefault}{\updefault}{\color[rgb]{0,0,0}$a$}%
}}}}
\put(451,-1501){\makebox(0,0)[lb]{\smash{{\SetFigFont{12}{14.4}{\rmdefault}{\mddefault}{\updefault}{\color[rgb]{0,0,0}$b$}%
}}}}
\end{picture}%

%% file: BadPath.pstex_t
\begin{picture}(0,0)%
\includegraphics{BadPath.pstex}%
\end{picture}%
\setlength{\unitlength}{4144sp}%
\begingroup\makeatletter\ifx\SetFigFont\undefined%
\gdef\SetFigFont#1#2#3#4#5{%
  \reset@font\fontsize{#1}{#2pt}%
  \fontfamily{#3}\fontseries{#4}\fontshape{#5}%
  \selectfont}%
\fi\endgroup%
\begin{picture}(8574,2293)(799,-2072)
\end{picture}%

%% file: Paths1-1.tex
% GNUPLOT: LaTeX picture with Postscript
\begingroup%
\makeatletter%
\newcommand{\GNUPLOTspecial}{%
  \@sanitize\catcode`\%=14\relax\special}%
\setlength{\unitlength}{0.0500bp}%
\begin{picture}(7200,5040)(0,0)%
  {\GNUPLOTspecial{"
%!PS-Adobe-2.0 EPSF-2.0
%%Title: C:/Documents and Settings/ponty/My Documents/Tex/culminants/Paths1-1.tex
%%Creator: gnuplot 4.2 patchlevel rc3
%%CreationDate: Thu Mar 01 02:37:54 2007
%%DocumentFonts: 
%%BoundingBox: 0 0 360 252
%%EndComments
%%BeginProlog
/gnudict 256 dict def
gnudict begin
%
% The following 6 true/false flags may be edited by hand if required
% The unit line width may also be changed
%
/Color false def
/Blacktext false def
/Solid false def
/Dashlength 1 def
/Landscape false def
/Level1 false def
/Rounded false def
/TransparentPatterns false def
/gnulinewidth 5.000 def
/userlinewidth gnulinewidth def
/vshift -66 def
/dl1 {
  10.0 Dashlength mul mul
  Rounded { currentlinewidth 0.75 mul sub dup 0 le { pop 0.01 } if } if
} def
/dl2 {
  10.0 Dashlength mul mul
  Rounded { currentlinewidth 0.75 mul add } if
} def
/hpt_ 31.5 def
/vpt_ 31.5 def
/hpt hpt_ def
/vpt vpt_ def
Level1 {} {
/SDict 10 dict def
systemdict /pdfmark known not {
  userdict /pdfmark systemdict /cleartomark get put
} if
SDict begin [
  /Title (C:/Documents and Settings/ponty/My Documents/Tex/culminants/Paths1-1.tex)
  /Subject (gnuplot plot)
  /Creator (gnuplot 4.2 patchlevel rc3)
  /Author (ponty)
%  /Producer (gnuplot)
%  /Keywords ()
  /CreationDate (Thu Mar 01 02:37:54 2007)
  /DOCINFO pdfmark
end
} ifelse
%
% Gnuplot Prolog Version 4.2 (August 2006)
%
/M {moveto} bind def
/L {lineto} bind def
/R {rmoveto} bind def
/V {rlineto} bind def
/N {newpath moveto} bind def
/Z {closepath} bind def
/C {setrgbcolor} bind def
/f {rlineto fill} bind def
/vpt2 vpt 2 mul def
/hpt2 hpt 2 mul def
/Lshow {currentpoint stroke M 0 vshift R 
	Blacktext {gsave 0 setgray show grestore} {show} ifelse} def
/Rshow {currentpoint stroke M dup stringwidth pop neg vshift R
	Blacktext {gsave 0 setgray show grestore} {show} ifelse} def
/Cshow {currentpoint stroke M dup stringwidth pop -2 div vshift R 
	Blacktext {gsave 0 setgray show grestore} {show} ifelse} def
/UP {dup vpt_ mul /vpt exch def hpt_ mul /hpt exch def
  /hpt2 hpt 2 mul def /vpt2 vpt 2 mul def} def
/DL {Color {setrgbcolor Solid {pop []} if 0 setdash}
 {pop pop pop 0 setgray Solid {pop []} if 0 setdash} ifelse} def
/BL {stroke userlinewidth 2 mul setlinewidth
	Rounded {1 setlinejoin 1 setlinecap} if} def
/AL {stroke userlinewidth 2 div setlinewidth
	Rounded {1 setlinejoin 1 setlinecap} if} def
/UL {dup gnulinewidth mul /userlinewidth exch def
	dup 1 lt {pop 1} if 10 mul /udl exch def} def
/PL {stroke userlinewidth setlinewidth
	Rounded {1 setlinejoin 1 setlinecap} if} def
% Default Line colors
/LCw {1 1 1} def
/LCb {0 0 0} def
/LCa {0 0 0} def
/LC0 {1 0 0} def
/LC1 {0 1 0} def
/LC2 {0 0 1} def
/LC3 {1 0 1} def
/LC4 {0 1 1} def
/LC5 {1 1 0} def
/LC6 {0 0 0} def
/LC7 {1 0.3 0} def
/LC8 {0.5 0.5 0.5} def
% Default Line Types
/LTw {PL [] 1 setgray} def
/LTb {BL [] LCb DL} def
/LTa {AL [1 udl mul 2 udl mul] 0 setdash LCa setrgbcolor} def
/LT0 {PL [] LC0 DL} def
/LT1 {PL [4 dl1 2 dl2] LC1 DL} def
/LT2 {PL [2 dl1 3 dl2] LC2 DL} def
/LT3 {PL [1 dl1 1.5 dl2] LC3 DL} def
/LT4 {PL [6 dl1 2 dl2 1 dl1 2 dl2] LC4 DL} def
/LT5 {PL [3 dl1 3 dl2 1 dl1 3 dl2] LC5 DL} def
/LT6 {PL [2 dl1 2 dl2 2 dl1 6 dl2] LC6 DL} def
/LT7 {PL [1 dl1 2 dl2 6 dl1 2 dl2 1 dl1 2 dl2] LC7 DL} def
/LT8 {PL [2 dl1 2 dl2 2 dl1 2 dl2 2 dl1 2 dl2 2 dl1 4 dl2] LC8 DL} def
/Pnt {stroke [] 0 setdash gsave 1 setlinecap M 0 0 V stroke grestore} def
/Dia {stroke [] 0 setdash 2 copy vpt add M
  hpt neg vpt neg V hpt vpt neg V
  hpt vpt V hpt neg vpt V closepath stroke
  Pnt} def
/Pls {stroke [] 0 setdash vpt sub M 0 vpt2 V
  currentpoint stroke M
  hpt neg vpt neg R hpt2 0 V stroke
 } def
/Box {stroke [] 0 setdash 2 copy exch hpt sub exch vpt add M
  0 vpt2 neg V hpt2 0 V 0 vpt2 V
  hpt2 neg 0 V closepath stroke
  Pnt} def
/Crs {stroke [] 0 setdash exch hpt sub exch vpt add M
  hpt2 vpt2 neg V currentpoint stroke M
  hpt2 neg 0 R hpt2 vpt2 V stroke} def
/TriU {stroke [] 0 setdash 2 copy vpt 1.12 mul add M
  hpt neg vpt -1.62 mul V
  hpt 2 mul 0 V
  hpt neg vpt 1.62 mul V closepath stroke
  Pnt} def
/Star {2 copy Pls Crs} def
/BoxF {stroke [] 0 setdash exch hpt sub exch vpt add M
  0 vpt2 neg V hpt2 0 V 0 vpt2 V
  hpt2 neg 0 V closepath fill} def
/TriUF {stroke [] 0 setdash vpt 1.12 mul add M
  hpt neg vpt -1.62 mul V
  hpt 2 mul 0 V
  hpt neg vpt 1.62 mul V closepath fill} def
/TriD {stroke [] 0 setdash 2 copy vpt 1.12 mul sub M
  hpt neg vpt 1.62 mul V
  hpt 2 mul 0 V
  hpt neg vpt -1.62 mul V closepath stroke
  Pnt} def
/TriDF {stroke [] 0 setdash vpt 1.12 mul sub M
  hpt neg vpt 1.62 mul V
  hpt 2 mul 0 V
  hpt neg vpt -1.62 mul V closepath fill} def
/DiaF {stroke [] 0 setdash vpt add M
  hpt neg vpt neg V hpt vpt neg V
  hpt vpt V hpt neg vpt V closepath fill} def
/Pent {stroke [] 0 setdash 2 copy gsave
  translate 0 hpt M 4 {72 rotate 0 hpt L} repeat
  closepath stroke grestore Pnt} def
/PentF {stroke [] 0 setdash gsave
  translate 0 hpt M 4 {72 rotate 0 hpt L} repeat
  closepath fill grestore} def
/Circle {stroke [] 0 setdash 2 copy
  hpt 0 360 arc stroke Pnt} def
/CircleF {stroke [] 0 setdash hpt 0 360 arc fill} def
/C0 {BL [] 0 setdash 2 copy moveto vpt 90 450 arc} bind def
/C1 {BL [] 0 setdash 2 copy moveto
	2 copy vpt 0 90 arc closepath fill
	vpt 0 360 arc closepath} bind def
/C2 {BL [] 0 setdash 2 copy moveto
	2 copy vpt 90 180 arc closepath fill
	vpt 0 360 arc closepath} bind def
/C3 {BL [] 0 setdash 2 copy moveto
	2 copy vpt 0 180 arc closepath fill
	vpt 0 360 arc closepath} bind def
/C4 {BL [] 0 setdash 2 copy moveto
	2 copy vpt 180 270 arc closepath fill
	vpt 0 360 arc closepath} bind def
/C5 {BL [] 0 setdash 2 copy moveto
	2 copy vpt 0 90 arc
	2 copy moveto
	2 copy vpt 180 270 arc closepath fill
	vpt 0 360 arc} bind def
/C6 {BL [] 0 setdash 2 copy moveto
	2 copy vpt 90 270 arc closepath fill
	vpt 0 360 arc closepath} bind def
/C7 {BL [] 0 setdash 2 copy moveto
	2 copy vpt 0 270 arc closepath fill
	vpt 0 360 arc closepath} bind def
/C8 {BL [] 0 setdash 2 copy moveto
	2 copy vpt 270 360 arc closepath fill
	vpt 0 360 arc closepath} bind def
/C9 {BL [] 0 setdash 2 copy moveto
	2 copy vpt 270 450 arc closepath fill
	vpt 0 360 arc closepath} bind def
/C10 {BL [] 0 setdash 2 copy 2 copy moveto vpt 270 360 arc closepath fill
	2 copy moveto
	2 copy vpt 90 180 arc closepath fill
	vpt 0 360 arc closepath} bind def
/C11 {BL [] 0 setdash 2 copy moveto
	2 copy vpt 0 180 arc closepath fill
	2 copy moveto
	2 copy vpt 270 360 arc closepath fill
	vpt 0 360 arc closepath} bind def
/C12 {BL [] 0 setdash 2 copy moveto
	2 copy vpt 180 360 arc closepath fill
	vpt 0 360 arc closepath} bind def
/C13 {BL [] 0 setdash 2 copy moveto
	2 copy vpt 0 90 arc closepath fill
	2 copy moveto
	2 copy vpt 180 360 arc closepath fill
	vpt 0 360 arc closepath} bind def
/C14 {BL [] 0 setdash 2 copy moveto
	2 copy vpt 90 360 arc closepath fill
	vpt 0 360 arc} bind def
/C15 {BL [] 0 setdash 2 copy vpt 0 360 arc closepath fill
	vpt 0 360 arc closepath} bind def
/Rec {newpath 4 2 roll moveto 1 index 0 rlineto 0 exch rlineto
	neg 0 rlineto closepath} bind def
/Square {dup Rec} bind def
/Bsquare {vpt sub exch vpt sub exch vpt2 Square} bind def
/S0 {BL [] 0 setdash 2 copy moveto 0 vpt rlineto BL Bsquare} bind def
/S1 {BL [] 0 setdash 2 copy vpt Square fill Bsquare} bind def
/S2 {BL [] 0 setdash 2 copy exch vpt sub exch vpt Square fill Bsquare} bind def
/S3 {BL [] 0 setdash 2 copy exch vpt sub exch vpt2 vpt Rec fill Bsquare} bind def
/S4 {BL [] 0 setdash 2 copy exch vpt sub exch vpt sub vpt Square fill Bsquare} bind def
/S5 {BL [] 0 setdash 2 copy 2 copy vpt Square fill
	exch vpt sub exch vpt sub vpt Square fill Bsquare} bind def
/S6 {BL [] 0 setdash 2 copy exch vpt sub exch vpt sub vpt vpt2 Rec fill Bsquare} bind def
/S7 {BL [] 0 setdash 2 copy exch vpt sub exch vpt sub vpt vpt2 Rec fill
	2 copy vpt Square fill Bsquare} bind def
/S8 {BL [] 0 setdash 2 copy vpt sub vpt Square fill Bsquare} bind def
/S9 {BL [] 0 setdash 2 copy vpt sub vpt vpt2 Rec fill Bsquare} bind def
/S10 {BL [] 0 setdash 2 copy vpt sub vpt Square fill 2 copy exch vpt sub exch vpt Square fill
	Bsquare} bind def
/S11 {BL [] 0 setdash 2 copy vpt sub vpt Square fill 2 copy exch vpt sub exch vpt2 vpt Rec fill
	Bsquare} bind def
/S12 {BL [] 0 setdash 2 copy exch vpt sub exch vpt sub vpt2 vpt Rec fill Bsquare} bind def
/S13 {BL [] 0 setdash 2 copy exch vpt sub exch vpt sub vpt2 vpt Rec fill
	2 copy vpt Square fill Bsquare} bind def
/S14 {BL [] 0 setdash 2 copy exch vpt sub exch vpt sub vpt2 vpt Rec fill
	2 copy exch vpt sub exch vpt Square fill Bsquare} bind def
/S15 {BL [] 0 setdash 2 copy Bsquare fill Bsquare} bind def
/D0 {gsave translate 45 rotate 0 0 S0 stroke grestore} bind def
/D1 {gsave translate 45 rotate 0 0 S1 stroke grestore} bind def
/D2 {gsave translate 45 rotate 0 0 S2 stroke grestore} bind def
/D3 {gsave translate 45 rotate 0 0 S3 stroke grestore} bind def
/D4 {gsave translate 45 rotate 0 0 S4 stroke grestore} bind def
/D5 {gsave translate 45 rotate 0 0 S5 stroke grestore} bind def
/D6 {gsave translate 45 rotate 0 0 S6 stroke grestore} bind def
/D7 {gsave translate 45 rotate 0 0 S7 stroke grestore} bind def
/D8 {gsave translate 45 rotate 0 0 S8 stroke grestore} bind def
/D9 {gsave translate 45 rotate 0 0 S9 stroke grestore} bind def
/D10 {gsave translate 45 rotate 0 0 S10 stroke grestore} bind def
/D11 {gsave translate 45 rotate 0 0 S11 stroke grestore} bind def
/D12 {gsave translate 45 rotate 0 0 S12 stroke grestore} bind def
/D13 {gsave translate 45 rotate 0 0 S13 stroke grestore} bind def
/D14 {gsave translate 45 rotate 0 0 S14 stroke grestore} bind def
/D15 {gsave translate 45 rotate 0 0 S15 stroke grestore} bind def
/DiaE {stroke [] 0 setdash vpt add M
  hpt neg vpt neg V hpt vpt neg V
  hpt vpt V hpt neg vpt V closepath stroke} def
/BoxE {stroke [] 0 setdash exch hpt sub exch vpt add M
  0 vpt2 neg V hpt2 0 V 0 vpt2 V
  hpt2 neg 0 V closepath stroke} def
/TriUE {stroke [] 0 setdash vpt 1.12 mul add M
  hpt neg vpt -1.62 mul V
  hpt 2 mul 0 V
  hpt neg vpt 1.62 mul V closepath stroke} def
/TriDE {stroke [] 0 setdash vpt 1.12 mul sub M
  hpt neg vpt 1.62 mul V
  hpt 2 mul 0 V
  hpt neg vpt -1.62 mul V closepath stroke} def
/PentE {stroke [] 0 setdash gsave
  translate 0 hpt M 4 {72 rotate 0 hpt L} repeat
  closepath stroke grestore} def
/CircE {stroke [] 0 setdash 
  hpt 0 360 arc stroke} def
/Opaque {gsave closepath 1 setgray fill grestore 0 setgray closepath} def
/DiaW {stroke [] 0 setdash vpt add M
  hpt neg vpt neg V hpt vpt neg V
  hpt vpt V hpt neg vpt V Opaque stroke} def
/BoxW {stroke [] 0 setdash exch hpt sub exch vpt add M
  0 vpt2 neg V hpt2 0 V 0 vpt2 V
  hpt2 neg 0 V Opaque stroke} def
/TriUW {stroke [] 0 setdash vpt 1.12 mul add M
  hpt neg vpt -1.62 mul V
  hpt 2 mul 0 V
  hpt neg vpt 1.62 mul V Opaque stroke} def
/TriDW {stroke [] 0 setdash vpt 1.12 mul sub M
  hpt neg vpt 1.62 mul V
  hpt 2 mul 0 V
  hpt neg vpt -1.62 mul V Opaque stroke} def
/PentW {stroke [] 0 setdash gsave
  translate 0 hpt M 4 {72 rotate 0 hpt L} repeat
  Opaque stroke grestore} def
/CircW {stroke [] 0 setdash 
  hpt 0 360 arc Opaque stroke} def
/BoxFill {gsave Rec 1 setgray fill grestore} def
/Density {
  /Fillden exch def
  currentrgbcolor
  /ColB exch def /ColG exch def /ColR exch def
  /ColR ColR Fillden mul Fillden sub 1 add def
  /ColG ColG Fillden mul Fillden sub 1 add def
  /ColB ColB Fillden mul Fillden sub 1 add def
  ColR ColG ColB setrgbcolor} def
/BoxColFill {gsave Rec PolyFill} def
/PolyFill {gsave Density fill grestore grestore} def
/h {rlineto rlineto rlineto gsave fill grestore} bind def
%
% PostScript Level 1 Pattern Fill routine for rectangles
% Usage: x y w h s a XX PatternFill
%	x,y = lower left corner of box to be filled
%	w,h = width and height of box
%	  a = angle in degrees between lines and x-axis
%	 XX = 0/1 for no/yes cross-hatch
%
/PatternFill {gsave /PFa [ 9 2 roll ] def
  PFa 0 get PFa 2 get 2 div add PFa 1 get PFa 3 get 2 div add translate
  PFa 2 get -2 div PFa 3 get -2 div PFa 2 get PFa 3 get Rec
  gsave 1 setgray fill grestore clip
  currentlinewidth 0.5 mul setlinewidth
  /PFs PFa 2 get dup mul PFa 3 get dup mul add sqrt def
  0 0 M PFa 5 get rotate PFs -2 div dup translate
  0 1 PFs PFa 4 get div 1 add floor cvi
	{PFa 4 get mul 0 M 0 PFs V} for
  0 PFa 6 get ne {
	0 1 PFs PFa 4 get div 1 add floor cvi
	{PFa 4 get mul 0 2 1 roll M PFs 0 V} for
 } if
  stroke grestore} def
/languagelevel where
 {pop languagelevel} {1} ifelse
 2 lt
	{/InterpretLevel1 true def}
	{/InterpretLevel1 Level1 def}
 ifelse
%
% PostScript level 2 pattern fill definitions
%
/Level2PatternFill {
/Tile8x8 {/PaintType 2 /PatternType 1 /TilingType 1 /BBox [0 0 8 8] /XStep 8 /YStep 8}
	bind def
/KeepColor {currentrgbcolor [/Pattern /DeviceRGB] setcolorspace} bind def
<< Tile8x8
 /PaintProc {0.5 setlinewidth pop 0 0 M 8 8 L 0 8 M 8 0 L stroke} 
>> matrix makepattern
/Pat1 exch def
<< Tile8x8
 /PaintProc {0.5 setlinewidth pop 0 0 M 8 8 L 0 8 M 8 0 L stroke
	0 4 M 4 8 L 8 4 L 4 0 L 0 4 L stroke}
>> matrix makepattern
/Pat2 exch def
<< Tile8x8
 /PaintProc {0.5 setlinewidth pop 0 0 M 0 8 L
	8 8 L 8 0 L 0 0 L fill}
>> matrix makepattern
/Pat3 exch def
<< Tile8x8
 /PaintProc {0.5 setlinewidth pop -4 8 M 8 -4 L
	0 12 M 12 0 L stroke}
>> matrix makepattern
/Pat4 exch def
<< Tile8x8
 /PaintProc {0.5 setlinewidth pop -4 0 M 8 12 L
	0 -4 M 12 8 L stroke}
>> matrix makepattern
/Pat5 exch def
<< Tile8x8
 /PaintProc {0.5 setlinewidth pop -2 8 M 4 -4 L
	0 12 M 8 -4 L 4 12 M 10 0 L stroke}
>> matrix makepattern
/Pat6 exch def
<< Tile8x8
 /PaintProc {0.5 setlinewidth pop -2 0 M 4 12 L
	0 -4 M 8 12 L 4 -4 M 10 8 L stroke}
>> matrix makepattern
/Pat7 exch def
<< Tile8x8
 /PaintProc {0.5 setlinewidth pop 8 -2 M -4 4 L
	12 0 M -4 8 L 12 4 M 0 10 L stroke}
>> matrix makepattern
/Pat8 exch def
<< Tile8x8
 /PaintProc {0.5 setlinewidth pop 0 -2 M 12 4 L
	-4 0 M 12 8 L -4 4 M 8 10 L stroke}
>> matrix makepattern
/Pat9 exch def
/Pattern1 {PatternBgnd KeepColor Pat1 setpattern} bind def
/Pattern2 {PatternBgnd KeepColor Pat2 setpattern} bind def
/Pattern3 {PatternBgnd KeepColor Pat3 setpattern} bind def
/Pattern4 {PatternBgnd KeepColor Landscape {Pat5} {Pat4} ifelse setpattern} bind def
/Pattern5 {PatternBgnd KeepColor Landscape {Pat4} {Pat5} ifelse setpattern} bind def
/Pattern6 {PatternBgnd KeepColor Landscape {Pat9} {Pat6} ifelse setpattern} bind def
/Pattern7 {PatternBgnd KeepColor Landscape {Pat8} {Pat7} ifelse setpattern} bind def
} def
%
%
%End of PostScript Level 2 code
%
/PatternBgnd {
  TransparentPatterns {} {gsave 1 setgray fill grestore} ifelse
} def
%
% Substitute for Level 2 pattern fill codes with
% grayscale if Level 2 support is not selected.
%
/Level1PatternFill {
/Pattern1 {0.250 Density} bind def
/Pattern2 {0.500 Density} bind def
/Pattern3 {0.750 Density} bind def
/Pattern4 {0.125 Density} bind def
/Pattern5 {0.375 Density} bind def
/Pattern6 {0.625 Density} bind def
/Pattern7 {0.875 Density} bind def
} def
%
% Now test for support of Level 2 code
%
Level1 {Level1PatternFill} {Level2PatternFill} ifelse
/Symbol-Oblique /Symbol findfont [1 0 .167 1 0 0] makefont
dup length dict begin {1 index /FID eq {pop pop} {def} ifelse} forall
currentdict end definefont pop
end
gnudict begin
gsave
0 0 translate
0.050 0.050 scale
0 setgray
newpath
1.000 UL
LTb
700 400 M
63 0 V
6097 0 R
-63 0 V
700 950 M
63 0 V
6097 0 R
-63 0 V
700 1500 M
63 0 V
6097 0 R
-63 0 V
700 2050 M
63 0 V
6097 0 R
-63 0 V
700 2600 M
63 0 V
6097 0 R
-63 0 V
700 3150 M
63 0 V
6097 0 R
-63 0 V
700 3700 M
63 0 V
6097 0 R
-63 0 V
700 4250 M
63 0 V
6097 0 R
-63 0 V
700 4800 M
63 0 V
6097 0 R
-63 0 V
700 400 M
0 63 V
0 4337 R
0 -63 V
1932 400 M
0 63 V
0 4337 R
0 -63 V
3164 400 M
0 63 V
0 4337 R
0 -63 V
4396 400 M
0 63 V
0 4337 R
0 -63 V
5628 400 M
0 63 V
0 4337 R
0 -63 V
6860 400 M
0 63 V
0 4337 R
0 -63 V
700 4800 M
700 400 L
6160 0 V
0 4400 V
-6160 0 V
1.000 UP
stroke
LT0
700 400 M
6 55 V
6 55 V
6 55 V
7 -55 V
6 55 V
6 55 V
6 -55 V
6 -55 V
6 55 V
7 -55 V
6 55 V
6 55 V
6 -55 V
6 55 V
6 -55 V
7 -55 V
6 -55 V
6 55 V
6 -55 V
6 55 V
6 55 V
7 55 V
6 -55 V
6 -55 V
6 55 V
6 -55 V
6 55 V
6 -55 V
7 55 V
6 55 V
6 55 V
6 -55 V
6 -55 V
6 55 V
7 55 V
6 -55 V
6 -55 V
6 55 V
6 55 V
6 55 V
7 -55 V
6 -55 V
6 -55 V
6 -55 V
6 55 V
6 -55 V
7 55 V
6 55 V
6 55 V
6 55 V
6 55 V
6 55 V
6 55 V
7 55 V
6 -55 V
6 55 V
6 55 V
6 55 V
6 -55 V
7 55 V
6 55 V
6 55 V
6 55 V
6 55 V
6 -55 V
7 55 V
6 -55 V
6 -55 V
6 -55 V
6 -55 V
6 55 V
7 -55 V
6 -55 V
6 55 V
6 55 V
6 -55 V
6 -55 V
6 55 V
7 55 V
6 55 V
6 55 V
6 -55 V
6 -55 V
6 55 V
7 55 V
6 -55 V
6 -55 V
6 -55 V
6 55 V
6 55 V
7 -55 V
6 55 V
6 -55 V
6 -55 V
6 55 V
6 -55 V
7 -55 V
6 -55 V
6 -55 V
6 55 V
6 55 V
6 -55 V
6 -55 V
7 55 V
1347 1005 L
6 -55 V
6 55 V
6 -55 V
6 55 V
7 55 V
6 -55 V
6 55 V
6 -55 V
6 -55 V
6 -55 V
7 55 V
6 55 V
6 -55 V
6 55 V
6 -55 V
6 -55 V
7 55 V
6 55 V
6 -55 V
6 55 V
6 -55 V
6 -55 V
6 55 V
7 55 V
6 -55 V
6 55 V
6 55 V
6 55 V
6 55 V
7 55 V
6 55 V
6 -55 V
6 55 V
6 55 V
6 55 V
7 55 V
6 -55 V
6 -55 V
6 -55 V
6 -55 V
6 55 V
7 55 V
6 55 V
6 55 V
6 55 V
6 55 V
6 -55 V
6 55 V
7 55 V
6 55 V
6 -55 V
6 55 V
6 -55 V
6 55 V
7 55 V
6 55 V
6 55 V
6 55 V
6 -55 V
6 55 V
7 -55 V
6 55 V
6 -55 V
6 55 V
6 55 V
6 -55 V
7 55 V
6 -55 V
6 -55 V
6 55 V
6 55 V
6 55 V
6 -55 V
7 -55 V
6 55 V
6 55 V
6 -55 V
6 -55 V
6 55 V
7 -55 V
6 55 V
6 -55 V
6 55 V
6 -55 V
6 55 V
7 55 V
6 55 V
6 -55 V
6 -55 V
6 -55 V
6 55 V
7 -55 V
6 55 V
6 -55 V
6 55 V
6 55 V
6 -55 V
6 55 V
7 55 V
6 -55 V
6 55 V
6 -55 V
6 55 V
6 55 V
1994 2050 L
6 55 V
6 55 V
6 -55 V
6 55 V
6 55 V
7 -55 V
6 55 V
6 55 V
6 -55 V
6 -55 V
6 -55 V
7 55 V
6 -55 V
6 -55 V
6 -55 V
6 -55 V
6 -55 V
6 55 V
7 55 V
6 55 V
6 -55 V
6 -55 V
6 -55 V
6 -55 V
7 -55 V
6 55 V
6 55 V
6 55 V
6 -55 V
6 55 V
7 -55 V
6 -55 V
6 -55 V
6 -55 V
6 55 V
6 -55 V
7 55 V
6 55 V
6 55 V
6 -55 V
6 -55 V
6 55 V
6 55 V
7 55 V
6 55 V
6 55 V
6 -55 V
6 -55 V
6 55 V
7 -55 V
6 -55 V
6 55 V
6 55 V
6 55 V
6 55 V
7 55 V
6 -55 V
6 -55 V
6 -55 V
6 -55 V
6 55 V
7 55 V
6 55 V
6 -55 V
6 55 V
6 -55 V
6 55 V
6 -55 V
7 55 V
6 -55 V
6 55 V
6 55 V
6 -55 V
6 55 V
7 55 V
6 55 V
6 -55 V
6 55 V
6 -55 V
6 55 V
7 -55 V
6 -55 V
6 55 V
6 55 V
6 -55 V
6 55 V
7 55 V
6 55 V
6 55 V
6 -55 V
6 -55 V
6 55 V
6 55 V
7 -55 V
6 55 V
6 -55 V
6 55 V
6 55 V
6 55 V
7 55 V
6 -55 V
6 55 V
6 -55 V
6 55 V
2640 2545 L
7 -55 V
6 -55 V
6 55 V
6 -55 V
6 55 V
6 -55 V
7 55 V
6 -55 V
6 -55 V
6 -55 V
6 -55 V
6 -55 V
6 -55 V
7 -55 V
6 -55 V
6 55 V
6 55 V
6 55 V
6 -55 V
7 -55 V
6 -55 V
6 55 V
6 -55 V
6 55 V
6 -55 V
7 55 V
6 -55 V
6 55 V
6 -55 V
6 -55 V
6 -55 V
7 55 V
6 55 V
6 55 V
6 55 V
6 55 V
6 55 V
6 -55 V
7 55 V
6 55 V
6 -55 V
6 -55 V
6 55 V
6 -55 V
7 -55 V
6 -55 V
6 55 V
6 55 V
6 55 V
6 55 V
7 -55 V
6 -55 V
6 -55 V
6 55 V
6 55 V
6 55 V
7 -55 V
6 55 V
6 55 V
6 55 V
6 55 V
6 -55 V
6 55 V
7 -55 V
6 -55 V
6 -55 V
6 -55 V
6 55 V
6 55 V
7 55 V
6 55 V
6 55 V
6 -55 V
6 55 V
6 -55 V
7 -55 V
6 55 V
6 -55 V
6 55 V
6 -55 V
6 -55 V
7 55 V
6 -55 V
6 -55 V
6 55 V
6 55 V
6 55 V
6 -55 V
7 55 V
6 -55 V
6 -55 V
6 55 V
6 -55 V
6 55 V
7 -55 V
6 -55 V
6 -55 V
6 -55 V
6 55 V
6 -55 V
7 55 V
6 -55 V
6 55 V
6 -55 V
3287 2270 L
6 -55 V
7 -55 V
6 55 V
6 55 V
6 55 V
6 -55 V
6 55 V
6 55 V
7 -55 V
6 55 V
6 55 V
6 55 V
6 55 V
6 55 V
7 -55 V
6 55 V
6 55 V
6 55 V
6 55 V
6 55 V
7 -55 V
6 55 V
6 -55 V
6 -55 V
6 55 V
6 55 V
7 55 V
6 -55 V
6 -55 V
6 55 V
6 -55 V
6 -55 V
6 55 V
7 55 V
6 55 V
6 -55 V
6 -55 V
6 55 V
6 -55 V
7 55 V
6 -55 V
6 -55 V
6 55 V
6 55 V
6 -55 V
7 55 V
6 -55 V
6 -55 V
6 55 V
6 -55 V
6 55 V
7 -55 V
6 -55 V
6 55 V
6 55 V
6 55 V
6 -55 V
6 -55 V
7 -55 V
6 -55 V
6 55 V
6 55 V
6 55 V
6 55 V
7 55 V
6 55 V
6 55 V
6 55 V
6 -55 V
6 55 V
7 55 V
6 -55 V
6 55 V
6 -55 V
6 -55 V
6 55 V
7 -55 V
6 55 V
6 -55 V
6 -55 V
6 55 V
6 -55 V
6 -55 V
7 55 V
6 55 V
6 55 V
6 55 V
6 55 V
6 55 V
7 -55 V
6 -55 V
6 -55 V
6 -55 V
6 55 V
6 -55 V
7 -55 V
6 55 V
6 -55 V
6 55 V
6 -55 V
6 55 V
7 55 V
6 -55 V
6 -55 V
3934 2875 L
6 55 V
6 -55 V
6 -55 V
7 -55 V
6 55 V
6 55 V
6 -55 V
6 -55 V
6 -55 V
7 -55 V
6 -55 V
6 55 V
6 55 V
6 -55 V
6 -55 V
7 -55 V
6 -55 V
6 -55 V
6 55 V
6 -55 V
6 -55 V
7 -55 V
6 55 V
6 55 V
6 -55 V
6 -55 V
6 -55 V
6 -55 V
7 -55 V
6 -55 V
6 -55 V
6 55 V
6 -55 V
6 55 V
7 55 V
6 55 V
6 55 V
6 -55 V
6 -55 V
6 55 V
7 -55 V
6 -55 V
6 55 V
6 -55 V
6 55 V
6 -55 V
7 55 V
6 -55 V
6 55 V
6 -55 V
6 55 V
6 55 V
6 -55 V
7 55 V
6 55 V
6 -55 V
6 55 V
6 -55 V
6 55 V
7 -55 V
6 55 V
6 -55 V
6 55 V
6 -55 V
6 55 V
7 -55 V
6 -55 V
6 -55 V
6 55 V
6 55 V
6 55 V
7 -55 V
6 -55 V
6 55 V
6 -55 V
6 55 V
6 55 V
6 55 V
7 55 V
6 -55 V
6 55 V
6 -55 V
6 55 V
6 -55 V
7 55 V
6 55 V
6 55 V
6 -55 V
6 -55 V
6 -55 V
7 55 V
6 55 V
6 55 V
6 -55 V
6 55 V
6 -55 V
7 -55 V
6 55 V
6 -55 V
6 -55 V
6 -55 V
6 -55 V
6 55 V
7 -55 V
4581 2270 L
6 55 V
6 -55 V
6 -55 V
6 -55 V
7 55 V
6 -55 V
6 55 V
6 -55 V
6 -55 V
6 -55 V
7 -55 V
6 -55 V
6 -55 V
6 55 V
6 -55 V
6 55 V
7 55 V
6 -55 V
6 -55 V
6 -55 V
6 -55 V
6 55 V
6 55 V
7 55 V
6 55 V
6 55 V
6 -55 V
6 -55 V
6 55 V
7 55 V
6 -55 V
6 -55 V
6 -55 V
6 55 V
6 55 V
7 55 V
6 -55 V
6 55 V
6 55 V
6 -55 V
6 -55 V
7 -55 V
6 55 V
6 55 V
6 55 V
6 -55 V
6 -55 V
6 55 V
7 55 V
6 -55 V
6 -55 V
6 -55 V
6 55 V
6 55 V
7 55 V
6 -55 V
6 55 V
6 55 V
6 -55 V
6 55 V
7 -55 V
6 -55 V
6 55 V
6 -55 V
6 -55 V
6 -55 V
7 55 V
6 -55 V
6 -55 V
6 -55 V
6 -55 V
6 55 V
6 -55 V
7 55 V
6 -55 V
6 -55 V
6 55 V
6 55 V
6 -55 V
7 55 V
6 -55 V
6 55 V
6 55 V
6 55 V
6 -55 V
7 -55 V
6 -55 V
6 -55 V
6 55 V
6 -55 V
6 55 V
7 55 V
6 55 V
6 -55 V
6 -55 V
6 55 V
6 -55 V
6 -55 V
7 55 V
6 55 V
6 55 V
6 -55 V
6 55 V
6 55 V
5228 1995 L
6 55 V
6 -55 V
6 -55 V
6 55 V
6 -55 V
7 -55 V
6 55 V
6 55 V
6 55 V
6 55 V
6 -55 V
7 55 V
6 55 V
6 -55 V
6 55 V
6 55 V
6 -55 V
6 55 V
7 -55 V
6 -55 V
6 55 V
6 55 V
6 -55 V
6 55 V
7 -55 V
6 -55 V
6 -55 V
6 55 V
6 -55 V
6 -55 V
7 -55 V
6 55 V
6 55 V
6 55 V
6 55 V
6 55 V
7 55 V
6 55 V
6 55 V
6 55 V
6 -55 V
6 -55 V
6 55 V
7 -55 V
6 -55 V
6 55 V
6 55 V
6 55 V
6 -55 V
7 55 V
6 55 V
6 55 V
6 55 V
6 55 V
6 55 V
7 55 V
6 55 V
6 -55 V
6 -55 V
6 55 V
6 55 V
7 55 V
6 -55 V
6 -55 V
6 55 V
6 -55 V
6 55 V
6 -55 V
7 -55 V
6 55 V
6 -55 V
6 -55 V
6 -55 V
6 55 V
7 -55 V
6 -55 V
6 55 V
6 -55 V
6 55 V
6 -55 V
7 55 V
6 55 V
6 55 V
6 55 V
6 -55 V
6 -55 V
7 55 V
6 55 V
6 55 V
6 55 V
6 -55 V
6 -55 V
6 55 V
7 55 V
6 -55 V
6 55 V
6 55 V
6 55 V
6 55 V
7 55 V
6 -55 V
6 55 V
6 55 V
6 -55 V
5874 3040 L
7 55 V
6 55 V
6 55 V
6 55 V
6 -55 V
6 -55 V
7 -55 V
6 -55 V
6 55 V
6 55 V
6 -55 V
6 55 V
6 -55 V
7 55 V
6 55 V
6 55 V
6 55 V
6 55 V
6 -55 V
7 55 V
6 55 V
6 55 V
6 -55 V
6 55 V
6 -55 V
7 55 V
6 -55 V
6 -55 V
6 55 V
6 55 V
6 -55 V
7 -55 V
6 55 V
6 -55 V
6 -55 V
6 55 V
6 -55 V
6 -55 V
7 -55 V
6 -55 V
6 55 V
6 55 V
6 55 V
6 55 V
7 55 V
6 -55 V
6 -55 V
6 -55 V
6 55 V
6 -55 V
7 -55 V
6 55 V
6 -55 V
6 55 V
6 55 V
6 -55 V
7 -55 V
6 -55 V
6 -55 V
6 55 V
6 -55 V
6 -55 V
6 55 V
7 -55 V
6 -55 V
6 -55 V
6 -55 V
6 -55 V
6 55 V
7 -55 V
6 55 V
6 -55 V
6 55 V
6 -55 V
6 -55 V
7 55 V
6 55 V
6 -55 V
6 55 V
6 -55 V
6 55 V
7 -55 V
6 55 V
6 55 V
6 55 V
6 55 V
6 -55 V
6 -55 V
7 55 V
6 55 V
6 -55 V
6 55 V
6 -55 V
6 -55 V
7 -55 V
6 55 V
6 55 V
6 55 V
6 55 V
6 -55 V
7 -55 V
6 55 V
6 -55 V
6 55 V
6521 3095 L
6 55 V
7 -55 V
6 55 V
6 55 V
6 55 V
6 55 V
6 -55 V
6 55 V
7 -55 V
6 55 V
6 -55 V
6 55 V
6 -55 V
6 55 V
7 55 V
6 -55 V
6 -55 V
6 55 V
6 55 V
6 -55 V
7 -55 V
6 -55 V
6 55 V
6 55 V
6 55 V
6 55 V
7 55 V
6 -55 V
6 55 V
6 55 V
6 55 V
6 55 V
6 -55 V
7 -55 V
6 55 V
6 55 V
6 55 V
6 -55 V
6 55 V
7 -55 V
6 55 V
6 55 V
6 55 V
6 -55 V
6 55 V
7 55 V
6 -55 V
6 55 V
6 -55 V
6 55 V
6 -55 V
7 55 V
6 55 V
6 55 V
6 55 V
stroke
LT1
700 400 M
6 55 V
6 55 V
6 55 V
7 -55 V
6 55 V
6 55 V
6 -55 V
6 55 V
6 55 V
7 55 V
6 -55 V
6 -55 V
6 55 V
6 55 V
6 55 V
7 55 V
6 55 V
6 55 V
6 -55 V
6 55 V
6 55 V
7 55 V
6 55 V
6 -55 V
6 55 V
6 55 V
6 -55 V
6 55 V
7 55 V
6 -55 V
6 -55 V
6 55 V
6 -55 V
6 55 V
7 -55 V
6 -55 V
6 55 V
6 55 V
6 55 V
6 -55 V
7 55 V
6 -55 V
6 -55 V
6 55 V
6 55 V
6 -55 V
7 -55 V
6 -55 V
6 -55 V
6 55 V
6 -55 V
6 55 V
6 -55 V
7 55 V
6 -55 V
6 55 V
6 -55 V
6 55 V
6 55 V
7 -55 V
6 55 V
6 55 V
6 -55 V
6 -55 V
6 55 V
7 -55 V
6 -55 V
6 -55 V
6 55 V
6 -55 V
6 -55 V
7 -55 V
6 -55 V
6 55 V
6 -55 V
6 55 V
6 -55 V
6 55 V
7 55 V
6 -55 V
6 -55 V
6 55 V
6 55 V
6 -55 V
7 -55 V
6 -55 V
6 -55 V
6 55 V
6 -55 V
6 -55 V
7 55 V
6 -55 V
6 55 V
6 55 V
6 -55 V
6 55 V
7 55 V
6 -55 V
6 55 V
6 55 V
6 -55 V
6 -55 V
6 55 V
7 55 V
1347 895 L
6 55 V
6 55 V
6 -55 V
6 55 V
7 55 V
6 55 V
6 -55 V
6 -55 V
6 55 V
6 55 V
7 -55 V
6 -55 V
6 -55 V
6 -55 V
6 -55 V
6 -55 V
7 -55 V
6 55 V
6 55 V
6 -55 V
6 -55 V
6 -55 V
6 -55 V
7 55 V
6 -55 V
6 -55 V
6 55 V
6 55 V
6 -55 V
7 -55 V
6 55 V
6 55 V
6 55 V
6 -55 V
6 55 V
7 -55 V
6 55 V
6 55 V
6 -55 V
6 55 V
6 -55 V
7 55 V
6 55 V
6 -55 V
6 55 V
6 -55 V
6 -55 V
6 -55 V
7 55 V
6 -55 V
6 55 V
6 -55 V
6 -55 V
6 55 V
7 55 V
6 -55 V
6 55 V
6 55 V
6 55 V
6 55 V
7 55 V
6 -55 V
6 55 V
6 55 V
6 55 V
6 -55 V
7 -55 V
6 55 V
6 -55 V
6 55 V
6 55 V
6 55 V
6 -55 V
7 55 V
6 55 V
6 55 V
6 -55 V
6 55 V
6 -55 V
7 -55 V
6 55 V
6 -55 V
6 55 V
6 -55 V
6 -55 V
7 -55 V
6 -55 V
6 55 V
6 55 V
6 55 V
6 55 V
7 55 V
6 55 V
6 -55 V
6 -55 V
6 -55 V
6 -55 V
6 55 V
7 -55 V
6 55 V
6 -55 V
6 55 V
6 -55 V
6 -55 V
1994 950 L
6 55 V
6 -55 V
6 -55 V
6 55 V
6 55 V
7 55 V
6 -55 V
6 -55 V
6 -55 V
6 -55 V
6 -55 V
7 -55 V
6 55 V
6 55 V
6 55 V
6 55 V
6 -55 V
6 -55 V
7 55 V
6 55 V
6 -55 V
6 55 V
6 -55 V
6 55 V
7 55 V
6 -55 V
6 55 V
6 55 V
6 55 V
6 -55 V
7 55 V
6 55 V
6 -55 V
6 -55 V
6 55 V
6 -55 V
7 -55 V
6 -55 V
6 55 V
6 -55 V
6 55 V
6 55 V
6 -55 V
7 -55 V
6 55 V
6 -55 V
6 -55 V
6 55 V
6 55 V
7 -55 V
6 -55 V
6 55 V
6 -55 V
6 55 V
6 55 V
7 55 V
6 -55 V
6 55 V
6 -55 V
6 -55 V
6 55 V
7 55 V
6 55 V
6 -55 V
6 55 V
6 55 V
6 55 V
6 -55 V
7 55 V
6 55 V
6 -55 V
6 -55 V
6 55 V
6 55 V
7 55 V
6 55 V
6 55 V
6 55 V
6 -55 V
6 -55 V
7 55 V
6 -55 V
6 -55 V
6 55 V
6 55 V
6 -55 V
7 55 V
6 -55 V
6 55 V
6 -55 V
6 -55 V
6 -55 V
6 55 V
7 55 V
6 55 V
6 55 V
6 55 V
6 55 V
6 -55 V
7 55 V
6 -55 V
6 -55 V
6 -55 V
6 -55 V
2640 1335 L
7 -55 V
6 -55 V
6 55 V
6 -55 V
6 -55 V
6 55 V
7 55 V
6 -55 V
6 55 V
6 55 V
6 55 V
6 55 V
6 -55 V
7 55 V
6 -55 V
6 55 V
6 55 V
6 55 V
6 -55 V
7 55 V
6 55 V
6 55 V
6 55 V
6 -55 V
6 55 V
7 55 V
6 -55 V
6 -55 V
6 55 V
6 55 V
6 -55 V
7 55 V
6 55 V
6 55 V
6 55 V
6 55 V
6 -55 V
6 -55 V
7 55 V
6 -55 V
6 55 V
6 -55 V
6 55 V
6 -55 V
7 55 V
6 55 V
6 -55 V
6 55 V
6 55 V
6 -55 V
7 -55 V
6 55 V
6 55 V
6 55 V
6 -55 V
6 55 V
7 55 V
6 55 V
6 55 V
6 55 V
6 55 V
6 55 V
6 55 V
7 -55 V
6 55 V
6 55 V
6 55 V
6 -55 V
6 -55 V
7 -55 V
6 -55 V
6 55 V
6 55 V
6 -55 V
6 55 V
7 55 V
6 55 V
6 -55 V
6 55 V
6 -55 V
6 55 V
7 -55 V
6 55 V
6 55 V
6 -55 V
6 -55 V
6 55 V
6 55 V
7 -55 V
6 55 V
6 55 V
6 55 V
6 55 V
6 -55 V
7 -55 V
6 55 V
6 -55 V
6 55 V
6 -55 V
6 -55 V
7 -55 V
6 -55 V
6 55 V
6 55 V
3287 2710 L
6 -55 V
7 55 V
6 55 V
6 55 V
6 -55 V
6 55 V
6 55 V
6 55 V
7 -55 V
6 55 V
6 -55 V
6 -55 V
6 -55 V
6 -55 V
7 -55 V
6 -55 V
6 -55 V
6 55 V
6 -55 V
6 -55 V
7 -55 V
6 -55 V
6 55 V
6 55 V
6 55 V
6 55 V
7 -55 V
6 55 V
6 55 V
6 -55 V
6 55 V
6 55 V
6 -55 V
7 55 V
6 55 V
6 -55 V
6 55 V
6 -55 V
6 -55 V
7 -55 V
6 55 V
6 55 V
6 55 V
6 55 V
6 -55 V
7 55 V
6 55 V
6 -55 V
6 55 V
6 -55 V
6 55 V
7 55 V
6 55 V
6 -55 V
6 55 V
6 -55 V
6 55 V
6 -55 V
7 55 V
6 55 V
6 55 V
6 -55 V
6 55 V
6 55 V
7 55 V
6 55 V
6 55 V
6 -55 V
6 -55 V
6 -55 V
7 -55 V
6 -55 V
6 55 V
6 55 V
6 -55 V
6 55 V
7 55 V
6 -55 V
6 -55 V
6 -55 V
6 55 V
6 55 V
6 -55 V
7 -55 V
6 -55 V
6 -55 V
6 -55 V
6 -55 V
6 55 V
7 55 V
6 55 V
6 -55 V
6 -55 V
6 55 V
6 -55 V
7 55 V
6 -55 V
6 -55 V
6 55 V
6 -55 V
6 -55 V
7 -55 V
6 -55 V
6 55 V
3934 2765 L
6 -55 V
6 -55 V
6 55 V
7 55 V
6 -55 V
6 55 V
6 -55 V
6 -55 V
6 -55 V
7 55 V
6 55 V
6 -55 V
6 55 V
6 55 V
6 55 V
7 -55 V
6 -55 V
6 55 V
6 -55 V
6 55 V
6 -55 V
7 -55 V
6 -55 V
6 -55 V
6 55 V
6 55 V
6 55 V
6 55 V
7 55 V
6 55 V
6 -55 V
6 55 V
6 55 V
6 55 V
7 -55 V
6 55 V
6 55 V
6 -55 V
6 -55 V
6 55 V
7 -55 V
6 -55 V
6 55 V
6 55 V
6 -55 V
6 55 V
7 -55 V
6 55 V
6 -55 V
6 -55 V
6 -55 V
6 -55 V
6 -55 V
7 -55 V
6 -55 V
6 55 V
6 55 V
6 55 V
6 -55 V
7 -55 V
6 -55 V
6 -55 V
6 -55 V
6 -55 V
6 55 V
7 55 V
6 -55 V
6 -55 V
6 55 V
6 55 V
6 55 V
7 55 V
6 -55 V
6 -55 V
6 -55 V
6 55 V
6 55 V
6 55 V
7 55 V
6 -55 V
6 55 V
6 -55 V
6 55 V
6 -55 V
7 -55 V
6 55 V
6 55 V
6 -55 V
6 55 V
6 55 V
7 55 V
6 55 V
6 55 V
6 55 V
6 -55 V
6 -55 V
7 -55 V
6 55 V
6 -55 V
6 -55 V
6 -55 V
6 55 V
6 -55 V
7 55 V
4581 2710 L
6 55 V
6 -55 V
6 55 V
6 -55 V
7 55 V
6 -55 V
6 -55 V
6 -55 V
6 -55 V
6 -55 V
7 55 V
6 -55 V
6 -55 V
6 -55 V
6 55 V
6 -55 V
7 55 V
6 -55 V
6 55 V
6 55 V
6 55 V
6 55 V
6 55 V
7 -55 V
6 55 V
6 55 V
6 55 V
6 55 V
6 -55 V
7 -55 V
6 -55 V
6 -55 V
6 -55 V
6 55 V
6 -55 V
7 55 V
6 55 V
6 55 V
6 55 V
6 -55 V
6 55 V
7 -55 V
6 55 V
6 55 V
6 -55 V
6 -55 V
6 -55 V
6 -55 V
7 -55 V
6 55 V
6 -55 V
6 55 V
6 55 V
6 55 V
7 -55 V
6 -55 V
6 -55 V
6 55 V
6 55 V
6 55 V
7 -55 V
6 -55 V
6 -55 V
6 -55 V
6 55 V
6 55 V
7 55 V
6 -55 V
6 -55 V
6 55 V
6 -55 V
6 -55 V
6 55 V
7 55 V
6 55 V
6 55 V
6 55 V
6 55 V
6 55 V
7 -55 V
6 55 V
6 55 V
6 -55 V
6 -55 V
6 55 V
7 -55 V
6 -55 V
6 -55 V
6 55 V
6 -55 V
6 -55 V
7 55 V
6 55 V
6 -55 V
6 -55 V
6 -55 V
6 -55 V
6 -55 V
7 -55 V
6 -55 V
6 -55 V
6 55 V
6 -55 V
6 -55 V
5228 2325 L
6 55 V
6 55 V
6 55 V
6 -55 V
6 -55 V
7 55 V
6 55 V
6 55 V
6 -55 V
6 -55 V
6 55 V
7 -55 V
6 55 V
6 55 V
6 55 V
6 55 V
6 -55 V
6 -55 V
7 -55 V
6 -55 V
6 55 V
6 -55 V
6 -55 V
6 55 V
7 55 V
6 55 V
6 55 V
6 -55 V
6 55 V
6 -55 V
7 55 V
6 -55 V
6 55 V
6 55 V
6 -55 V
6 -55 V
7 55 V
6 55 V
6 55 V
6 -55 V
6 55 V
6 55 V
6 55 V
7 55 V
6 55 V
6 -55 V
6 -55 V
6 55 V
6 55 V
7 55 V
6 -55 V
6 -55 V
6 55 V
6 -55 V
6 -55 V
7 -55 V
6 55 V
6 55 V
6 55 V
6 55 V
6 55 V
7 55 V
6 -55 V
6 55 V
6 55 V
6 -55 V
6 55 V
6 55 V
7 -55 V
6 -55 V
6 -55 V
6 55 V
6 55 V
6 55 V
7 -55 V
6 55 V
6 55 V
6 55 V
6 55 V
6 -55 V
7 -55 V
6 -55 V
6 -55 V
6 -55 V
6 55 V
6 -55 V
7 55 V
6 55 V
6 55 V
6 -55 V
6 55 V
6 55 V
6 -55 V
7 55 V
6 -55 V
6 -55 V
6 55 V
6 55 V
6 -55 V
7 55 V
6 -55 V
6 -55 V
6 -55 V
6 -55 V
5874 3150 L
7 55 V
6 -55 V
6 55 V
6 -55 V
6 55 V
6 -55 V
7 55 V
6 -55 V
6 55 V
6 -55 V
6 55 V
6 55 V
6 55 V
7 55 V
6 55 V
6 55 V
6 55 V
6 -55 V
6 55 V
7 -55 V
6 -55 V
6 -55 V
6 55 V
6 55 V
6 -55 V
7 -55 V
6 -55 V
6 55 V
6 55 V
6 55 V
6 -55 V
7 55 V
6 55 V
6 -55 V
6 55 V
6 -55 V
6 55 V
6 -55 V
7 55 V
6 55 V
6 -55 V
6 -55 V
6 55 V
6 55 V
7 -55 V
6 55 V
6 55 V
6 55 V
6 55 V
6 55 V
7 55 V
6 -55 V
6 -55 V
6 -55 V
6 55 V
6 55 V
7 55 V
6 -55 V
6 -55 V
6 -55 V
6 55 V
6 55 V
6 -55 V
7 -55 V
6 -55 V
6 55 V
6 55 V
6 55 V
6 -55 V
7 55 V
6 55 V
6 55 V
6 -55 V
6 55 V
6 55 V
7 -55 V
6 55 V
6 55 V
6 55 V
6 -55 V
6 55 V
7 55 V
6 -55 V
6 -55 V
6 55 V
6 55 V
6 55 V
6 -55 V
7 -55 V
6 -55 V
6 55 V
6 55 V
6 55 V
6 -55 V
7 -55 V
6 -55 V
6 55 V
6 55 V
6 -55 V
6 55 V
7 -55 V
6 -55 V
6 55 V
6 -55 V
6521 4085 L
6 -55 V
7 55 V
6 -55 V
6 55 V
6 55 V
6 -55 V
6 -55 V
6 55 V
7 55 V
6 55 V
6 55 V
6 55 V
6 55 V
6 55 V
7 -55 V
6 -55 V
6 -55 V
6 -55 V
6 55 V
6 -55 V
7 -55 V
6 -55 V
6 -55 V
6 -55 V
6 55 V
6 55 V
7 55 V
6 55 V
6 55 V
6 -55 V
6 -55 V
6 55 V
6 55 V
7 55 V
6 -55 V
6 55 V
6 -55 V
6 -55 V
6 55 V
7 -55 V
6 55 V
6 55 V
6 -55 V
6 -55 V
6 55 V
7 55 V
6 55 V
6 55 V
6 -55 V
6 -55 V
6 55 V
7 55 V
6 55 V
6 55 V
6 55 V
stroke
LT2
700 400 M
6 55 V
6 55 V
6 -55 V
7 55 V
6 55 V
6 55 V
6 -55 V
6 55 V
6 55 V
7 55 V
6 -55 V
6 55 V
6 -55 V
6 -55 V
6 55 V
7 -55 V
6 55 V
6 55 V
6 55 V
6 55 V
6 55 V
7 55 V
6 55 V
6 55 V
6 55 V
6 -55 V
6 55 V
6 -55 V
7 55 V
6 -55 V
6 55 V
6 55 V
6 55 V
6 -55 V
7 -55 V
6 55 V
6 -55 V
6 -55 V
6 55 V
6 55 V
7 55 V
6 -55 V
6 -55 V
6 55 V
6 -55 V
6 55 V
7 55 V
6 -55 V
6 55 V
6 -55 V
6 55 V
6 -55 V
6 -55 V
7 55 V
6 55 V
6 55 V
6 55 V
6 55 V
6 55 V
7 -55 V
6 55 V
6 55 V
6 -55 V
6 55 V
6 55 V
7 55 V
6 55 V
6 55 V
6 55 V
6 55 V
6 55 V
7 55 V
6 -55 V
6 55 V
6 -55 V
6 -55 V
6 -55 V
6 -55 V
7 55 V
6 -55 V
6 55 V
6 -55 V
6 -55 V
6 -55 V
7 -55 V
6 55 V
6 55 V
6 -55 V
6 -55 V
6 55 V
7 55 V
6 55 V
6 -55 V
6 55 V
6 -55 V
6 55 V
7 -55 V
6 55 V
6 55 V
6 55 V
6 55 V
6 -55 V
6 -55 V
7 55 V
1347 1775 L
6 -55 V
6 -55 V
6 55 V
6 -55 V
7 -55 V
6 55 V
6 -55 V
6 55 V
6 -55 V
6 55 V
7 55 V
6 -55 V
6 55 V
6 55 V
6 55 V
6 55 V
7 -55 V
6 -55 V
6 -55 V
6 -55 V
6 -55 V
6 -55 V
6 -55 V
7 55 V
6 55 V
6 55 V
6 55 V
6 -55 V
6 55 V
7 55 V
6 55 V
6 -55 V
6 -55 V
6 55 V
6 -55 V
7 55 V
6 -55 V
6 55 V
6 -55 V
6 -55 V
6 -55 V
7 -55 V
6 -55 V
6 55 V
6 55 V
6 55 V
6 55 V
6 55 V
7 55 V
6 55 V
6 55 V
6 55 V
6 -55 V
6 55 V
7 55 V
6 55 V
6 55 V
6 55 V
6 -55 V
6 -55 V
7 55 V
6 55 V
6 -55 V
6 55 V
6 55 V
6 -55 V
7 55 V
6 -55 V
6 55 V
6 55 V
6 -55 V
6 -55 V
6 -55 V
7 -55 V
6 55 V
6 55 V
6 55 V
6 -55 V
6 55 V
7 55 V
6 -55 V
6 55 V
6 55 V
6 -55 V
6 -55 V
7 55 V
6 55 V
6 -55 V
6 -55 V
6 55 V
6 -55 V
7 55 V
6 -55 V
6 55 V
6 55 V
6 -55 V
6 -55 V
6 55 V
7 -55 V
6 55 V
6 55 V
6 55 V
6 55 V
6 55 V
1994 2490 L
6 55 V
6 55 V
6 55 V
6 -55 V
6 -55 V
7 55 V
6 55 V
6 55 V
6 55 V
6 -55 V
6 55 V
7 -55 V
6 -55 V
6 55 V
6 55 V
6 55 V
6 55 V
6 -55 V
7 -55 V
6 -55 V
6 55 V
6 -55 V
6 55 V
6 55 V
7 55 V
6 -55 V
6 -55 V
6 55 V
6 55 V
6 -55 V
7 -55 V
6 55 V
6 55 V
6 55 V
6 -55 V
6 -55 V
7 55 V
6 -55 V
6 -55 V
6 55 V
6 55 V
6 55 V
6 55 V
7 -55 V
6 55 V
6 55 V
6 55 V
6 -55 V
6 55 V
7 55 V
6 55 V
6 55 V
6 -55 V
6 55 V
6 -55 V
7 -55 V
6 -55 V
6 -55 V
6 55 V
6 55 V
6 55 V
7 55 V
6 -55 V
6 -55 V
6 -55 V
6 -55 V
6 -55 V
6 55 V
7 55 V
6 -55 V
6 55 V
6 55 V
6 55 V
6 55 V
7 -55 V
6 -55 V
6 55 V
6 -55 V
6 -55 V
6 55 V
7 55 V
6 -55 V
6 55 V
6 -55 V
6 55 V
6 -55 V
7 -55 V
6 55 V
6 55 V
6 -55 V
6 -55 V
6 -55 V
6 55 V
7 -55 V
6 55 V
6 -55 V
6 55 V
6 55 V
6 -55 V
7 55 V
6 -55 V
6 55 V
6 -55 V
6 55 V
2640 3095 L
7 55 V
6 -55 V
6 -55 V
6 55 V
6 -55 V
6 55 V
7 55 V
6 55 V
6 55 V
6 -55 V
6 -55 V
6 55 V
6 55 V
7 -55 V
6 -55 V
6 55 V
6 55 V
6 55 V
6 55 V
7 55 V
6 55 V
6 -55 V
6 -55 V
6 55 V
6 55 V
7 55 V
6 55 V
6 55 V
6 55 V
6 55 V
6 -55 V
7 -55 V
6 55 V
6 -55 V
6 -55 V
6 -55 V
6 -55 V
6 -55 V
7 55 V
6 -55 V
6 55 V
6 -55 V
6 -55 V
6 -55 V
7 -55 V
6 -55 V
6 -55 V
6 55 V
6 55 V
6 -55 V
7 55 V
6 -55 V
6 55 V
6 -55 V
6 -55 V
6 -55 V
7 55 V
6 -55 V
6 -55 V
6 55 V
6 55 V
6 55 V
6 -55 V
7 55 V
6 -55 V
6 -55 V
6 -55 V
6 55 V
6 -55 V
7 55 V
6 -55 V
6 -55 V
6 55 V
6 -55 V
6 -55 V
7 55 V
6 55 V
6 -55 V
6 -55 V
6 55 V
6 55 V
7 -55 V
6 -55 V
6 -55 V
6 55 V
6 55 V
6 -55 V
6 55 V
7 55 V
6 55 V
6 -55 V
6 55 V
6 55 V
6 55 V
7 -55 V
6 -55 V
6 55 V
6 -55 V
6 55 V
6 -55 V
7 -55 V
6 55 V
6 55 V
6 55 V
3287 3260 L
6 -55 V
7 -55 V
6 55 V
6 -55 V
6 55 V
6 -55 V
6 -55 V
6 55 V
7 55 V
6 -55 V
6 55 V
6 55 V
6 -55 V
6 55 V
7 55 V
6 -55 V
6 55 V
6 55 V
6 55 V
6 -55 V
7 -55 V
6 55 V
6 55 V
6 -55 V
6 -55 V
6 55 V
7 -55 V
6 -55 V
6 -55 V
6 55 V
6 55 V
6 -55 V
6 55 V
7 -55 V
6 55 V
6 -55 V
6 55 V
6 55 V
6 55 V
7 55 V
6 -55 V
6 -55 V
6 55 V
6 55 V
6 55 V
7 55 V
6 -55 V
6 -55 V
6 -55 V
6 -55 V
6 55 V
7 55 V
6 -55 V
6 55 V
6 -55 V
6 55 V
6 -55 V
6 55 V
7 -55 V
6 -55 V
6 -55 V
6 -55 V
6 55 V
6 -55 V
7 55 V
6 55 V
6 55 V
6 -55 V
6 55 V
6 -55 V
7 -55 V
6 55 V
6 -55 V
6 55 V
6 -55 V
6 -55 V
7 55 V
6 -55 V
6 55 V
6 -55 V
6 -55 V
6 -55 V
6 -55 V
7 -55 V
6 55 V
6 55 V
6 55 V
6 -55 V
6 -55 V
7 -55 V
6 55 V
6 55 V
6 55 V
6 -55 V
6 -55 V
7 -55 V
6 55 V
6 -55 V
6 55 V
6 -55 V
6 55 V
7 -55 V
6 55 V
6 55 V
3934 3095 L
6 -55 V
6 -55 V
6 -55 V
7 -55 V
6 -55 V
6 55 V
6 -55 V
6 -55 V
6 -55 V
7 55 V
6 -55 V
6 -55 V
6 -55 V
6 -55 V
6 55 V
7 55 V
6 55 V
6 55 V
6 55 V
6 -55 V
6 -55 V
7 55 V
6 -55 V
6 -55 V
6 55 V
6 55 V
6 55 V
6 55 V
7 55 V
6 -55 V
6 55 V
6 -55 V
6 55 V
6 55 V
7 55 V
6 -55 V
6 55 V
6 55 V
6 -55 V
6 55 V
7 -55 V
6 -55 V
6 55 V
6 -55 V
6 55 V
6 -55 V
7 55 V
6 -55 V
6 -55 V
6 -55 V
6 55 V
6 -55 V
6 55 V
7 -55 V
6 55 V
6 55 V
6 -55 V
6 -55 V
6 -55 V
7 55 V
6 -55 V
6 -55 V
6 55 V
6 55 V
6 -55 V
7 -55 V
6 55 V
6 -55 V
6 -55 V
6 55 V
6 -55 V
7 -55 V
6 55 V
6 -55 V
6 -55 V
6 55 V
6 55 V
6 55 V
7 -55 V
6 55 V
6 -55 V
6 -55 V
6 55 V
6 -55 V
7 -55 V
6 -55 V
6 -55 V
6 -55 V
6 -55 V
6 55 V
7 55 V
6 55 V
6 55 V
6 55 V
6 55 V
6 55 V
7 55 V
6 55 V
6 55 V
6 55 V
6 55 V
6 55 V
6 -55 V
7 -55 V
4581 2930 L
6 -55 V
6 -55 V
6 55 V
6 -55 V
7 55 V
6 -55 V
6 -55 V
6 -55 V
6 -55 V
6 -55 V
7 -55 V
6 55 V
6 -55 V
6 55 V
6 -55 V
6 55 V
7 -55 V
6 55 V
6 55 V
6 55 V
6 55 V
6 -55 V
6 55 V
7 55 V
6 55 V
6 -55 V
6 55 V
6 55 V
6 55 V
7 55 V
6 55 V
6 -55 V
6 55 V
6 55 V
6 55 V
7 -55 V
6 -55 V
6 -55 V
6 -55 V
6 55 V
6 -55 V
7 -55 V
6 -55 V
6 -55 V
6 -55 V
6 55 V
6 55 V
6 -55 V
7 55 V
6 55 V
6 55 V
6 55 V
6 55 V
6 55 V
7 55 V
6 -55 V
6 -55 V
6 55 V
6 55 V
6 55 V
7 55 V
6 -55 V
6 55 V
6 -55 V
6 -55 V
6 55 V
7 -55 V
6 55 V
6 55 V
6 55 V
6 55 V
6 55 V
6 55 V
7 55 V
6 -55 V
6 55 V
6 55 V
6 55 V
6 -55 V
7 -55 V
6 -55 V
6 -55 V
6 55 V
6 -55 V
6 55 V
7 55 V
6 55 V
6 55 V
6 -55 V
6 55 V
6 55 V
7 55 V
6 55 V
6 -55 V
6 -55 V
6 55 V
6 55 V
6 -55 V
7 55 V
6 -55 V
6 55 V
6 -55 V
6 -55 V
6 55 V
5228 3755 L
6 55 V
6 55 V
6 55 V
6 -55 V
6 -55 V
7 55 V
6 55 V
6 55 V
6 -55 V
6 55 V
6 55 V
7 55 V
6 -55 V
6 -55 V
6 55 V
6 -55 V
6 55 V
6 55 V
7 55 V
6 -55 V
6 55 V
6 55 V
6 55 V
6 55 V
7 55 V
6 -55 V
6 -55 V
6 -55 V
6 -55 V
6 -55 V
7 55 V
6 -55 V
6 -55 V
6 55 V
6 55 V
6 55 V
7 55 V
6 -55 V
6 55 V
6 55 V
6 -55 V
6 55 V
6 55 V
7 -55 V
6 -55 V
6 55 V
6 -55 V
6 -55 V
6 -55 V
7 -55 V
6 55 V
6 -55 V
6 -55 V
6 -55 V
6 55 V
7 -55 V
6 55 V
6 -55 V
6 -55 V
6 55 V
6 55 V
7 -55 V
6 55 V
6 55 V
6 55 V
6 -55 V
6 -55 V
6 55 V
7 -55 V
6 55 V
6 -55 V
6 -55 V
6 55 V
6 55 V
7 55 V
6 -55 V
6 -55 V
6 -55 V
6 -55 V
6 55 V
7 -55 V
6 55 V
6 55 V
6 55 V
6 -55 V
6 -55 V
7 -55 V
6 -55 V
6 55 V
6 55 V
6 55 V
6 55 V
6 55 V
7 55 V
6 -55 V
6 -55 V
6 -55 V
6 55 V
6 -55 V
7 55 V
6 55 V
6 -55 V
6 -55 V
6 -55 V
5874 3920 L
7 55 V
6 55 V
6 55 V
6 55 V
6 -55 V
6 -55 V
7 -55 V
6 -55 V
6 -55 V
6 55 V
6 55 V
6 -55 V
6 -55 V
7 -55 V
6 -55 V
6 55 V
6 -55 V
6 -55 V
6 -55 V
7 -55 V
6 -55 V
6 -55 V
6 55 V
6 -55 V
6 -55 V
7 55 V
6 55 V
6 -55 V
6 55 V
6 55 V
6 55 V
7 55 V
6 -55 V
6 55 V
6 -55 V
6 -55 V
6 55 V
6 55 V
7 -55 V
6 -55 V
6 -55 V
6 -55 V
6 55 V
6 -55 V
7 -55 V
6 55 V
6 -55 V
6 55 V
6 55 V
6 -55 V
7 -55 V
6 55 V
6 -55 V
6 -55 V
6 55 V
6 -55 V
7 55 V
6 -55 V
6 -55 V
6 55 V
6 -55 V
6 55 V
6 55 V
7 -55 V
6 55 V
6 -55 V
6 -55 V
6 55 V
6 -55 V
7 55 V
6 -55 V
6 55 V
6 55 V
6 55 V
6 -55 V
7 55 V
6 55 V
6 55 V
6 55 V
6 -55 V
6 -55 V
7 55 V
6 55 V
6 55 V
6 -55 V
6 55 V
6 -55 V
6 55 V
7 55 V
6 -55 V
6 55 V
6 55 V
6 55 V
6 55 V
7 -55 V
6 55 V
6 55 V
6 -55 V
6 55 V
6 -55 V
7 -55 V
6 -55 V
6 55 V
6 55 V
6521 3975 L
6 55 V
7 55 V
6 -55 V
6 -55 V
6 55 V
6 -55 V
6 55 V
6 -55 V
7 55 V
6 55 V
6 -55 V
6 -55 V
6 -55 V
6 55 V
7 55 V
6 -55 V
6 55 V
6 55 V
6 55 V
6 -55 V
7 55 V
6 -55 V
6 55 V
6 -55 V
6 -55 V
6 -55 V
7 55 V
6 55 V
6 55 V
6 55 V
6 -55 V
6 -55 V
6 55 V
7 -55 V
6 55 V
6 55 V
6 55 V
6 -55 V
6 55 V
7 55 V
6 55 V
6 -55 V
6 -55 V
6 55 V
6 -55 V
7 -55 V
6 -55 V
6 -55 V
6 55 V
6 55 V
6 55 V
7 55 V
6 55 V
6 55 V
6 55 V
stroke
LT3
700 400 M
6 55 V
6 55 V
6 55 V
7 55 V
6 -55 V
6 55 V
6 55 V
6 55 V
6 55 V
7 -55 V
6 55 V
6 -55 V
6 -55 V
6 -55 V
6 -55 V
7 55 V
6 -55 V
6 55 V
6 -55 V
6 -55 V
6 55 V
7 -55 V
6 55 V
6 -55 V
6 55 V
6 55 V
6 55 V
6 -55 V
7 55 V
6 -55 V
6 55 V
6 -55 V
6 -55 V
6 55 V
7 55 V
6 55 V
6 55 V
6 55 V
6 55 V
6 55 V
7 55 V
6 55 V
6 -55 V
6 -55 V
6 -55 V
6 55 V
7 55 V
6 55 V
6 55 V
6 55 V
6 -55 V
6 -55 V
6 55 V
7 -55 V
6 -55 V
6 55 V
6 -55 V
6 55 V
6 -55 V
7 55 V
6 -55 V
6 55 V
6 55 V
6 55 V
6 -55 V
7 -55 V
6 55 V
6 -55 V
6 -55 V
6 -55 V
6 -55 V
7 -55 V
6 -55 V
6 -55 V
6 -55 V
6 -55 V
6 55 V
6 -55 V
7 55 V
6 -55 V
6 55 V
6 55 V
6 55 V
6 -55 V
7 55 V
6 -55 V
6 -55 V
6 -55 V
6 55 V
6 -55 V
7 55 V
6 55 V
6 -55 V
6 -55 V
6 -55 V
6 55 V
7 55 V
6 55 V
6 55 V
6 55 V
6 55 V
6 -55 V
6 55 V
7 -55 V
1347 895 L
6 -55 V
6 -55 V
6 -55 V
6 -55 V
7 55 V
6 -55 V
6 55 V
6 55 V
6 55 V
6 -55 V
7 55 V
6 -55 V
6 55 V
6 -55 V
6 55 V
6 -55 V
7 55 V
6 55 V
6 -55 V
6 -55 V
6 -55 V
6 -55 V
6 55 V
7 -55 V
6 55 V
6 -55 V
6 -55 V
6 55 V
6 55 V
7 55 V
6 -55 V
6 -55 V
6 55 V
6 -55 V
6 -55 V
7 -55 V
6 55 V
6 55 V
6 -55 V
6 55 V
6 55 V
7 -55 V
6 55 V
6 -55 V
6 55 V
6 -55 V
6 55 V
6 -55 V
7 -55 V
6 55 V
6 55 V
6 -55 V
6 -55 V
6 55 V
7 55 V
6 55 V
6 55 V
6 55 V
6 55 V
6 55 V
7 55 V
6 -55 V
6 -55 V
6 55 V
6 -55 V
6 -55 V
7 -55 V
6 55 V
6 -55 V
6 -55 V
6 -55 V
6 55 V
6 -55 V
7 -55 V
6 55 V
6 55 V
6 55 V
6 55 V
6 -55 V
7 55 V
6 55 V
6 55 V
6 55 V
6 55 V
6 55 V
7 -55 V
6 -55 V
6 -55 V
6 55 V
6 -55 V
6 55 V
7 -55 V
6 55 V
6 -55 V
6 -55 V
6 55 V
6 -55 V
6 -55 V
7 55 V
6 -55 V
6 55 V
6 -55 V
6 -55 V
6 55 V
1994 840 L
6 55 V
6 55 V
6 55 V
6 -55 V
6 -55 V
7 -55 V
6 -55 V
6 55 V
6 55 V
6 -55 V
6 55 V
7 55 V
6 -55 V
6 -55 V
6 55 V
6 -55 V
6 55 V
6 55 V
7 -55 V
6 55 V
6 55 V
6 55 V
6 55 V
6 -55 V
7 -55 V
6 55 V
6 -55 V
6 -55 V
6 55 V
6 55 V
7 -55 V
6 -55 V
6 -55 V
6 -55 V
6 55 V
6 -55 V
7 -55 V
6 -55 V
6 -55 V
6 55 V
6 -55 V
6 55 V
6 -55 V
7 55 V
6 -55 V
6 55 V
6 55 V
6 -55 V
6 55 V
7 55 V
6 55 V
6 -55 V
6 55 V
6 55 V
6 55 V
7 -55 V
6 -55 V
6 -55 V
6 55 V
6 55 V
6 -55 V
7 55 V
6 -55 V
6 55 V
6 55 V
6 -55 V
6 55 V
6 55 V
7 55 V
6 55 V
6 55 V
6 -55 V
6 55 V
6 -55 V
7 -55 V
6 55 V
6 55 V
6 55 V
6 55 V
6 -55 V
7 -55 V
6 -55 V
6 55 V
6 55 V
6 55 V
6 -55 V
7 -55 V
6 -55 V
6 -55 V
6 55 V
6 -55 V
6 55 V
6 -55 V
7 55 V
6 55 V
6 55 V
6 55 V
6 55 V
6 55 V
7 55 V
6 55 V
6 -55 V
6 -55 V
6 -55 V
2640 1445 L
7 -55 V
6 55 V
6 -55 V
6 -55 V
6 -55 V
6 55 V
7 55 V
6 -55 V
6 -55 V
6 -55 V
6 -55 V
6 -55 V
6 55 V
7 -55 V
6 -55 V
6 55 V
6 -55 V
6 -55 V
6 55 V
7 55 V
6 -55 V
6 55 V
6 55 V
6 -55 V
6 55 V
7 55 V
6 -55 V
6 55 V
6 55 V
6 55 V
6 -55 V
7 -55 V
6 55 V
6 -55 V
6 -55 V
6 55 V
6 -55 V
6 -55 V
7 -55 V
6 55 V
6 -55 V
6 -55 V
6 -55 V
6 55 V
7 55 V
6 -55 V
6 55 V
6 -55 V
6 55 V
6 -55 V
7 -55 V
6 -55 V
6 55 V
6 -55 V
6 -55 V
6 -55 V
7 55 V
6 -55 V
6 55 V
6 -55 V
6 55 V
6 -55 V
6 -55 V
7 55 V
6 55 V
6 -55 V
6 -55 V
6 55 V
6 -55 V
7 -55 V
6 55 V
6 55 V
6 55 V
6 -55 V
6 55 V
7 55 V
6 -55 V
6 -55 V
6 55 V
6 55 V
6 55 V
7 -55 V
6 -55 V
6 55 V
6 55 V
6 -55 V
6 -55 V
6 55 V
7 -55 V
6 -55 V
6 55 V
6 55 V
6 -55 V
6 55 V
7 -55 V
6 -55 V
6 -55 V
6 55 V
6 55 V
6 -55 V
7 -55 V
6 55 V
6 -55 V
6 55 V
3287 730 L
6 55 V
7 -55 V
6 55 V
6 -55 V
6 55 V
6 55 V
6 55 V
6 55 V
7 55 V
6 -55 V
6 -55 V
6 -55 V
6 55 V
6 55 V
7 55 V
6 55 V
6 55 V
6 -55 V
6 -55 V
6 55 V
7 55 V
6 -55 V
6 -55 V
6 55 V
6 55 V
6 55 V
7 -55 V
6 55 V
6 -55 V
6 55 V
6 55 V
6 55 V
6 -55 V
7 -55 V
6 55 V
6 -55 V
6 -55 V
6 -55 V
6 55 V
7 -55 V
6 -55 V
6 55 V
6 55 V
6 55 V
6 55 V
7 55 V
6 -55 V
6 -55 V
6 -55 V
6 55 V
6 -55 V
7 55 V
6 55 V
6 -55 V
6 55 V
6 55 V
6 55 V
6 55 V
7 55 V
6 55 V
6 55 V
6 55 V
6 55 V
6 -55 V
7 -55 V
6 -55 V
6 -55 V
6 55 V
6 -55 V
6 -55 V
7 55 V
6 55 V
6 -55 V
6 -55 V
6 -55 V
6 55 V
7 -55 V
6 -55 V
6 55 V
6 -55 V
6 55 V
6 55 V
6 -55 V
7 55 V
6 -55 V
6 55 V
6 -55 V
6 55 V
6 55 V
7 55 V
6 55 V
6 55 V
6 -55 V
6 -55 V
6 55 V
7 55 V
6 -55 V
6 -55 V
6 -55 V
6 -55 V
6 55 V
7 -55 V
6 55 V
6 55 V
3934 1555 L
6 55 V
6 55 V
6 -55 V
7 55 V
6 55 V
6 55 V
6 -55 V
6 -55 V
6 -55 V
7 55 V
6 55 V
6 -55 V
6 -55 V
6 55 V
6 55 V
7 -55 V
6 -55 V
6 55 V
6 -55 V
6 55 V
6 -55 V
7 -55 V
6 55 V
6 -55 V
6 55 V
6 -55 V
6 55 V
6 55 V
7 -55 V
6 -55 V
6 55 V
6 55 V
6 -55 V
6 -55 V
7 -55 V
6 55 V
6 55 V
6 55 V
6 -55 V
6 55 V
7 55 V
6 55 V
6 55 V
6 -55 V
6 -55 V
6 -55 V
7 -55 V
6 55 V
6 55 V
6 -55 V
6 -55 V
6 55 V
6 55 V
7 55 V
6 55 V
6 55 V
6 55 V
6 55 V
6 55 V
7 -55 V
6 55 V
6 55 V
6 55 V
6 -55 V
6 -55 V
7 55 V
6 -55 V
6 -55 V
6 -55 V
6 -55 V
6 55 V
7 -55 V
6 -55 V
6 55 V
6 55 V
6 55 V
6 55 V
6 55 V
7 55 V
6 55 V
6 55 V
6 -55 V
6 -55 V
6 55 V
7 55 V
6 55 V
6 55 V
6 -55 V
6 -55 V
6 55 V
7 -55 V
6 55 V
6 -55 V
6 -55 V
6 55 V
6 -55 V
7 55 V
6 55 V
6 55 V
6 -55 V
6 -55 V
6 55 V
6 -55 V
7 55 V
4581 2380 L
6 -55 V
6 -55 V
6 -55 V
6 55 V
7 55 V
6 -55 V
6 55 V
6 55 V
6 -55 V
6 -55 V
7 55 V
6 -55 V
6 55 V
6 55 V
6 55 V
6 55 V
7 55 V
6 -55 V
6 -55 V
6 -55 V
6 55 V
6 -55 V
6 55 V
7 -55 V
6 -55 V
6 -55 V
6 -55 V
6 -55 V
6 -55 V
7 -55 V
6 -55 V
6 -55 V
6 55 V
6 55 V
6 55 V
7 55 V
6 -55 V
6 -55 V
6 55 V
6 -55 V
6 -55 V
7 -55 V
6 55 V
6 -55 V
6 55 V
6 -55 V
6 -55 V
6 55 V
7 -55 V
6 55 V
6 55 V
6 -55 V
6 55 V
6 55 V
7 -55 V
6 55 V
6 -55 V
6 -55 V
6 -55 V
6 -55 V
7 55 V
6 55 V
6 -55 V
6 -55 V
6 -55 V
6 55 V
7 -55 V
6 -55 V
6 55 V
6 -55 V
6 55 V
6 55 V
6 -55 V
7 -55 V
6 55 V
6 -55 V
6 55 V
6 -55 V
6 55 V
7 -55 V
6 -55 V
6 -55 V
6 -55 V
6 55 V
6 55 V
7 55 V
6 55 V
6 55 V
6 55 V
6 -55 V
6 -55 V
7 -55 V
6 -55 V
6 -55 V
6 55 V
6 55 V
6 55 V
6 55 V
7 55 V
6 -55 V
6 55 V
6 -55 V
6 -55 V
6 -55 V
5228 1665 L
6 55 V
6 55 V
6 -55 V
6 55 V
6 -55 V
7 -55 V
6 -55 V
6 -55 V
6 -55 V
6 -55 V
6 55 V
7 55 V
6 -55 V
6 55 V
6 -55 V
6 55 V
6 55 V
6 55 V
7 55 V
6 55 V
6 -55 V
6 55 V
6 55 V
6 -55 V
7 55 V
6 -55 V
6 55 V
6 -55 V
6 55 V
6 -55 V
7 55 V
6 -55 V
6 -55 V
6 55 V
6 -55 V
6 -55 V
7 55 V
6 55 V
6 55 V
6 -55 V
6 55 V
6 55 V
6 55 V
7 -55 V
6 -55 V
6 55 V
6 -55 V
6 -55 V
6 55 V
7 55 V
6 -55 V
6 55 V
6 -55 V
6 -55 V
6 -55 V
7 -55 V
6 55 V
6 55 V
6 -55 V
6 55 V
6 -55 V
7 -55 V
6 55 V
6 -55 V
6 -55 V
6 55 V
6 55 V
6 55 V
7 -55 V
6 55 V
6 55 V
6 55 V
6 -55 V
6 -55 V
7 -55 V
6 55 V
6 -55 V
6 55 V
6 -55 V
6 55 V
7 -55 V
6 -55 V
6 55 V
6 55 V
6 -55 V
6 55 V
7 -55 V
6 -55 V
6 -55 V
6 -55 V
6 -55 V
6 55 V
6 55 V
7 55 V
6 55 V
6 -55 V
6 -55 V
6 55 V
6 -55 V
7 55 V
6 -55 V
6 -55 V
6 55 V
6 55 V
5874 1610 L
7 55 V
6 -55 V
6 55 V
6 -55 V
6 55 V
6 55 V
7 55 V
6 -55 V
6 -55 V
6 55 V
6 -55 V
6 -55 V
6 55 V
7 55 V
6 55 V
6 55 V
6 -55 V
6 -55 V
6 -55 V
7 55 V
6 55 V
6 -55 V
6 -55 V
6 -55 V
6 -55 V
7 55 V
6 55 V
6 -55 V
6 -55 V
6 55 V
6 55 V
7 55 V
6 -55 V
6 -55 V
6 55 V
6 -55 V
6 -55 V
6 -55 V
7 55 V
6 55 V
6 55 V
6 55 V
6 -55 V
6 55 V
7 55 V
6 -55 V
6 -55 V
6 55 V
6 55 V
6 -55 V
7 55 V
6 55 V
6 -55 V
6 55 V
6 55 V
6 55 V
7 -55 V
6 55 V
6 55 V
6 55 V
6 55 V
6 -55 V
6 55 V
7 55 V
6 -55 V
6 -55 V
6 55 V
6 -55 V
6 -55 V
7 55 V
6 55 V
6 -55 V
6 -55 V
6 55 V
6 55 V
7 -55 V
6 55 V
6 -55 V
6 55 V
6 -55 V
6 -55 V
7 -55 V
6 -55 V
6 55 V
6 -55 V
6 55 V
6 55 V
6 -55 V
7 -55 V
6 -55 V
6 -55 V
6 -55 V
6 -55 V
6 55 V
7 55 V
6 55 V
6 -55 V
6 -55 V
6 55 V
6 -55 V
7 55 V
6 55 V
6 -55 V
6 55 V
6521 1885 L
6 -55 V
7 55 V
6 55 V
6 55 V
6 55 V
6 -55 V
6 55 V
6 -55 V
7 55 V
6 -55 V
6 55 V
6 55 V
6 55 V
6 -55 V
7 -55 V
6 55 V
6 55 V
6 -55 V
6 -55 V
6 55 V
7 -55 V
6 -55 V
6 55 V
6 55 V
6 -55 V
6 55 V
7 -55 V
6 55 V
6 55 V
6 55 V
6 55 V
6 55 V
6 55 V
7 55 V
6 55 V
6 -55 V
6 -55 V
6 -55 V
6 55 V
7 -55 V
6 -55 V
6 55 V
6 -55 V
6 -55 V
6 55 V
7 -55 V
6 55 V
6 55 V
6 55 V
6 -55 V
6 55 V
7 55 V
6 55 V
6 55 V
6 55 V
stroke
LTb
700 4800 M
700 400 L
6160 0 V
0 4400 V
-6160 0 V
1.000 UP
stroke
grestore
end
showpage
  }}%
  \put(6860,200){\makebox(0,0){\strut{} 1000}}%
  \put(5628,200){\makebox(0,0){\strut{} 800}}%
  \put(4396,200){\makebox(0,0){\strut{} 600}}%
  \put(3164,200){\makebox(0,0){\strut{} 400}}%
  \put(1932,200){\makebox(0,0){\strut{} 200}}%
  \put(700,200){\makebox(0,0){\strut{} 0}}%
  \put(580,4800){\makebox(0,0)[r]{\strut{} 80}}%
  \put(580,4250){\makebox(0,0)[r]{\strut{} 70}}%
  \put(580,3700){\makebox(0,0)[r]{\strut{} 60}}%
  \put(580,3150){\makebox(0,0)[r]{\strut{} 50}}%
  \put(580,2600){\makebox(0,0)[r]{\strut{} 40}}%
  \put(580,2050){\makebox(0,0)[r]{\strut{} 30}}%
  \put(580,1500){\makebox(0,0)[r]{\strut{} 20}}%
  \put(580,950){\makebox(0,0)[r]{\strut{} 10}}%
  \put(580,400){\makebox(0,0)[r]{\strut{} 0}}%
\end{picture}%
\endgroup
 

%% file: Paths2-1.tex
% GNUPLOT: LaTeX picture with Postscript
\begingroup%
\makeatletter%
\newcommand{\GNUPLOTspecial}{%
  \@sanitize\catcode`\%=14\relax\special}%
\setlength{\unitlength}{0.0500bp}%
\begin{picture}(7200,5040)(0,0)%
  {\GNUPLOTspecial{"
%!PS-Adobe-2.0 EPSF-2.0
%%Title: C:/Documents and Settings/ponty/My Documents/Tex/culminants/Paths2-1.tex
%%Creator: gnuplot 4.2 patchlevel rc3
%%CreationDate: Thu Mar 01 02:37:54 2007
%%DocumentFonts: 
%%BoundingBox: 0 0 360 252
%%EndComments
%%BeginProlog
/gnudict 256 dict def
gnudict begin
%
% The following 6 true/false flags may be edited by hand if required
% The unit line width may also be changed
%
/Color false def
/Blacktext false def
/Solid false def
/Dashlength 1 def
/Landscape false def
/Level1 false def
/Rounded false def
/TransparentPatterns false def
/gnulinewidth 5.000 def
/userlinewidth gnulinewidth def
/vshift -66 def
/dl1 {
  10.0 Dashlength mul mul
  Rounded { currentlinewidth 0.75 mul sub dup 0 le { pop 0.01 } if } if
} def
/dl2 {
  10.0 Dashlength mul mul
  Rounded { currentlinewidth 0.75 mul add } if
} def
/hpt_ 31.5 def
/vpt_ 31.5 def
/hpt hpt_ def
/vpt vpt_ def
Level1 {} {
/SDict 10 dict def
systemdict /pdfmark known not {
  userdict /pdfmark systemdict /cleartomark get put
} if
SDict begin [
  /Title (C:/Documents and Settings/ponty/My Documents/Tex/culminants/Paths2-1.tex)
  /Subject (gnuplot plot)
  /Creator (gnuplot 4.2 patchlevel rc3)
  /Author (ponty)
%  /Producer (gnuplot)
%  /Keywords ()
  /CreationDate (Thu Mar 01 02:37:54 2007)
  /DOCINFO pdfmark
end
} ifelse
%
% Gnuplot Prolog Version 4.2 (August 2006)
%
/M {moveto} bind def
/L {lineto} bind def
/R {rmoveto} bind def
/V {rlineto} bind def
/N {newpath moveto} bind def
/Z {closepath} bind def
/C {setrgbcolor} bind def
/f {rlineto fill} bind def
/vpt2 vpt 2 mul def
/hpt2 hpt 2 mul def
/Lshow {currentpoint stroke M 0 vshift R 
	Blacktext {gsave 0 setgray show grestore} {show} ifelse} def
/Rshow {currentpoint stroke M dup stringwidth pop neg vshift R
	Blacktext {gsave 0 setgray show grestore} {show} ifelse} def
/Cshow {currentpoint stroke M dup stringwidth pop -2 div vshift R 
	Blacktext {gsave 0 setgray show grestore} {show} ifelse} def
/UP {dup vpt_ mul /vpt exch def hpt_ mul /hpt exch def
  /hpt2 hpt 2 mul def /vpt2 vpt 2 mul def} def
/DL {Color {setrgbcolor Solid {pop []} if 0 setdash}
 {pop pop pop 0 setgray Solid {pop []} if 0 setdash} ifelse} def
/BL {stroke userlinewidth 2 mul setlinewidth
	Rounded {1 setlinejoin 1 setlinecap} if} def
/AL {stroke userlinewidth 2 div setlinewidth
	Rounded {1 setlinejoin 1 setlinecap} if} def
/UL {dup gnulinewidth mul /userlinewidth exch def
	dup 1 lt {pop 1} if 10 mul /udl exch def} def
/PL {stroke userlinewidth setlinewidth
	Rounded {1 setlinejoin 1 setlinecap} if} def
% Default Line colors
/LCw {1 1 1} def
/LCb {0 0 0} def
/LCa {0 0 0} def
/LC0 {1 0 0} def
/LC1 {0 1 0} def
/LC2 {0 0 1} def
/LC3 {1 0 1} def
/LC4 {0 1 1} def
/LC5 {1 1 0} def
/LC6 {0 0 0} def
/LC7 {1 0.3 0} def
/LC8 {0.5 0.5 0.5} def
% Default Line Types
/LTw {PL [] 1 setgray} def
/LTb {BL [] LCb DL} def
/LTa {AL [1 udl mul 2 udl mul] 0 setdash LCa setrgbcolor} def
/LT0 {PL [] LC0 DL} def
/LT1 {PL [4 dl1 2 dl2] LC1 DL} def
/LT2 {PL [2 dl1 3 dl2] LC2 DL} def
/LT3 {PL [1 dl1 1.5 dl2] LC3 DL} def
/LT4 {PL [6 dl1 2 dl2 1 dl1 2 dl2] LC4 DL} def
/LT5 {PL [3 dl1 3 dl2 1 dl1 3 dl2] LC5 DL} def
/LT6 {PL [2 dl1 2 dl2 2 dl1 6 dl2] LC6 DL} def
/LT7 {PL [1 dl1 2 dl2 6 dl1 2 dl2 1 dl1 2 dl2] LC7 DL} def
/LT8 {PL [2 dl1 2 dl2 2 dl1 2 dl2 2 dl1 2 dl2 2 dl1 4 dl2] LC8 DL} def
/Pnt {stroke [] 0 setdash gsave 1 setlinecap M 0 0 V stroke grestore} def
/Dia {stroke [] 0 setdash 2 copy vpt add M
  hpt neg vpt neg V hpt vpt neg V
  hpt vpt V hpt neg vpt V closepath stroke
  Pnt} def
/Pls {stroke [] 0 setdash vpt sub M 0 vpt2 V
  currentpoint stroke M
  hpt neg vpt neg R hpt2 0 V stroke
 } def
/Box {stroke [] 0 setdash 2 copy exch hpt sub exch vpt add M
  0 vpt2 neg V hpt2 0 V 0 vpt2 V
  hpt2 neg 0 V closepath stroke
  Pnt} def
/Crs {stroke [] 0 setdash exch hpt sub exch vpt add M
  hpt2 vpt2 neg V currentpoint stroke M
  hpt2 neg 0 R hpt2 vpt2 V stroke} def
/TriU {stroke [] 0 setdash 2 copy vpt 1.12 mul add M
  hpt neg vpt -1.62 mul V
  hpt 2 mul 0 V
  hpt neg vpt 1.62 mul V closepath stroke
  Pnt} def
/Star {2 copy Pls Crs} def
/BoxF {stroke [] 0 setdash exch hpt sub exch vpt add M
  0 vpt2 neg V hpt2 0 V 0 vpt2 V
  hpt2 neg 0 V closepath fill} def
/TriUF {stroke [] 0 setdash vpt 1.12 mul add M
  hpt neg vpt -1.62 mul V
  hpt 2 mul 0 V
  hpt neg vpt 1.62 mul V closepath fill} def
/TriD {stroke [] 0 setdash 2 copy vpt 1.12 mul sub M
  hpt neg vpt 1.62 mul V
  hpt 2 mul 0 V
  hpt neg vpt -1.62 mul V closepath stroke
  Pnt} def
/TriDF {stroke [] 0 setdash vpt 1.12 mul sub M
  hpt neg vpt 1.62 mul V
  hpt 2 mul 0 V
  hpt neg vpt -1.62 mul V closepath fill} def
/DiaF {stroke [] 0 setdash vpt add M
  hpt neg vpt neg V hpt vpt neg V
  hpt vpt V hpt neg vpt V closepath fill} def
/Pent {stroke [] 0 setdash 2 copy gsave
  translate 0 hpt M 4 {72 rotate 0 hpt L} repeat
  closepath stroke grestore Pnt} def
/PentF {stroke [] 0 setdash gsave
  translate 0 hpt M 4 {72 rotate 0 hpt L} repeat
  closepath fill grestore} def
/Circle {stroke [] 0 setdash 2 copy
  hpt 0 360 arc stroke Pnt} def
/CircleF {stroke [] 0 setdash hpt 0 360 arc fill} def
/C0 {BL [] 0 setdash 2 copy moveto vpt 90 450 arc} bind def
/C1 {BL [] 0 setdash 2 copy moveto
	2 copy vpt 0 90 arc closepath fill
	vpt 0 360 arc closepath} bind def
/C2 {BL [] 0 setdash 2 copy moveto
	2 copy vpt 90 180 arc closepath fill
	vpt 0 360 arc closepath} bind def
/C3 {BL [] 0 setdash 2 copy moveto
	2 copy vpt 0 180 arc closepath fill
	vpt 0 360 arc closepath} bind def
/C4 {BL [] 0 setdash 2 copy moveto
	2 copy vpt 180 270 arc closepath fill
	vpt 0 360 arc closepath} bind def
/C5 {BL [] 0 setdash 2 copy moveto
	2 copy vpt 0 90 arc
	2 copy moveto
	2 copy vpt 180 270 arc closepath fill
	vpt 0 360 arc} bind def
/C6 {BL [] 0 setdash 2 copy moveto
	2 copy vpt 90 270 arc closepath fill
	vpt 0 360 arc closepath} bind def
/C7 {BL [] 0 setdash 2 copy moveto
	2 copy vpt 0 270 arc closepath fill
	vpt 0 360 arc closepath} bind def
/C8 {BL [] 0 setdash 2 copy moveto
	2 copy vpt 270 360 arc closepath fill
	vpt 0 360 arc closepath} bind def
/C9 {BL [] 0 setdash 2 copy moveto
	2 copy vpt 270 450 arc closepath fill
	vpt 0 360 arc closepath} bind def
/C10 {BL [] 0 setdash 2 copy 2 copy moveto vpt 270 360 arc closepath fill
	2 copy moveto
	2 copy vpt 90 180 arc closepath fill
	vpt 0 360 arc closepath} bind def
/C11 {BL [] 0 setdash 2 copy moveto
	2 copy vpt 0 180 arc closepath fill
	2 copy moveto
	2 copy vpt 270 360 arc closepath fill
	vpt 0 360 arc closepath} bind def
/C12 {BL [] 0 setdash 2 copy moveto
	2 copy vpt 180 360 arc closepath fill
	vpt 0 360 arc closepath} bind def
/C13 {BL [] 0 setdash 2 copy moveto
	2 copy vpt 0 90 arc closepath fill
	2 copy moveto
	2 copy vpt 180 360 arc closepath fill
	vpt 0 360 arc closepath} bind def
/C14 {BL [] 0 setdash 2 copy moveto
	2 copy vpt 90 360 arc closepath fill
	vpt 0 360 arc} bind def
/C15 {BL [] 0 setdash 2 copy vpt 0 360 arc closepath fill
	vpt 0 360 arc closepath} bind def
/Rec {newpath 4 2 roll moveto 1 index 0 rlineto 0 exch rlineto
	neg 0 rlineto closepath} bind def
/Square {dup Rec} bind def
/Bsquare {vpt sub exch vpt sub exch vpt2 Square} bind def
/S0 {BL [] 0 setdash 2 copy moveto 0 vpt rlineto BL Bsquare} bind def
/S1 {BL [] 0 setdash 2 copy vpt Square fill Bsquare} bind def
/S2 {BL [] 0 setdash 2 copy exch vpt sub exch vpt Square fill Bsquare} bind def
/S3 {BL [] 0 setdash 2 copy exch vpt sub exch vpt2 vpt Rec fill Bsquare} bind def
/S4 {BL [] 0 setdash 2 copy exch vpt sub exch vpt sub vpt Square fill Bsquare} bind def
/S5 {BL [] 0 setdash 2 copy 2 copy vpt Square fill
	exch vpt sub exch vpt sub vpt Square fill Bsquare} bind def
/S6 {BL [] 0 setdash 2 copy exch vpt sub exch vpt sub vpt vpt2 Rec fill Bsquare} bind def
/S7 {BL [] 0 setdash 2 copy exch vpt sub exch vpt sub vpt vpt2 Rec fill
	2 copy vpt Square fill Bsquare} bind def
/S8 {BL [] 0 setdash 2 copy vpt sub vpt Square fill Bsquare} bind def
/S9 {BL [] 0 setdash 2 copy vpt sub vpt vpt2 Rec fill Bsquare} bind def
/S10 {BL [] 0 setdash 2 copy vpt sub vpt Square fill 2 copy exch vpt sub exch vpt Square fill
	Bsquare} bind def
/S11 {BL [] 0 setdash 2 copy vpt sub vpt Square fill 2 copy exch vpt sub exch vpt2 vpt Rec fill
	Bsquare} bind def
/S12 {BL [] 0 setdash 2 copy exch vpt sub exch vpt sub vpt2 vpt Rec fill Bsquare} bind def
/S13 {BL [] 0 setdash 2 copy exch vpt sub exch vpt sub vpt2 vpt Rec fill
	2 copy vpt Square fill Bsquare} bind def
/S14 {BL [] 0 setdash 2 copy exch vpt sub exch vpt sub vpt2 vpt Rec fill
	2 copy exch vpt sub exch vpt Square fill Bsquare} bind def
/S15 {BL [] 0 setdash 2 copy Bsquare fill Bsquare} bind def
/D0 {gsave translate 45 rotate 0 0 S0 stroke grestore} bind def
/D1 {gsave translate 45 rotate 0 0 S1 stroke grestore} bind def
/D2 {gsave translate 45 rotate 0 0 S2 stroke grestore} bind def
/D3 {gsave translate 45 rotate 0 0 S3 stroke grestore} bind def
/D4 {gsave translate 45 rotate 0 0 S4 stroke grestore} bind def
/D5 {gsave translate 45 rotate 0 0 S5 stroke grestore} bind def
/D6 {gsave translate 45 rotate 0 0 S6 stroke grestore} bind def
/D7 {gsave translate 45 rotate 0 0 S7 stroke grestore} bind def
/D8 {gsave translate 45 rotate 0 0 S8 stroke grestore} bind def
/D9 {gsave translate 45 rotate 0 0 S9 stroke grestore} bind def
/D10 {gsave translate 45 rotate 0 0 S10 stroke grestore} bind def
/D11 {gsave translate 45 rotate 0 0 S11 stroke grestore} bind def
/D12 {gsave translate 45 rotate 0 0 S12 stroke grestore} bind def
/D13 {gsave translate 45 rotate 0 0 S13 stroke grestore} bind def
/D14 {gsave translate 45 rotate 0 0 S14 stroke grestore} bind def
/D15 {gsave translate 45 rotate 0 0 S15 stroke grestore} bind def
/DiaE {stroke [] 0 setdash vpt add M
  hpt neg vpt neg V hpt vpt neg V
  hpt vpt V hpt neg vpt V closepath stroke} def
/BoxE {stroke [] 0 setdash exch hpt sub exch vpt add M
  0 vpt2 neg V hpt2 0 V 0 vpt2 V
  hpt2 neg 0 V closepath stroke} def
/TriUE {stroke [] 0 setdash vpt 1.12 mul add M
  hpt neg vpt -1.62 mul V
  hpt 2 mul 0 V
  hpt neg vpt 1.62 mul V closepath stroke} def
/TriDE {stroke [] 0 setdash vpt 1.12 mul sub M
  hpt neg vpt 1.62 mul V
  hpt 2 mul 0 V
  hpt neg vpt -1.62 mul V closepath stroke} def
/PentE {stroke [] 0 setdash gsave
  translate 0 hpt M 4 {72 rotate 0 hpt L} repeat
  closepath stroke grestore} def
/CircE {stroke [] 0 setdash 
  hpt 0 360 arc stroke} def
/Opaque {gsave closepath 1 setgray fill grestore 0 setgray closepath} def
/DiaW {stroke [] 0 setdash vpt add M
  hpt neg vpt neg V hpt vpt neg V
  hpt vpt V hpt neg vpt V Opaque stroke} def
/BoxW {stroke [] 0 setdash exch hpt sub exch vpt add M
  0 vpt2 neg V hpt2 0 V 0 vpt2 V
  hpt2 neg 0 V Opaque stroke} def
/TriUW {stroke [] 0 setdash vpt 1.12 mul add M
  hpt neg vpt -1.62 mul V
  hpt 2 mul 0 V
  hpt neg vpt 1.62 mul V Opaque stroke} def
/TriDW {stroke [] 0 setdash vpt 1.12 mul sub M
  hpt neg vpt 1.62 mul V
  hpt 2 mul 0 V
  hpt neg vpt -1.62 mul V Opaque stroke} def
/PentW {stroke [] 0 setdash gsave
  translate 0 hpt M 4 {72 rotate 0 hpt L} repeat
  Opaque stroke grestore} def
/CircW {stroke [] 0 setdash 
  hpt 0 360 arc Opaque stroke} def
/BoxFill {gsave Rec 1 setgray fill grestore} def
/Density {
  /Fillden exch def
  currentrgbcolor
  /ColB exch def /ColG exch def /ColR exch def
  /ColR ColR Fillden mul Fillden sub 1 add def
  /ColG ColG Fillden mul Fillden sub 1 add def
  /ColB ColB Fillden mul Fillden sub 1 add def
  ColR ColG ColB setrgbcolor} def
/BoxColFill {gsave Rec PolyFill} def
/PolyFill {gsave Density fill grestore grestore} def
/h {rlineto rlineto rlineto gsave fill grestore} bind def
%
% PostScript Level 1 Pattern Fill routine for rectangles
% Usage: x y w h s a XX PatternFill
%	x,y = lower left corner of box to be filled
%	w,h = width and height of box
%	  a = angle in degrees between lines and x-axis
%	 XX = 0/1 for no/yes cross-hatch
%
/PatternFill {gsave /PFa [ 9 2 roll ] def
  PFa 0 get PFa 2 get 2 div add PFa 1 get PFa 3 get 2 div add translate
  PFa 2 get -2 div PFa 3 get -2 div PFa 2 get PFa 3 get Rec
  gsave 1 setgray fill grestore clip
  currentlinewidth 0.5 mul setlinewidth
  /PFs PFa 2 get dup mul PFa 3 get dup mul add sqrt def
  0 0 M PFa 5 get rotate PFs -2 div dup translate
  0 1 PFs PFa 4 get div 1 add floor cvi
	{PFa 4 get mul 0 M 0 PFs V} for
  0 PFa 6 get ne {
	0 1 PFs PFa 4 get div 1 add floor cvi
	{PFa 4 get mul 0 2 1 roll M PFs 0 V} for
 } if
  stroke grestore} def
/languagelevel where
 {pop languagelevel} {1} ifelse
 2 lt
	{/InterpretLevel1 true def}
	{/InterpretLevel1 Level1 def}
 ifelse
%
% PostScript level 2 pattern fill definitions
%
/Level2PatternFill {
/Tile8x8 {/PaintType 2 /PatternType 1 /TilingType 1 /BBox [0 0 8 8] /XStep 8 /YStep 8}
	bind def
/KeepColor {currentrgbcolor [/Pattern /DeviceRGB] setcolorspace} bind def
<< Tile8x8
 /PaintProc {0.5 setlinewidth pop 0 0 M 8 8 L 0 8 M 8 0 L stroke} 
>> matrix makepattern
/Pat1 exch def
<< Tile8x8
 /PaintProc {0.5 setlinewidth pop 0 0 M 8 8 L 0 8 M 8 0 L stroke
	0 4 M 4 8 L 8 4 L 4 0 L 0 4 L stroke}
>> matrix makepattern
/Pat2 exch def
<< Tile8x8
 /PaintProc {0.5 setlinewidth pop 0 0 M 0 8 L
	8 8 L 8 0 L 0 0 L fill}
>> matrix makepattern
/Pat3 exch def
<< Tile8x8
 /PaintProc {0.5 setlinewidth pop -4 8 M 8 -4 L
	0 12 M 12 0 L stroke}
>> matrix makepattern
/Pat4 exch def
<< Tile8x8
 /PaintProc {0.5 setlinewidth pop -4 0 M 8 12 L
	0 -4 M 12 8 L stroke}
>> matrix makepattern
/Pat5 exch def
<< Tile8x8
 /PaintProc {0.5 setlinewidth pop -2 8 M 4 -4 L
	0 12 M 8 -4 L 4 12 M 10 0 L stroke}
>> matrix makepattern
/Pat6 exch def
<< Tile8x8
 /PaintProc {0.5 setlinewidth pop -2 0 M 4 12 L
	0 -4 M 8 12 L 4 -4 M 10 8 L stroke}
>> matrix makepattern
/Pat7 exch def
<< Tile8x8
 /PaintProc {0.5 setlinewidth pop 8 -2 M -4 4 L
	12 0 M -4 8 L 12 4 M 0 10 L stroke}
>> matrix makepattern
/Pat8 exch def
<< Tile8x8
 /PaintProc {0.5 setlinewidth pop 0 -2 M 12 4 L
	-4 0 M 12 8 L -4 4 M 8 10 L stroke}
>> matrix makepattern
/Pat9 exch def
/Pattern1 {PatternBgnd KeepColor Pat1 setpattern} bind def
/Pattern2 {PatternBgnd KeepColor Pat2 setpattern} bind def
/Pattern3 {PatternBgnd KeepColor Pat3 setpattern} bind def
/Pattern4 {PatternBgnd KeepColor Landscape {Pat5} {Pat4} ifelse setpattern} bind def
/Pattern5 {PatternBgnd KeepColor Landscape {Pat4} {Pat5} ifelse setpattern} bind def
/Pattern6 {PatternBgnd KeepColor Landscape {Pat9} {Pat6} ifelse setpattern} bind def
/Pattern7 {PatternBgnd KeepColor Landscape {Pat8} {Pat7} ifelse setpattern} bind def
} def
%
%
%End of PostScript Level 2 code
%
/PatternBgnd {
  TransparentPatterns {} {gsave 1 setgray fill grestore} ifelse
} def
%
% Substitute for Level 2 pattern fill codes with
% grayscale if Level 2 support is not selected.
%
/Level1PatternFill {
/Pattern1 {0.250 Density} bind def
/Pattern2 {0.500 Density} bind def
/Pattern3 {0.750 Density} bind def
/Pattern4 {0.125 Density} bind def
/Pattern5 {0.375 Density} bind def
/Pattern6 {0.625 Density} bind def
/Pattern7 {0.875 Density} bind def
} def
%
% Now test for support of Level 2 code
%
Level1 {Level1PatternFill} {Level2PatternFill} ifelse
/Symbol-Oblique /Symbol findfont [1 0 .167 1 0 0] makefont
dup length dict begin {1 index /FID eq {pop pop} {def} ifelse} forall
currentdict end definefont pop
end
gnudict begin
gsave
0 0 translate
0.050 0.050 scale
0 setgray
newpath
1.000 UL
LTb
820 400 M
63 0 V
5977 0 R
-63 0 V
820 1133 M
63 0 V
5977 0 R
-63 0 V
820 1867 M
63 0 V
5977 0 R
-63 0 V
820 2600 M
63 0 V
5977 0 R
-63 0 V
820 3333 M
63 0 V
5977 0 R
-63 0 V
820 4067 M
63 0 V
5977 0 R
-63 0 V
820 4800 M
63 0 V
5977 0 R
-63 0 V
820 400 M
0 63 V
0 4337 R
0 -63 V
2028 400 M
0 63 V
0 4337 R
0 -63 V
3236 400 M
0 63 V
0 4337 R
0 -63 V
4444 400 M
0 63 V
0 4337 R
0 -63 V
5652 400 M
0 63 V
0 4337 R
0 -63 V
6860 400 M
0 63 V
0 4337 R
0 -63 V
820 4800 M
820 400 L
6040 0 V
0 4400 V
-6040 0 V
1.000 UP
stroke
LT0
820 400 M
6 15 V
6 -8 V
6 15 V
6 15 V
6 14 V
6 -7 V
6 -7 V
6 -8 V
6 -7 V
6 15 V
6 14 V
6 -7 V
7 15 V
6 -8 V
6 -7 V
6 -7 V
6 -8 V
6 15 V
6 15 V
6 -8 V
6 -7 V
6 15 V
6 14 V
6 15 V
6 -7 V
6 14 V
6 15 V
6 15 V
6 14 V
6 15 V
6 15 V
6 14 V
6 -7 V
6 15 V
6 14 V
6 15 V
6 -7 V
7 -8 V
6 -7 V
6 -7 V
6 14 V
6 15 V
6 -7 V
6 14 V
6 15 V
6 15 V
6 -8 V
6 -7 V
6 -7 V
6 14 V
6 15 V
6 -7 V
6 -8 V
6 -7 V
6 -7 V
6 -8 V
6 15 V
6 -7 V
6 14 V
6 15 V
6 15 V
6 14 V
7 -7 V
6 -7 V
6 14 V
6 -7 V
6 15 V
6 -8 V
6 -7 V
6 -7 V
6 -8 V
6 -7 V
6 15 V
6 14 V
6 15 V
6 15 V
6 14 V
6 15 V
6 15 V
6 14 V
6 -7 V
6 -7 V
6 14 V
6 -7 V
6 -7 V
6 -8 V
6 15 V
7 -7 V
6 14 V
6 -7 V
6 15 V
6 -8 V
6 -7 V
6 15 V
6 -8 V
6 15 V
6 15 V
6 14 V
6 15 V
6 -7 V
6 -8 V
6 -7 V
6 15 V
6 -8 V
1454 818 L
6 -7 V
6 -8 V
6 15 V
6 15 V
6 14 V
6 -7 V
6 -7 V
7 -8 V
6 15 V
6 -7 V
6 -8 V
6 -7 V
6 15 V
6 -8 V
6 -7 V
6 15 V
6 -8 V
6 15 V
6 -7 V
6 14 V
6 15 V
6 -7 V
6 14 V
6 -7 V
6 -7 V
6 -8 V
6 -7 V
6 15 V
6 14 V
6 15 V
6 -7 V
6 14 V
7 -7 V
6 -7 V
6 14 V
6 15 V
6 15 V
6 -8 V
6 15 V
6 15 V
6 -8 V
6 15 V
6 15 V
6 -8 V
6 -7 V
6 15 V
6 -8 V
6 -7 V
6 -7 V
6 14 V
6 15 V
6 -7 V
6 -8 V
6 15 V
6 -7 V
6 14 V
6 -7 V
7 -7 V
6 -8 V
6 15 V
6 -7 V
6 -8 V
6 15 V
6 -7 V
6 -8 V
6 15 V
6 -7 V
6 -8 V
6 15 V
6 15 V
6 -8 V
6 15 V
6 15 V
6 14 V
6 -7 V
6 -7 V
6 14 V
6 15 V
6 -7 V
6 14 V
6 15 V
6 15 V
7 14 V
6 -7 V
6 -7 V
6 -8 V
6 15 V
6 15 V
6 14 V
6 15 V
6 15 V
6 14 V
6 15 V
6 15 V
6 14 V
6 15 V
6 -7 V
6 -8 V
6 15 V
6 -7 V
6 -8 V
6 15 V
6 15 V
6 14 V
2088 1236 L
6 -7 V
6 -8 V
7 -7 V
6 -7 V
6 -8 V
6 15 V
6 15 V
6 -8 V
6 15 V
6 -7 V
6 -8 V
6 15 V
6 -7 V
6 14 V
6 15 V
6 -7 V
6 -8 V
6 -7 V
6 -7 V
6 14 V
6 -7 V
6 -7 V
6 -8 V
6 -7 V
6 -7 V
6 14 V
6 15 V
7 15 V
6 14 V
6 15 V
6 15 V
6 14 V
6 -7 V
6 15 V
6 14 V
6 -7 V
6 -7 V
6 14 V
6 -7 V
6 -7 V
6 -8 V
6 15 V
6 -7 V
6 14 V
6 15 V
6 -7 V
6 14 V
6 -7 V
6 15 V
6 -8 V
6 15 V
6 15 V
7 14 V
6 15 V
6 15 V
6 14 V
6 15 V
6 15 V
6 -8 V
6 -7 V
6 -7 V
6 -8 V
6 -7 V
6 15 V
6 -8 V
6 -7 V
6 15 V
6 -8 V
6 -7 V
6 15 V
6 14 V
6 -7 V
6 15 V
6 14 V
6 -7 V
6 15 V
6 -8 V
7 -7 V
6 15 V
6 -8 V
6 -7 V
6 15 V
6 -8 V
6 15 V
6 -7 V
6 -8 V
6 -7 V
6 -7 V
6 14 V
6 -7 V
6 -7 V
6 14 V
6 -7 V
6 15 V
6 -8 V
6 -7 V
6 15 V
6 -8 V
6 15 V
6 15 V
6 14 V
6 -7 V
7 15 V
6 14 V
2723 1544 L
6 15 V
6 -8 V
6 15 V
6 -7 V
6 -8 V
6 -7 V
6 -7 V
6 14 V
6 -7 V
6 15 V
6 -8 V
6 15 V
6 -7 V
6 14 V
6 15 V
6 -7 V
6 -8 V
6 -7 V
6 15 V
6 -8 V
6 15 V
6 -7 V
7 -8 V
6 15 V
6 -7 V
6 -8 V
6 15 V
6 -7 V
6 14 V
6 -7 V
6 15 V
6 -8 V
6 15 V
6 15 V
6 -8 V
6 15 V
6 -7 V
6 14 V
6 -7 V
6 15 V
6 14 V
6 15 V
6 -7 V
6 -8 V
6 15 V
6 15 V
6 -8 V
7 15 V
6 15 V
6 14 V
6 -7 V
6 -7 V
6 -8 V
6 15 V
6 -7 V
6 14 V
6 -7 V
6 -7 V
6 14 V
6 15 V
6 -7 V
6 -8 V
6 15 V
6 15 V
6 14 V
6 -7 V
6 15 V
6 -8 V
6 -7 V
6 -7 V
6 14 V
6 15 V
7 -7 V
6 14 V
6 15 V
6 15 V
6 14 V
6 -7 V
6 15 V
6 14 V
6 -7 V
6 15 V
6 14 V
6 -7 V
6 15 V
6 14 V
6 -7 V
6 -7 V
6 14 V
6 15 V
6 -7 V
6 14 V
6 15 V
6 -7 V
6 -8 V
6 -7 V
6 -7 V
7 14 V
6 15 V
6 -7 V
6 -8 V
6 15 V
6 15 V
6 14 V
3357 1984 L
6 15 V
6 -8 V
6 15 V
6 -7 V
6 -8 V
6 15 V
6 15 V
6 14 V
6 -7 V
6 15 V
6 -8 V
6 -7 V
6 -7 V
6 -8 V
6 -7 V
6 -7 V
6 14 V
7 -7 V
6 15 V
6 -8 V
6 -7 V
6 15 V
6 14 V
6 15 V
6 -7 V
6 -8 V
6 15 V
6 15 V
6 -8 V
6 15 V
6 15 V
6 14 V
6 15 V
6 -7 V
6 14 V
6 15 V
6 -7 V
6 14 V
6 -7 V
6 -7 V
6 14 V
6 15 V
7 15 V
6 -8 V
6 -7 V
6 15 V
6 -8 V
6 15 V
6 -7 V
6 14 V
6 -7 V
6 15 V
6 -8 V
6 15 V
6 15 V
6 -8 V
6 15 V
6 15 V
6 14 V
6 15 V
6 -7 V
6 14 V
6 -7 V
6 -7 V
6 -8 V
6 15 V
6 15 V
7 -8 V
6 -7 V
6 -7 V
6 -8 V
6 -7 V
6 -7 V
6 14 V
6 15 V
6 15 V
6 -8 V
6 -7 V
6 -7 V
6 -8 V
6 15 V
6 15 V
6 14 V
6 15 V
6 -7 V
6 -8 V
6 15 V
6 15 V
6 -8 V
6 15 V
6 -7 V
6 14 V
7 -7 V
6 -7 V
6 -8 V
6 15 V
6 15 V
6 -8 V
6 15 V
6 15 V
6 14 V
6 15 V
6 -7 V
6 14 V
3991 2402 L
6 15 V
6 -8 V
6 15 V
6 15 V
6 14 V
6 -7 V
6 15 V
6 -8 V
6 -7 V
6 15 V
6 -8 V
6 -7 V
7 -7 V
6 -8 V
6 -7 V
6 15 V
6 14 V
6 -7 V
6 15 V
6 14 V
6 -7 V
6 -7 V
6 -8 V
6 -7 V
6 15 V
6 -8 V
6 15 V
6 15 V
6 14 V
6 15 V
6 15 V
6 -8 V
6 -7 V
6 15 V
6 -8 V
6 15 V
6 15 V
7 -8 V
6 15 V
6 -7 V
6 -8 V
6 -7 V
6 -7 V
6 14 V
6 -7 V
6 15 V
6 -8 V
6 -7 V
6 -7 V
6 14 V
6 15 V
6 15 V
6 -8 V
6 15 V
6 -7 V
6 -8 V
6 -7 V
6 -7 V
6 14 V
6 -7 V
6 15 V
6 14 V
7 15 V
6 -7 V
6 -8 V
6 15 V
6 -7 V
6 14 V
6 -7 V
6 15 V
6 -8 V
6 -7 V
6 15 V
6 -8 V
6 -7 V
6 -7 V
6 -8 V
6 -7 V
6 -7 V
6 14 V
6 -7 V
6 15 V
6 -8 V
6 15 V
6 -7 V
6 14 V
6 15 V
7 15 V
6 -8 V
6 15 V
6 15 V
6 -8 V
6 15 V
6 -7 V
6 -8 V
6 -7 V
6 -7 V
6 14 V
6 15 V
6 15 V
6 -8 V
6 -7 V
6 15 V
6 14 V
4625 2710 L
6 -7 V
6 -8 V
6 -7 V
6 -7 V
6 -8 V
6 -7 V
6 -7 V
7 -8 V
6 -7 V
6 15 V
6 -8 V
6 15 V
6 15 V
6 -8 V
6 -7 V
6 15 V
6 -8 V
6 -7 V
6 15 V
6 14 V
6 15 V
6 15 V
6 -8 V
6 -7 V
6 15 V
6 14 V
6 -7 V
6 15 V
6 -8 V
6 15 V
6 -7 V
6 14 V
7 -7 V
6 15 V
6 -8 V
6 -7 V
6 -7 V
6 14 V
6 15 V
6 -7 V
6 -8 V
6 15 V
6 -7 V
6 -8 V
6 -7 V
6 -7 V
6 14 V
6 -7 V
6 -7 V
6 14 V
6 15 V
6 -7 V
6 14 V
6 -7 V
6 15 V
6 -8 V
6 15 V
7 -7 V
6 -8 V
6 -7 V
6 -7 V
6 14 V
6 15 V
6 15 V
6 14 V
6 15 V
6 -7 V
6 14 V
6 -7 V
6 -7 V
6 -8 V
6 15 V
6 15 V
6 14 V
6 -7 V
6 -7 V
6 14 V
6 -7 V
6 15 V
6 14 V
6 -7 V
6 15 V
7 -8 V
6 15 V
6 15 V
6 -8 V
6 -7 V
6 15 V
6 -8 V
6 15 V
6 15 V
6 14 V
6 15 V
6 -7 V
6 14 V
6 15 V
6 15 V
6 -8 V
6 15 V
6 15 V
6 14 V
6 -7 V
6 -7 V
6 14 V
5259 3040 L
6 15 V
6 -8 V
7 -7 V
6 -7 V
6 14 V
6 -7 V
6 -7 V
6 14 V
6 15 V
6 15 V
6 14 V
6 -7 V
6 15 V
6 -8 V
6 -7 V
6 -7 V
6 -8 V
6 -7 V
6 15 V
6 -8 V
6 15 V
6 -7 V
6 14 V
6 -7 V
6 15 V
6 14 V
6 -7 V
7 15 V
6 -8 V
6 -7 V
6 15 V
6 14 V
6 -7 V
6 15 V
6 -8 V
6 -7 V
6 -7 V
6 -8 V
6 15 V
6 -7 V
6 -8 V
6 -7 V
6 15 V
6 -8 V
6 15 V
6 -7 V
6 14 V
6 -7 V
6 -7 V
6 -8 V
6 -7 V
6 15 V
7 14 V
6 15 V
6 -7 V
6 14 V
6 15 V
6 -7 V
6 14 V
6 -7 V
6 15 V
6 -8 V
6 -7 V
6 -7 V
6 14 V
6 -7 V
6 -7 V
6 14 V
6 15 V
6 15 V
6 -8 V
6 15 V
6 -7 V
6 -8 V
6 -7 V
6 -7 V
6 14 V
7 -7 V
6 15 V
6 -8 V
6 -7 V
6 15 V
6 -8 V
6 -7 V
6 15 V
6 -8 V
6 -7 V
6 -7 V
6 -8 V
6 15 V
6 15 V
6 14 V
6 15 V
6 -7 V
6 -8 V
6 15 V
6 -7 V
6 14 V
6 -7 V
6 -7 V
6 -8 V
6 15 V
7 15 V
6 -8 V
5894 3260 L
6 15 V
6 14 V
6 15 V
6 -7 V
6 -8 V
6 -7 V
6 -7 V
6 -8 V
6 -7 V
6 -7 V
6 -8 V
6 -7 V
6 -7 V
6 14 V
6 -7 V
6 -7 V
6 -8 V
6 15 V
6 -7 V
6 -8 V
6 -7 V
6 -7 V
7 14 V
6 15 V
6 15 V
6 14 V
6 -7 V
6 15 V
6 14 V
6 -7 V
6 15 V
6 -8 V
6 15 V
6 15 V
6 -8 V
6 -7 V
6 -7 V
6 -8 V
6 -7 V
6 15 V
6 -8 V
6 -7 V
6 -7 V
6 14 V
6 15 V
6 15 V
6 -8 V
7 15 V
6 15 V
6 14 V
6 -7 V
6 15 V
6 -8 V
6 15 V
6 15 V
6 14 V
6 15 V
6 15 V
6 -8 V
6 15 V
6 -7 V
6 14 V
6 15 V
6 15 V
6 14 V
6 15 V
6 -7 V
6 14 V
6 -7 V
6 -7 V
6 -8 V
6 -7 V
7 -7 V
6 -8 V
6 15 V
6 15 V
6 -8 V
6 15 V
6 -7 V
6 -8 V
6 15 V
6 15 V
6 -8 V
6 15 V
6 -7 V
6 14 V
6 15 V
6 15 V
6 -8 V
6 -7 V
6 -7 V
6 -8 V
6 15 V
6 -7 V
6 -8 V
6 -7 V
6 15 V
7 -8 V
6 -7 V
6 15 V
6 14 V
6 -7 V
6 15 V
6 -8 V
6528 3546 L
6 -7 V
6 14 V
6 15 V
6 -7 V
6 -8 V
6 15 V
6 15 V
6 -8 V
6 15 V
6 15 V
6 -8 V
6 -7 V
6 15 V
6 -8 V
6 -7 V
6 -7 V
6 14 V
7 15 V
6 15 V
6 14 V
6 -7 V
6 -7 V
6 -8 V
6 -7 V
6 15 V
6 -8 V
6 -7 V
6 -7 V
6 14 V
6 -7 V
6 15 V
6 -8 V
6 15 V
6 15 V
6 14 V
6 -7 V
6 15 V
6 -8 V
6 15 V
6 -7 V
6 -8 V
6 15 V
7 15 V
6 14 V
6 -7 V
6 -7 V
6 14 V
6 15 V
6 15 V
6 14 V
6 15 V
6 15 V
6 -8 V
6 15 V
6 15 V
stroke
LT1
820 400 M
6 15 V
6 14 V
6 -7 V
6 15 V
6 -8 V
6 15 V
6 15 V
6 14 V
6 15 V
6 -7 V
6 -8 V
6 15 V
7 15 V
6 14 V
6 15 V
6 -7 V
6 -8 V
6 -7 V
6 -7 V
6 -8 V
6 15 V
6 15 V
6 -8 V
6 -7 V
6 15 V
6 -8 V
6 15 V
6 15 V
6 -8 V
6 15 V
6 -7 V
6 -8 V
6 -7 V
6 15 V
6 -8 V
6 15 V
6 15 V
7 -8 V
6 15 V
6 -7 V
6 -8 V
6 15 V
6 -7 V
6 -8 V
6 -7 V
6 -7 V
6 14 V
6 15 V
6 -7 V
6 14 V
6 15 V
6 15 V
6 14 V
6 15 V
6 15 V
6 -8 V
6 15 V
6 -7 V
6 -8 V
6 15 V
6 -7 V
6 14 V
7 15 V
6 15 V
6 -8 V
6 15 V
6 -7 V
6 14 V
6 -7 V
6 -7 V
6 -8 V
6 -7 V
6 -7 V
6 14 V
6 15 V
6 -7 V
6 14 V
6 15 V
6 15 V
6 -8 V
6 15 V
6 -7 V
6 -8 V
6 15 V
6 15 V
6 14 V
6 -7 V
7 -7 V
6 -8 V
6 -7 V
6 -7 V
6 -8 V
6 15 V
6 -7 V
6 -8 V
6 -7 V
6 15 V
6 -8 V
6 15 V
6 15 V
6 14 V
6 15 V
6 -7 V
6 -8 V
1454 796 L
6 15 V
6 -8 V
6 -7 V
6 15 V
6 -8 V
6 15 V
6 -7 V
7 14 V
6 15 V
6 -7 V
6 -8 V
6 -7 V
6 15 V
6 -8 V
6 15 V
6 15 V
6 -8 V
6 15 V
6 -7 V
6 -8 V
6 -7 V
6 -7 V
6 -8 V
6 -7 V
6 15 V
6 14 V
6 15 V
6 15 V
6 -8 V
6 -7 V
6 15 V
6 14 V
7 -7 V
6 15 V
6 14 V
6 -7 V
6 15 V
6 14 V
6 -7 V
6 -7 V
6 -8 V
6 -7 V
6 15 V
6 14 V
6 15 V
6 -7 V
6 14 V
6 15 V
6 -7 V
6 -8 V
6 15 V
6 15 V
6 14 V
6 -7 V
6 15 V
6 -8 V
6 15 V
7 15 V
6 14 V
6 15 V
6 15 V
6 -8 V
6 15 V
6 -7 V
6 14 V
6 -7 V
6 15 V
6 -8 V
6 15 V
6 -7 V
6 14 V
6 -7 V
6 -7 V
6 14 V
6 15 V
6 -7 V
6 -8 V
6 -7 V
6 -7 V
6 -8 V
6 15 V
6 -7 V
7 14 V
6 15 V
6 -7 V
6 -8 V
6 -7 V
6 15 V
6 -8 V
6 -7 V
6 15 V
6 14 V
6 15 V
6 15 V
6 14 V
6 15 V
6 -7 V
6 -8 V
6 -7 V
6 15 V
6 14 V
6 -7 V
6 15 V
6 14 V
2088 1236 L
6 -7 V
6 14 V
7 15 V
6 -7 V
6 -8 V
6 -7 V
6 -7 V
6 -8 V
6 15 V
6 -7 V
6 14 V
6 15 V
6 15 V
6 14 V
6 15 V
6 15 V
6 14 V
6 -7 V
6 -7 V
6 -8 V
6 -7 V
6 15 V
6 14 V
6 -7 V
6 15 V
6 -8 V
6 15 V
7 15 V
6 -8 V
6 -7 V
6 15 V
6 -8 V
6 15 V
6 15 V
6 -8 V
6 -7 V
6 15 V
6 -8 V
6 15 V
6 -7 V
6 -8 V
6 -7 V
6 -7 V
6 -8 V
6 15 V
6 15 V
6 14 V
6 15 V
6 -7 V
6 -8 V
6 15 V
6 -7 V
7 -8 V
6 15 V
6 15 V
6 14 V
6 -7 V
6 -7 V
6 14 V
6 15 V
6 15 V
6 14 V
6 15 V
6 15 V
6 14 V
6 -7 V
6 15 V
6 14 V
6 -7 V
6 -7 V
6 14 V
6 15 V
6 -7 V
6 -8 V
6 15 V
6 -7 V
6 14 V
7 15 V
6 15 V
6 -8 V
6 15 V
6 -7 V
6 14 V
6 -7 V
6 15 V
6 14 V
6 15 V
6 -7 V
6 14 V
6 15 V
6 -7 V
6 -8 V
6 -7 V
6 15 V
6 -8 V
6 15 V
6 15 V
6 14 V
6 -7 V
6 -7 V
6 -8 V
6 15 V
7 15 V
6 14 V
2723 1742 L
6 -7 V
6 -8 V
6 15 V
6 -7 V
6 14 V
6 15 V
6 15 V
6 14 V
6 15 V
6 -7 V
6 14 V
6 -7 V
6 -7 V
6 -8 V
6 -7 V
6 15 V
6 -8 V
6 15 V
6 -7 V
6 14 V
6 15 V
6 -7 V
7 14 V
6 -7 V
6 15 V
6 14 V
6 -7 V
6 -7 V
6 -8 V
6 15 V
6 -7 V
6 14 V
6 -7 V
6 -7 V
6 -8 V
6 15 V
6 -7 V
6 14 V
6 15 V
6 -7 V
6 14 V
6 15 V
6 15 V
6 -8 V
6 15 V
6 15 V
6 14 V
7 15 V
6 15 V
6 14 V
6 15 V
6 15 V
6 -8 V
6 -7 V
6 15 V
6 14 V
6 15 V
6 15 V
6 -8 V
6 -7 V
6 15 V
6 14 V
6 15 V
6 -7 V
6 14 V
6 15 V
6 -7 V
6 14 V
6 15 V
6 -7 V
6 14 V
6 -7 V
7 15 V
6 -8 V
6 -7 V
6 -7 V
6 -8 V
6 15 V
6 15 V
6 -8 V
6 -7 V
6 15 V
6 -8 V
6 15 V
6 -7 V
6 -8 V
6 15 V
6 -7 V
6 14 V
6 -7 V
6 15 V
6 -8 V
6 15 V
6 -7 V
6 14 V
6 -7 V
6 15 V
7 -8 V
6 -7 V
6 -7 V
6 -8 V
6 -7 V
6 15 V
6 14 V
3357 2204 L
6 15 V
6 -8 V
6 15 V
6 -7 V
6 -8 V
6 -7 V
6 -7 V
6 -8 V
6 15 V
6 -7 V
6 14 V
6 15 V
6 15 V
6 -8 V
6 15 V
6 -7 V
6 -8 V
7 -7 V
6 -7 V
6 -8 V
6 -7 V
6 -7 V
6 -8 V
6 15 V
6 -7 V
6 14 V
6 15 V
6 15 V
6 -8 V
6 15 V
6 15 V
6 14 V
6 -7 V
6 -7 V
6 14 V
6 -7 V
6 -7 V
6 14 V
6 -7 V
6 -7 V
6 14 V
6 -7 V
7 15 V
6 -8 V
6 15 V
6 15 V
6 -8 V
6 -7 V
6 15 V
6 14 V
6 15 V
6 -7 V
6 -8 V
6 -7 V
6 15 V
6 14 V
6 15 V
6 15 V
6 -8 V
6 -7 V
6 15 V
6 14 V
6 -7 V
6 -7 V
6 14 V
6 -7 V
6 15 V
7 -8 V
6 -7 V
6 -7 V
6 -8 V
6 -7 V
6 15 V
6 -8 V
6 15 V
6 -7 V
6 14 V
6 -7 V
6 15 V
6 14 V
6 -7 V
6 15 V
6 14 V
6 15 V
6 -7 V
6 -8 V
6 -7 V
6 -7 V
6 -8 V
6 -7 V
6 15 V
6 -8 V
7 15 V
6 -7 V
6 -8 V
6 15 V
6 -7 V
6 -8 V
6 -7 V
6 -7 V
6 -8 V
6 15 V
6 15 V
6 -8 V
3991 2402 L
6 -7 V
6 -8 V
6 15 V
6 15 V
6 -8 V
6 -7 V
6 15 V
6 14 V
6 -7 V
6 -7 V
6 14 V
6 15 V
7 15 V
6 -8 V
6 15 V
6 15 V
6 14 V
6 -7 V
6 15 V
6 14 V
6 15 V
6 -7 V
6 -8 V
6 -7 V
6 -7 V
6 -8 V
6 -7 V
6 15 V
6 14 V
6 -7 V
6 15 V
6 -8 V
6 -7 V
6 -7 V
6 14 V
6 15 V
6 -7 V
7 -8 V
6 15 V
6 15 V
6 14 V
6 -7 V
6 15 V
6 -8 V
6 -7 V
6 15 V
6 -8 V
6 -7 V
6 15 V
6 -8 V
6 15 V
6 15 V
6 14 V
6 15 V
6 15 V
6 14 V
6 15 V
6 -7 V
6 -8 V
6 15 V
6 15 V
6 -8 V
7 15 V
6 -7 V
6 14 V
6 15 V
6 15 V
6 14 V
6 15 V
6 15 V
6 14 V
6 -7 V
6 -7 V
6 -8 V
6 15 V
6 15 V
6 -8 V
6 -7 V
6 15 V
6 -8 V
6 -7 V
6 15 V
6 -8 V
6 -7 V
6 15 V
6 -8 V
6 15 V
7 15 V
6 14 V
6 -7 V
6 -7 V
6 14 V
6 -7 V
6 -7 V
6 14 V
6 -7 V
6 15 V
6 -8 V
6 15 V
6 -7 V
6 -8 V
6 -7 V
6 15 V
6 -8 V
4625 2820 L
6 -7 V
6 14 V
6 -7 V
6 15 V
6 14 V
6 15 V
6 -7 V
7 14 V
6 -7 V
6 -7 V
6 -8 V
6 15 V
6 15 V
6 14 V
6 15 V
6 15 V
6 14 V
6 -7 V
6 -7 V
6 -8 V
6 15 V
6 -7 V
6 14 V
6 -7 V
6 -7 V
6 -8 V
6 15 V
6 15 V
6 14 V
6 15 V
6 -7 V
6 -8 V
7 15 V
6 15 V
6 14 V
6 -7 V
6 -7 V
6 -8 V
6 -7 V
6 15 V
6 14 V
6 15 V
6 15 V
6 14 V
6 15 V
6 -7 V
6 14 V
6 15 V
6 -7 V
6 14 V
6 15 V
6 -7 V
6 14 V
6 -7 V
6 15 V
6 -8 V
6 -7 V
7 -7 V
6 -8 V
6 15 V
6 15 V
6 14 V
6 -7 V
6 15 V
6 14 V
6 15 V
6 -7 V
6 14 V
6 -7 V
6 15 V
6 -8 V
6 15 V
6 15 V
6 14 V
6 -7 V
6 15 V
6 14 V
6 15 V
6 15 V
6 14 V
6 15 V
6 15 V
7 14 V
6 15 V
6 -7 V
6 14 V
6 15 V
6 15 V
6 -8 V
6 -7 V
6 15 V
6 -8 V
6 -7 V
6 -7 V
6 14 V
6 -7 V
6 15 V
6 14 V
6 -7 V
6 15 V
6 14 V
6 15 V
6 -7 V
6 14 V
5259 3436 L
6 -7 V
6 -8 V
7 15 V
6 -7 V
6 14 V
6 -7 V
6 15 V
6 -8 V
6 -7 V
6 15 V
6 14 V
6 15 V
6 15 V
6 14 V
6 -7 V
6 15 V
6 14 V
6 15 V
6 -7 V
6 14 V
6 15 V
6 15 V
6 -8 V
6 -7 V
6 -7 V
6 -8 V
6 -7 V
7 15 V
6 14 V
6 -7 V
6 -7 V
6 -8 V
6 -7 V
6 -7 V
6 14 V
6 15 V
6 15 V
6 -8 V
6 15 V
6 15 V
6 -8 V
6 15 V
6 -7 V
6 -8 V
6 15 V
6 -7 V
6 14 V
6 -7 V
6 15 V
6 -8 V
6 -7 V
6 -7 V
7 -8 V
6 15 V
6 -7 V
6 -8 V
6 15 V
6 -7 V
6 -8 V
6 15 V
6 15 V
6 14 V
6 -7 V
6 -7 V
6 -8 V
6 15 V
6 -7 V
6 14 V
6 15 V
6 -7 V
6 -8 V
6 15 V
6 -7 V
6 -8 V
6 15 V
6 15 V
6 14 V
7 15 V
6 -7 V
6 -8 V
6 15 V
6 15 V
6 14 V
6 15 V
6 -7 V
6 14 V
6 -7 V
6 -7 V
6 14 V
6 15 V
6 15 V
6 -8 V
6 15 V
6 -7 V
6 -8 V
6 -7 V
6 15 V
6 14 V
6 15 V
6 -7 V
6 -8 V
6 15 V
7 15 V
6 -8 V
5894 3810 L
6 -7 V
6 14 V
6 -7 V
6 -7 V
6 14 V
6 -7 V
6 15 V
6 14 V
6 15 V
6 -7 V
6 14 V
6 15 V
6 -7 V
6 -8 V
6 -7 V
6 -7 V
6 -8 V
6 -7 V
6 15 V
6 14 V
6 -7 V
6 15 V
7 -8 V
6 -7 V
6 -7 V
6 -8 V
6 15 V
6 15 V
6 -8 V
6 15 V
6 -7 V
6 -8 V
6 15 V
6 -7 V
6 14 V
6 -7 V
6 15 V
6 14 V
6 15 V
6 -7 V
6 -8 V
6 15 V
6 -7 V
6 -8 V
6 -7 V
6 15 V
6 14 V
7 -7 V
6 15 V
6 -8 V
6 15 V
6 -7 V
6 14 V
6 15 V
6 -7 V
6 -8 V
6 -7 V
6 -7 V
6 14 V
6 -7 V
6 15 V
6 -8 V
6 -7 V
6 15 V
6 14 V
6 15 V
6 -7 V
6 -8 V
6 15 V
6 -7 V
6 14 V
6 -7 V
7 -7 V
6 -8 V
6 -7 V
6 -7 V
6 14 V
6 15 V
6 15 V
6 -8 V
6 15 V
6 -7 V
6 14 V
6 15 V
6 -7 V
6 -8 V
6 15 V
6 -7 V
6 14 V
6 -7 V
6 -7 V
6 -8 V
6 15 V
6 -7 V
6 14 V
6 -7 V
6 15 V
7 -8 V
6 15 V
6 -7 V
6 -8 V
6 15 V
6 -7 V
6 -8 V
6528 4052 L
6 -7 V
6 -8 V
6 15 V
6 -7 V
6 -8 V
6 15 V
6 -7 V
6 14 V
6 -7 V
6 -7 V
6 -8 V
6 -7 V
6 15 V
6 14 V
6 -7 V
6 -7 V
6 -8 V
7 15 V
6 -7 V
6 -8 V
6 15 V
6 -7 V
6 -8 V
6 -7 V
6 15 V
6 14 V
6 15 V
6 -7 V
6 14 V
6 15 V
6 15 V
6 -8 V
6 -7 V
6 15 V
6 -8 V
6 15 V
6 -7 V
6 -8 V
6 -7 V
6 15 V
6 14 V
6 -7 V
7 -7 V
6 -8 V
6 15 V
6 15 V
6 14 V
6 15 V
6 -7 V
6 -8 V
6 -7 V
6 15 V
6 14 V
6 15 V
6 15 V
stroke
LT2
820 400 M
6 15 V
6 14 V
6 -7 V
6 15 V
6 -8 V
6 15 V
6 15 V
6 14 V
6 15 V
6 15 V
6 14 V
6 15 V
7 -7 V
6 -8 V
6 15 V
6 15 V
6 14 V
6 15 V
6 15 V
6 14 V
6 15 V
6 15 V
6 14 V
6 -7 V
6 15 V
6 14 V
6 15 V
6 15 V
6 -8 V
6 15 V
6 15 V
6 14 V
6 -7 V
6 15 V
6 14 V
6 15 V
6 -7 V
7 14 V
6 -7 V
6 -7 V
6 14 V
6 15 V
6 -7 V
6 -8 V
6 15 V
6 15 V
6 14 V
6 15 V
6 -7 V
6 14 V
6 15 V
6 -7 V
6 14 V
6 -7 V
6 -7 V
6 -8 V
6 -7 V
6 -7 V
6 -8 V
6 15 V
6 15 V
6 -8 V
7 15 V
6 15 V
6 14 V
6 15 V
6 15 V
6 14 V
6 -7 V
6 -7 V
6 14 V
6 15 V
6 15 V
6 14 V
6 -7 V
6 15 V
6 -8 V
6 15 V
6 15 V
6 -8 V
6 -7 V
6 15 V
6 14 V
6 15 V
6 15 V
6 14 V
6 15 V
7 15 V
6 -8 V
6 -7 V
6 15 V
6 14 V
6 15 V
6 -7 V
6 -8 V
6 -7 V
6 15 V
6 -8 V
6 15 V
6 -7 V
6 -8 V
6 15 V
6 15 V
6 14 V
1454 1148 L
6 15 V
6 -8 V
6 15 V
6 -7 V
6 -8 V
6 15 V
6 -7 V
7 -8 V
6 15 V
6 15 V
6 -8 V
6 -7 V
6 15 V
6 14 V
6 -7 V
6 15 V
6 14 V
6 15 V
6 15 V
6 -8 V
6 -7 V
6 -7 V
6 14 V
6 -7 V
6 -7 V
6 14 V
6 15 V
6 15 V
6 14 V
6 15 V
6 -7 V
6 14 V
7 -7 V
6 15 V
6 -8 V
6 -7 V
6 15 V
6 14 V
6 -7 V
6 -7 V
6 -8 V
6 15 V
6 -7 V
6 -8 V
6 -7 V
6 15 V
6 14 V
6 -7 V
6 15 V
6 14 V
6 -7 V
6 -7 V
6 -8 V
6 15 V
6 -7 V
6 14 V
6 15 V
7 -7 V
6 -8 V
6 -7 V
6 -7 V
6 -8 V
6 15 V
6 -7 V
6 14 V
6 15 V
6 15 V
6 -8 V
6 15 V
6 -7 V
6 -8 V
6 -7 V
6 15 V
6 14 V
6 -7 V
6 15 V
6 -8 V
6 -7 V
6 15 V
6 -8 V
6 15 V
6 15 V
7 14 V
6 -7 V
6 -7 V
6 -8 V
6 -7 V
6 -7 V
6 14 V
6 -7 V
6 -7 V
6 -8 V
6 15 V
6 -7 V
6 -8 V
6 15 V
6 15 V
6 -8 V
6 -7 V
6 -7 V
6 -8 V
6 -7 V
6 -7 V
6 -8 V
2088 1390 L
6 -7 V
6 -8 V
7 -7 V
6 15 V
6 14 V
6 15 V
6 15 V
6 -8 V
6 -7 V
6 15 V
6 -8 V
6 -7 V
6 -7 V
6 -8 V
6 -7 V
6 15 V
6 -8 V
6 15 V
6 -7 V
6 -8 V
6 -7 V
6 -7 V
6 14 V
6 -7 V
6 15 V
6 14 V
6 15 V
7 15 V
6 -8 V
6 -7 V
6 -7 V
6 -8 V
6 15 V
6 15 V
6 -8 V
6 -7 V
6 15 V
6 14 V
6 -7 V
6 -7 V
6 -8 V
6 15 V
6 15 V
6 14 V
6 15 V
6 -7 V
6 14 V
6 -7 V
6 15 V
6 -8 V
6 15 V
6 15 V
7 -8 V
6 15 V
6 -7 V
6 -8 V
6 15 V
6 15 V
6 -8 V
6 15 V
6 15 V
6 -8 V
6 15 V
6 15 V
6 -8 V
6 15 V
6 -7 V
6 -8 V
6 -7 V
6 -7 V
6 -8 V
6 15 V
6 15 V
6 -8 V
6 -7 V
6 -7 V
6 -8 V
7 15 V
6 15 V
6 14 V
6 -7 V
6 15 V
6 -8 V
6 -7 V
6 15 V
6 14 V
6 15 V
6 15 V
6 14 V
6 15 V
6 -7 V
6 -8 V
6 15 V
6 -7 V
6 14 V
6 -7 V
6 -7 V
6 -8 V
6 15 V
6 -7 V
6 -8 V
6 -7 V
7 15 V
6 14 V
2723 1698 L
6 15 V
6 14 V
6 15 V
6 -7 V
6 14 V
6 -7 V
6 15 V
6 14 V
6 15 V
6 -7 V
6 -8 V
6 -7 V
6 15 V
6 14 V
6 15 V
6 -7 V
6 14 V
6 15 V
6 -7 V
6 14 V
6 -7 V
6 15 V
7 -8 V
6 -7 V
6 15 V
6 14 V
6 -7 V
6 -7 V
6 -8 V
6 15 V
6 15 V
6 -8 V
6 15 V
6 15 V
6 14 V
6 15 V
6 15 V
6 -8 V
6 15 V
6 -7 V
6 14 V
6 15 V
6 15 V
6 -8 V
6 15 V
6 15 V
6 14 V
7 15 V
6 15 V
6 -8 V
6 -7 V
6 -7 V
6 -8 V
6 -7 V
6 15 V
6 14 V
6 -7 V
6 15 V
6 -8 V
6 -7 V
6 15 V
6 14 V
6 -7 V
6 15 V
6 -8 V
6 15 V
6 15 V
6 -8 V
6 -7 V
6 -7 V
6 -8 V
6 -7 V
7 15 V
6 14 V
6 15 V
6 15 V
6 -8 V
6 -7 V
6 -7 V
6 14 V
6 -7 V
6 -7 V
6 14 V
6 15 V
6 -7 V
6 -8 V
6 -7 V
6 15 V
6 -8 V
6 -7 V
6 -7 V
6 14 V
6 -7 V
6 -7 V
6 14 V
6 15 V
6 -7 V
7 -8 V
6 -7 V
6 15 V
6 14 V
6 15 V
6 -7 V
6 14 V
3357 2138 L
6 -7 V
6 14 V
6 15 V
6 15 V
6 14 V
6 15 V
6 -7 V
6 -8 V
6 -7 V
6 -7 V
6 14 V
6 15 V
6 15 V
6 -8 V
6 -7 V
6 -7 V
6 14 V
7 15 V
6 -7 V
6 14 V
6 -7 V
6 -7 V
6 -8 V
6 15 V
6 15 V
6 14 V
6 15 V
6 15 V
6 -8 V
6 15 V
6 15 V
6 14 V
6 -7 V
6 -7 V
6 14 V
6 15 V
6 15 V
6 -8 V
6 15 V
6 15 V
6 14 V
6 15 V
7 15 V
6 -8 V
6 15 V
6 15 V
6 14 V
6 15 V
6 -7 V
6 14 V
6 -7 V
6 15 V
6 -8 V
6 15 V
6 15 V
6 14 V
6 15 V
6 15 V
6 -8 V
6 15 V
6 -7 V
6 14 V
6 15 V
6 15 V
6 -8 V
6 -7 V
6 -7 V
7 -8 V
6 15 V
6 15 V
6 14 V
6 15 V
6 -7 V
6 -8 V
6 -7 V
6 15 V
6 -8 V
6 15 V
6 -7 V
6 14 V
6 -7 V
6 -7 V
6 14 V
6 15 V
6 15 V
6 14 V
6 -7 V
6 -7 V
6 -8 V
6 -7 V
6 15 V
6 14 V
7 -7 V
6 15 V
6 -8 V
6 15 V
6 -7 V
6 -8 V
6 -7 V
6 -7 V
6 -8 V
6 -7 V
6 -7 V
6 -8 V
3991 2622 L
6 -7 V
6 -8 V
6 15 V
6 -7 V
6 14 V
6 -7 V
6 -7 V
6 -8 V
6 -7 V
6 -7 V
6 14 V
6 -7 V
7 -7 V
6 14 V
6 15 V
6 -7 V
6 14 V
6 -7 V
6 15 V
6 14 V
6 15 V
6 15 V
6 14 V
6 15 V
6 -7 V
6 14 V
6 15 V
6 -7 V
6 -8 V
6 15 V
6 15 V
6 14 V
6 -7 V
6 15 V
6 14 V
6 15 V
6 -7 V
7 14 V
6 15 V
6 15 V
6 14 V
6 15 V
6 -7 V
6 -8 V
6 -7 V
6 15 V
6 14 V
6 -7 V
6 15 V
6 14 V
6 15 V
6 15 V
6 -8 V
6 -7 V
6 15 V
6 -8 V
6 15 V
6 15 V
6 14 V
6 -7 V
6 15 V
6 -8 V
7 -7 V
6 15 V
6 14 V
6 -7 V
6 15 V
6 14 V
6 -7 V
6 -7 V
6 14 V
6 -7 V
6 15 V
6 -8 V
6 15 V
6 -7 V
6 14 V
6 -7 V
6 -7 V
6 14 V
6 -7 V
6 -7 V
6 -8 V
6 -7 V
6 -7 V
6 14 V
6 -7 V
7 -7 V
6 14 V
6 15 V
6 -7 V
6 -8 V
6 -7 V
6 -7 V
6 14 V
6 -7 V
6 -7 V
6 14 V
6 -7 V
6 15 V
6 -8 V
6 -7 V
6 15 V
6 14 V
4625 3040 L
6 15 V
6 14 V
6 15 V
6 15 V
6 14 V
6 15 V
6 15 V
7 -8 V
6 -7 V
6 15 V
6 14 V
6 15 V
6 15 V
6 -8 V
6 -7 V
6 -7 V
6 14 V
6 15 V
6 15 V
6 14 V
6 15 V
6 -7 V
6 14 V
6 -7 V
6 -7 V
6 14 V
6 -7 V
6 15 V
6 14 V
6 15 V
6 15 V
6 -8 V
7 -7 V
6 -7 V
6 14 V
6 15 V
6 -7 V
6 -8 V
6 -7 V
6 15 V
6 14 V
6 -7 V
6 -7 V
6 -8 V
6 15 V
6 -7 V
6 14 V
6 15 V
6 15 V
6 -8 V
6 -7 V
6 -7 V
6 -8 V
6 15 V
6 -7 V
6 14 V
6 15 V
7 -7 V
6 -8 V
6 -7 V
6 15 V
6 14 V
6 15 V
6 -7 V
6 -8 V
6 15 V
6 -7 V
6 -8 V
6 15 V
6 -7 V
6 14 V
6 -7 V
6 15 V
6 14 V
6 15 V
6 15 V
6 14 V
6 15 V
6 15 V
6 -8 V
6 -7 V
6 15 V
7 14 V
6 -7 V
6 15 V
6 -8 V
6 -7 V
6 -7 V
6 -8 V
6 15 V
6 15 V
6 14 V
6 -7 V
6 15 V
6 -8 V
6 -7 V
6 15 V
6 14 V
6 -7 V
6 15 V
6 -8 V
6 -7 V
6 15 V
6 -8 V
5259 3546 L
6 15 V
6 14 V
7 -7 V
6 15 V
6 -8 V
6 -7 V
6 15 V
6 -8 V
6 -7 V
6 15 V
6 -8 V
6 -7 V
6 -7 V
6 14 V
6 15 V
6 -7 V
6 -8 V
6 -7 V
6 -7 V
6 14 V
6 -7 V
6 15 V
6 14 V
6 -7 V
6 -7 V
6 14 V
6 15 V
7 15 V
6 14 V
6 -7 V
6 15 V
6 14 V
6 15 V
6 -7 V
6 14 V
6 15 V
6 15 V
6 14 V
6 -7 V
6 -7 V
6 14 V
6 15 V
6 -7 V
6 -8 V
6 15 V
6 15 V
6 -8 V
6 15 V
6 15 V
6 -8 V
6 15 V
6 15 V
7 14 V
6 15 V
6 -7 V
6 14 V
6 -7 V
6 15 V
6 14 V
6 -7 V
6 15 V
6 14 V
6 15 V
6 15 V
6 14 V
6 15 V
6 -7 V
6 -8 V
6 -7 V
6 -7 V
6 14 V
6 15 V
6 -7 V
6 -8 V
6 -7 V
6 -7 V
6 -8 V
7 15 V
6 -7 V
6 14 V
6 15 V
6 -7 V
6 -8 V
6 15 V
6 15 V
6 -8 V
6 15 V
6 -7 V
6 14 V
6 -7 V
6 -7 V
6 -8 V
6 -7 V
6 15 V
6 14 V
6 -7 V
6 -7 V
6 -8 V
6 15 V
6 15 V
6 -8 V
6 -7 V
7 -7 V
6 14 V
5894 3964 L
6 15 V
6 14 V
6 -7 V
6 -7 V
6 14 V
6 15 V
6 -7 V
6 14 V
6 -7 V
6 -7 V
6 -8 V
6 15 V
6 -7 V
6 14 V
6 15 V
6 15 V
6 14 V
6 15 V
6 -7 V
6 -8 V
6 -7 V
6 15 V
7 14 V
6 15 V
6 15 V
6 14 V
6 -7 V
6 15 V
6 -8 V
6 15 V
6 15 V
6 -8 V
6 -7 V
6 -7 V
6 14 V
6 -7 V
6 15 V
6 -8 V
6 -7 V
6 15 V
6 -8 V
6 15 V
6 15 V
6 -8 V
6 -7 V
6 -7 V
6 14 V
7 15 V
6 15 V
6 -8 V
6 15 V
6 -7 V
6 -8 V
6 15 V
6 15 V
6 14 V
6 15 V
6 15 V
6 14 V
6 15 V
6 -7 V
6 -8 V
6 -7 V
6 15 V
6 -8 V
6 -7 V
6 -7 V
6 14 V
6 15 V
6 -7 V
6 14 V
6 -7 V
7 15 V
6 -8 V
6 -7 V
6 -7 V
6 14 V
6 15 V
6 -7 V
6 -8 V
6 -7 V
6 15 V
6 14 V
6 15 V
6 15 V
6 -8 V
6 15 V
6 -7 V
6 -8 V
6 -7 V
6 15 V
6 -8 V
6 15 V
6 -7 V
6 -8 V
6 -7 V
6 -7 V
7 14 V
6 -7 V
6 -7 V
6 14 V
6 15 V
6 15 V
6 14 V
6528 4382 L
6 15 V
6 -8 V
6 15 V
6 15 V
6 -8 V
6 -7 V
6 -7 V
6 14 V
6 15 V
6 15 V
6 -8 V
6 -7 V
6 15 V
6 14 V
6 15 V
6 15 V
6 -8 V
7 15 V
6 -7 V
6 14 V
6 -7 V
6 -7 V
6 14 V
6 15 V
6 15 V
6 14 V
6 -7 V
6 -7 V
6 14 V
6 -7 V
6 15 V
6 14 V
6 -7 V
6 15 V
6 -8 V
6 15 V
6 15 V
6 -8 V
6 15 V
6 15 V
6 -8 V
6 -7 V
7 15 V
6 -8 V
6 -7 V
6 -7 V
6 -8 V
6 -7 V
6 15 V
6 -8 V
6 -7 V
6 15 V
6 14 V
6 15 V
6 15 V
stroke
LT3
820 400 M
6 15 V
6 -8 V
6 15 V
6 15 V
6 14 V
6 15 V
6 15 V
6 -8 V
6 -7 V
6 15 V
6 -8 V
6 15 V
7 -7 V
6 -8 V
6 15 V
6 15 V
6 14 V
6 -7 V
6 15 V
6 -8 V
6 15 V
6 -7 V
6 -8 V
6 15 V
6 -7 V
6 -8 V
6 15 V
6 15 V
6 14 V
6 -7 V
6 15 V
6 14 V
6 15 V
6 15 V
6 14 V
6 -7 V
6 -7 V
7 -8 V
6 -7 V
6 -7 V
6 -8 V
6 15 V
6 15 V
6 -8 V
6 -7 V
6 15 V
6 14 V
6 -7 V
6 -7 V
6 14 V
6 15 V
6 15 V
6 -8 V
6 -7 V
6 -7 V
6 -8 V
6 -7 V
6 15 V
6 14 V
6 15 V
6 15 V
6 14 V
7 -7 V
6 15 V
6 14 V
6 -7 V
6 15 V
6 14 V
6 -7 V
6 -7 V
6 -8 V
6 -7 V
6 -7 V
6 14 V
6 15 V
6 -7 V
6 -8 V
6 -7 V
6 -7 V
6 -8 V
6 15 V
6 15 V
6 -8 V
6 -7 V
6 15 V
6 14 V
6 -7 V
7 -7 V
6 -8 V
6 -7 V
6 -7 V
6 14 V
6 -7 V
6 15 V
6 14 V
6 -7 V
6 15 V
6 14 V
6 -7 V
6 -7 V
6 -8 V
6 -7 V
6 15 V
6 14 V
1454 774 L
6 -7 V
6 -8 V
6 -7 V
6 15 V
6 -8 V
6 -7 V
6 -7 V
7 14 V
6 15 V
6 15 V
6 14 V
6 15 V
6 15 V
6 14 V
6 15 V
6 -7 V
6 -8 V
6 15 V
6 -7 V
6 14 V
6 15 V
6 15 V
6 14 V
6 -7 V
6 -7 V
6 -8 V
6 -7 V
6 15 V
6 14 V
6 -7 V
6 15 V
6 14 V
7 -7 V
6 -7 V
6 14 V
6 -7 V
6 -7 V
6 -8 V
6 15 V
6 15 V
6 -8 V
6 -7 V
6 -7 V
6 14 V
6 -7 V
6 -7 V
6 -8 V
6 -7 V
6 15 V
6 -8 V
6 -7 V
6 -7 V
6 -8 V
6 -7 V
6 -7 V
6 14 V
6 -7 V
7 -7 V
6 -8 V
6 15 V
6 -7 V
6 -8 V
6 -7 V
6 -7 V
6 -8 V
6 -7 V
6 15 V
6 -8 V
6 15 V
6 -7 V
6 -8 V
6 -7 V
6 15 V
6 14 V
6 15 V
6 -7 V
6 14 V
6 -7 V
6 -7 V
6 -8 V
6 -7 V
6 -7 V
7 -8 V
6 15 V
6 -7 V
6 14 V
6 -7 V
6 -7 V
6 14 V
6 15 V
6 15 V
6 14 V
6 -7 V
6 15 V
6 14 V
6 15 V
6 -7 V
6 14 V
6 -7 V
6 -7 V
6 -8 V
6 -7 V
6 15 V
6 14 V
2088 972 L
6 15 V
6 -8 V
7 -7 V
6 15 V
6 -8 V
6 -7 V
6 -7 V
6 14 V
6 -7 V
6 -7 V
6 14 V
6 -7 V
6 -7 V
6 14 V
6 15 V
6 15 V
6 -8 V
6 15 V
6 -7 V
6 14 V
6 15 V
6 15 V
6 14 V
6 15 V
6 -7 V
6 14 V
6 -7 V
7 -7 V
6 -8 V
6 15 V
6 15 V
6 14 V
6 15 V
6 -7 V
6 -8 V
6 15 V
6 15 V
6 -8 V
6 15 V
6 15 V
6 -8 V
6 15 V
6 -7 V
6 -8 V
6 -7 V
6 -7 V
6 14 V
6 15 V
6 15 V
6 14 V
6 15 V
6 -7 V
7 -8 V
6 15 V
6 15 V
6 14 V
6 -7 V
6 -7 V
6 14 V
6 -7 V
6 15 V
6 -8 V
6 15 V
6 15 V
6 14 V
6 15 V
6 -7 V
6 -8 V
6 -7 V
6 15 V
6 -8 V
6 15 V
6 -7 V
6 14 V
6 15 V
6 -7 V
6 -8 V
7 -7 V
6 -7 V
6 -8 V
6 -7 V
6 15 V
6 -8 V
6 15 V
6 -7 V
6 14 V
6 15 V
6 15 V
6 -8 V
6 15 V
6 -7 V
6 14 V
6 -7 V
6 -7 V
6 -8 V
6 -7 V
6 15 V
6 14 V
6 -7 V
6 -7 V
6 -8 V
6 15 V
7 -7 V
6 -8 V
2723 1324 L
6 15 V
6 14 V
6 -7 V
6 15 V
6 -8 V
6 15 V
6 -7 V
6 14 V
6 -7 V
6 15 V
6 14 V
6 -7 V
6 15 V
6 -8 V
6 15 V
6 -7 V
6 14 V
6 15 V
6 15 V
6 14 V
6 -7 V
6 -7 V
7 -8 V
6 -7 V
6 15 V
6 -8 V
6 -7 V
6 -7 V
6 14 V
6 -7 V
6 -7 V
6 -8 V
6 -7 V
6 15 V
6 -8 V
6 -7 V
6 -7 V
6 -8 V
6 15 V
6 -7 V
6 -8 V
6 -7 V
6 -7 V
6 -8 V
6 -7 V
6 -7 V
6 14 V
7 15 V
6 -7 V
6 -8 V
6 15 V
6 15 V
6 14 V
6 15 V
6 15 V
6 -8 V
6 -7 V
6 -7 V
6 -8 V
6 -7 V
6 -7 V
6 -8 V
6 15 V
6 -7 V
6 14 V
6 -7 V
6 -7 V
6 14 V
6 -7 V
6 15 V
6 -8 V
6 15 V
7 -7 V
6 -8 V
6 -7 V
6 -7 V
6 14 V
6 15 V
6 -7 V
6 14 V
6 15 V
6 15 V
6 -8 V
6 -7 V
6 15 V
6 14 V
6 -7 V
6 -7 V
6 14 V
6 15 V
6 15 V
6 14 V
6 15 V
6 15 V
6 -8 V
6 -7 V
6 -7 V
7 14 V
6 15 V
6 -7 V
6 -8 V
6 -7 V
6 15 V
6 14 V
3357 1588 L
6 15 V
6 14 V
6 15 V
6 15 V
6 14 V
6 -7 V
6 -7 V
6 14 V
6 15 V
6 -7 V
6 14 V
6 -7 V
6 15 V
6 14 V
6 15 V
6 -7 V
6 14 V
7 15 V
6 -7 V
6 -8 V
6 15 V
6 15 V
6 14 V
6 -7 V
6 -7 V
6 14 V
6 -7 V
6 -7 V
6 14 V
6 15 V
6 15 V
6 14 V
6 -7 V
6 -7 V
6 -8 V
6 -7 V
6 -7 V
6 -8 V
6 -7 V
6 -7 V
6 14 V
6 -7 V
7 15 V
6 -8 V
6 15 V
6 -7 V
6 14 V
6 15 V
6 15 V
6 -8 V
6 15 V
6 15 V
6 14 V
6 -7 V
6 15 V
6 14 V
6 -7 V
6 -7 V
6 14 V
6 15 V
6 -7 V
6 -8 V
6 15 V
6 15 V
6 14 V
6 -7 V
6 -7 V
7 -8 V
6 15 V
6 15 V
6 14 V
6 -7 V
6 -7 V
6 14 V
6 -7 V
6 -7 V
6 14 V
6 15 V
6 -7 V
6 14 V
6 15 V
6 15 V
6 14 V
6 15 V
6 -7 V
6 14 V
6 -7 V
6 -7 V
6 -8 V
6 -7 V
6 15 V
6 14 V
7 15 V
6 15 V
6 -8 V
6 15 V
6 15 V
6 -8 V
6 15 V
6 15 V
6 -8 V
6 15 V
6 15 V
6 14 V
3991 2138 L
6 15 V
6 -8 V
6 15 V
6 15 V
6 -8 V
6 15 V
6 15 V
6 14 V
6 -7 V
6 15 V
6 14 V
6 -7 V
7 15 V
6 -8 V
6 -7 V
6 -7 V
6 -8 V
6 -7 V
6 -7 V
6 14 V
6 15 V
6 -7 V
6 14 V
6 -7 V
6 -7 V
6 -8 V
6 -7 V
6 15 V
6 14 V
6 -7 V
6 -7 V
6 14 V
6 -7 V
6 -7 V
6 -8 V
6 -7 V
6 -7 V
7 -8 V
6 15 V
6 -7 V
6 -8 V
6 15 V
6 15 V
6 -8 V
6 -7 V
6 -7 V
6 14 V
6 -7 V
6 15 V
6 -8 V
6 15 V
6 -7 V
6 -8 V
6 -7 V
6 -7 V
6 -8 V
6 15 V
6 15 V
6 14 V
6 -7 V
6 15 V
6 14 V
7 -7 V
6 -7 V
6 -8 V
6 -7 V
6 15 V
6 14 V
6 -7 V
6 15 V
6 -8 V
6 -7 V
6 -7 V
6 14 V
6 -7 V
6 -7 V
6 14 V
6 -7 V
6 15 V
6 -8 V
6 -7 V
6 15 V
6 14 V
6 15 V
6 15 V
6 -8 V
6 15 V
7 -7 V
6 14 V
6 -7 V
6 -7 V
6 -8 V
6 -7 V
6 15 V
6 14 V
6 -7 V
6 -7 V
6 -8 V
6 15 V
6 15 V
6 -8 V
6 15 V
6 15 V
6 -8 V
4625 2336 L
6 -7 V
6 14 V
6 15 V
6 -7 V
6 -8 V
6 -7 V
6 15 V
7 -8 V
6 -7 V
6 -7 V
6 14 V
6 -7 V
6 15 V
6 14 V
6 -7 V
6 15 V
6 14 V
6 15 V
6 -7 V
6 -8 V
6 15 V
6 15 V
6 14 V
6 -7 V
6 -7 V
6 -8 V
6 15 V
6 15 V
6 -8 V
6 15 V
6 15 V
6 -8 V
7 -7 V
6 -7 V
6 -8 V
6 15 V
6 -7 V
6 -8 V
6 -7 V
6 -7 V
6 -8 V
6 -7 V
6 -7 V
6 -8 V
6 15 V
6 -7 V
6 14 V
6 -7 V
6 -7 V
6 14 V
6 15 V
6 15 V
6 14 V
6 15 V
6 15 V
6 14 V
6 15 V
7 15 V
6 14 V
6 15 V
6 -7 V
6 14 V
6 -7 V
6 15 V
6 14 V
6 15 V
6 -7 V
6 14 V
6 15 V
6 15 V
6 14 V
6 -7 V
6 15 V
6 -8 V
6 -7 V
6 15 V
6 -8 V
6 15 V
6 15 V
6 -8 V
6 15 V
6 15 V
7 -8 V
6 -7 V
6 15 V
6 -8 V
6 15 V
6 -7 V
6 14 V
6 -7 V
6 15 V
6 -8 V
6 -7 V
6 15 V
6 14 V
6 -7 V
6 15 V
6 -8 V
6 -7 V
6 -7 V
6 14 V
6 15 V
6 15 V
6 -8 V
5259 2776 L
6 -7 V
6 -8 V
7 -7 V
6 15 V
6 14 V
6 15 V
6 15 V
6 14 V
6 -7 V
6 15 V
6 -8 V
6 15 V
6 15 V
6 -8 V
6 15 V
6 -7 V
6 -8 V
6 -7 V
6 15 V
6 14 V
6 15 V
6 15 V
6 -8 V
6 -7 V
6 15 V
6 -8 V
6 -7 V
7 -7 V
6 -8 V
6 -7 V
6 -7 V
6 14 V
6 15 V
6 15 V
6 14 V
6 15 V
6 -7 V
6 -8 V
6 15 V
6 15 V
6 14 V
6 15 V
6 -7 V
6 14 V
6 -7 V
6 -7 V
6 14 V
6 -7 V
6 -7 V
6 -8 V
6 -7 V
6 15 V
7 -8 V
6 -7 V
6 -7 V
6 14 V
6 -7 V
6 -7 V
6 -8 V
6 -7 V
6 -7 V
6 -8 V
6 15 V
6 15 V
6 14 V
6 15 V
6 15 V
6 14 V
6 15 V
6 15 V
6 -8 V
6 15 V
6 -7 V
6 14 V
6 -7 V
6 15 V
6 -8 V
7 -7 V
6 -7 V
6 14 V
6 -7 V
6 15 V
6 -8 V
6 -7 V
6 -7 V
6 -8 V
6 15 V
6 -7 V
6 14 V
6 15 V
6 -7 V
6 14 V
6 -7 V
6 -7 V
6 -8 V
6 -7 V
6 15 V
6 -8 V
6 15 V
6 -7 V
6 -8 V
6 -7 V
7 15 V
6 14 V
5894 3084 L
6 -7 V
6 14 V
6 15 V
6 -7 V
6 -8 V
6 15 V
6 15 V
6 14 V
6 15 V
6 15 V
6 -8 V
6 -7 V
6 -7 V
6 14 V
6 15 V
6 -7 V
6 14 V
6 -7 V
6 -7 V
6 14 V
6 15 V
6 15 V
7 -8 V
6 15 V
6 -7 V
6 14 V
6 15 V
6 15 V
6 14 V
6 15 V
6 15 V
6 14 V
6 -7 V
6 15 V
6 14 V
6 -7 V
6 15 V
6 -8 V
6 -7 V
6 15 V
6 14 V
6 15 V
6 -7 V
6 14 V
6 15 V
6 -7 V
6 -8 V
7 -7 V
6 -7 V
6 -8 V
6 15 V
6 15 V
6 14 V
6 -7 V
6 -7 V
6 14 V
6 -7 V
6 15 V
6 14 V
6 15 V
6 -7 V
6 -8 V
6 15 V
6 -7 V
6 14 V
6 -7 V
6 -7 V
6 -8 V
6 15 V
6 15 V
6 -8 V
6 -7 V
7 15 V
6 14 V
6 -7 V
6 15 V
6 14 V
6 -7 V
6 15 V
6 14 V
6 -7 V
6 -7 V
6 -8 V
6 15 V
6 15 V
6 -8 V
6 -7 V
6 15 V
6 14 V
6 15 V
6 -7 V
6 14 V
6 15 V
6 -7 V
6 -8 V
6 -7 V
6 -7 V
7 -8 V
6 -7 V
6 -7 V
6 -8 V
6 15 V
6 -7 V
6 -8 V
6528 3502 L
6 15 V
6 -8 V
6 -7 V
6 -7 V
6 -8 V
6 15 V
6 -7 V
6 14 V
6 -7 V
6 -7 V
6 -8 V
6 15 V
6 15 V
6 -8 V
6 -7 V
6 15 V
6 14 V
7 -7 V
6 -7 V
6 14 V
6 -7 V
6 -7 V
6 -8 V
6 15 V
6 15 V
6 14 V
6 -7 V
6 15 V
6 -8 V
6 15 V
6 -7 V
6 14 V
6 15 V
6 15 V
6 -8 V
6 -7 V
6 15 V
6 -8 V
6 -7 V
6 15 V
6 -8 V
6 15 V
7 -7 V
6 -8 V
6 15 V
6 15 V
6 14 V
6 -7 V
6 15 V
6 -8 V
6 -7 V
6 15 V
6 14 V
6 -7 V
6 15 V
stroke
LTb
820 4800 M
820 400 L
6040 0 V
0 4400 V
-6040 0 V
1.000 UP
stroke
grestore
end
showpage
  }}%
  \put(6860,200){\makebox(0,0){\strut{} 1000}}%
  \put(5652,200){\makebox(0,0){\strut{} 800}}%
  \put(4444,200){\makebox(0,0){\strut{} 600}}%
  \put(3236,200){\makebox(0,0){\strut{} 400}}%
  \put(2028,200){\makebox(0,0){\strut{} 200}}%
  \put(820,200){\makebox(0,0){\strut{} 0}}%
  \put(700,4800){\makebox(0,0)[r]{\strut{} 600}}%
  \put(700,4067){\makebox(0,0)[r]{\strut{} 500}}%
  \put(700,3333){\makebox(0,0)[r]{\strut{} 400}}%
  \put(700,2600){\makebox(0,0)[r]{\strut{} 300}}%
  \put(700,1867){\makebox(0,0)[r]{\strut{} 200}}%
  \put(700,1133){\makebox(0,0)[r]{\strut{} 100}}%
  \put(700,400){\makebox(0,0)[r]{\strut{} 0}}%
\end{picture}%
\endgroup
 

%% file: Paths1-2-1.tex
% GNUPLOT: LaTeX picture with Postscript
\begingroup%
\makeatletter%
\newcommand{\GNUPLOTspecial}{%
  \@sanitize\catcode`\%=14\relax\special}%
\setlength{\unitlength}{0.0500bp}%
\begin{picture}(7200,5040)(0,0)%
  {\GNUPLOTspecial{"
%!PS-Adobe-2.0 EPSF-2.0
%%Title: C:/Documents and Settings/ponty/My Documents/Tex/culminants/Paths1-2-1.tex
%%Creator: gnuplot 4.2 patchlevel rc3
%%CreationDate: Sun Apr 08 09:55:26 2007
%%DocumentFonts: 
%%BoundingBox: 0 0 360 252
%%EndComments
%%BeginProlog
/gnudict 256 dict def
gnudict begin
%
% The following 6 true/false flags may be edited by hand if required
% The unit line width may also be changed
%
/Color false def
/Blacktext false def
/Solid false def
/Dashlength 1 def
/Landscape false def
/Level1 false def
/Rounded false def
/TransparentPatterns false def
/gnulinewidth 5.000 def
/userlinewidth gnulinewidth def
/vshift -66 def
/dl1 {
  10.0 Dashlength mul mul
  Rounded { currentlinewidth 0.75 mul sub dup 0 le { pop 0.01 } if } if
} def
/dl2 {
  10.0 Dashlength mul mul
  Rounded { currentlinewidth 0.75 mul add } if
} def
/hpt_ 31.5 def
/vpt_ 31.5 def
/hpt hpt_ def
/vpt vpt_ def
Level1 {} {
/SDict 10 dict def
systemdict /pdfmark known not {
  userdict /pdfmark systemdict /cleartomark get put
} if
SDict begin [
  /Title (C:/Documents and Settings/ponty/My Documents/Tex/culminants/Paths1-2-1.tex)
  /Subject (gnuplot plot)
  /Creator (gnuplot 4.2 patchlevel rc3)
  /Author (ponty)
%  /Producer (gnuplot)
%  /Keywords ()
  /CreationDate (Sun Apr 08 09:55:26 2007)
  /DOCINFO pdfmark
end
} ifelse
%
% Gnuplot Prolog Version 4.2 (August 2006)
%
/M {moveto} bind def
/L {lineto} bind def
/R {rmoveto} bind def
/V {rlineto} bind def
/N {newpath moveto} bind def
/Z {closepath} bind def
/C {setrgbcolor} bind def
/f {rlineto fill} bind def
/vpt2 vpt 2 mul def
/hpt2 hpt 2 mul def
/Lshow {currentpoint stroke M 0 vshift R 
	Blacktext {gsave 0 setgray show grestore} {show} ifelse} def
/Rshow {currentpoint stroke M dup stringwidth pop neg vshift R
	Blacktext {gsave 0 setgray show grestore} {show} ifelse} def
/Cshow {currentpoint stroke M dup stringwidth pop -2 div vshift R 
	Blacktext {gsave 0 setgray show grestore} {show} ifelse} def
/UP {dup vpt_ mul /vpt exch def hpt_ mul /hpt exch def
  /hpt2 hpt 2 mul def /vpt2 vpt 2 mul def} def
/DL {Color {setrgbcolor Solid {pop []} if 0 setdash}
 {pop pop pop 0 setgray Solid {pop []} if 0 setdash} ifelse} def
/BL {stroke userlinewidth 2 mul setlinewidth
	Rounded {1 setlinejoin 1 setlinecap} if} def
/AL {stroke userlinewidth 2 div setlinewidth
	Rounded {1 setlinejoin 1 setlinecap} if} def
/UL {dup gnulinewidth mul /userlinewidth exch def
	dup 1 lt {pop 1} if 10 mul /udl exch def} def
/PL {stroke userlinewidth setlinewidth
	Rounded {1 setlinejoin 1 setlinecap} if} def
% Default Line colors
/LCw {1 1 1} def
/LCb {0 0 0} def
/LCa {0 0 0} def
/LC0 {1 0 0} def
/LC1 {0 1 0} def
/LC2 {0 0 1} def
/LC3 {1 0 1} def
/LC4 {0 1 1} def
/LC5 {1 1 0} def
/LC6 {0 0 0} def
/LC7 {1 0.3 0} def
/LC8 {0.5 0.5 0.5} def
% Default Line Types
/LTw {PL [] 1 setgray} def
/LTb {BL [] LCb DL} def
/LTa {AL [1 udl mul 2 udl mul] 0 setdash LCa setrgbcolor} def
/LT0 {PL [] LC0 DL} def
/LT1 {PL [4 dl1 2 dl2] LC1 DL} def
/LT2 {PL [2 dl1 3 dl2] LC2 DL} def
/LT3 {PL [1 dl1 1.5 dl2] LC3 DL} def
/LT4 {PL [6 dl1 2 dl2 1 dl1 2 dl2] LC4 DL} def
/LT5 {PL [3 dl1 3 dl2 1 dl1 3 dl2] LC5 DL} def
/LT6 {PL [2 dl1 2 dl2 2 dl1 6 dl2] LC6 DL} def
/LT7 {PL [1 dl1 2 dl2 6 dl1 2 dl2 1 dl1 2 dl2] LC7 DL} def
/LT8 {PL [2 dl1 2 dl2 2 dl1 2 dl2 2 dl1 2 dl2 2 dl1 4 dl2] LC8 DL} def
/Pnt {stroke [] 0 setdash gsave 1 setlinecap M 0 0 V stroke grestore} def
/Dia {stroke [] 0 setdash 2 copy vpt add M
  hpt neg vpt neg V hpt vpt neg V
  hpt vpt V hpt neg vpt V closepath stroke
  Pnt} def
/Pls {stroke [] 0 setdash vpt sub M 0 vpt2 V
  currentpoint stroke M
  hpt neg vpt neg R hpt2 0 V stroke
 } def
/Box {stroke [] 0 setdash 2 copy exch hpt sub exch vpt add M
  0 vpt2 neg V hpt2 0 V 0 vpt2 V
  hpt2 neg 0 V closepath stroke
  Pnt} def
/Crs {stroke [] 0 setdash exch hpt sub exch vpt add M
  hpt2 vpt2 neg V currentpoint stroke M
  hpt2 neg 0 R hpt2 vpt2 V stroke} def
/TriU {stroke [] 0 setdash 2 copy vpt 1.12 mul add M
  hpt neg vpt -1.62 mul V
  hpt 2 mul 0 V
  hpt neg vpt 1.62 mul V closepath stroke
  Pnt} def
/Star {2 copy Pls Crs} def
/BoxF {stroke [] 0 setdash exch hpt sub exch vpt add M
  0 vpt2 neg V hpt2 0 V 0 vpt2 V
  hpt2 neg 0 V closepath fill} def
/TriUF {stroke [] 0 setdash vpt 1.12 mul add M
  hpt neg vpt -1.62 mul V
  hpt 2 mul 0 V
  hpt neg vpt 1.62 mul V closepath fill} def
/TriD {stroke [] 0 setdash 2 copy vpt 1.12 mul sub M
  hpt neg vpt 1.62 mul V
  hpt 2 mul 0 V
  hpt neg vpt -1.62 mul V closepath stroke
  Pnt} def
/TriDF {stroke [] 0 setdash vpt 1.12 mul sub M
  hpt neg vpt 1.62 mul V
  hpt 2 mul 0 V
  hpt neg vpt -1.62 mul V closepath fill} def
/DiaF {stroke [] 0 setdash vpt add M
  hpt neg vpt neg V hpt vpt neg V
  hpt vpt V hpt neg vpt V closepath fill} def
/Pent {stroke [] 0 setdash 2 copy gsave
  translate 0 hpt M 4 {72 rotate 0 hpt L} repeat
  closepath stroke grestore Pnt} def
/PentF {stroke [] 0 setdash gsave
  translate 0 hpt M 4 {72 rotate 0 hpt L} repeat
  closepath fill grestore} def
/Circle {stroke [] 0 setdash 2 copy
  hpt 0 360 arc stroke Pnt} def
/CircleF {stroke [] 0 setdash hpt 0 360 arc fill} def
/C0 {BL [] 0 setdash 2 copy moveto vpt 90 450 arc} bind def
/C1 {BL [] 0 setdash 2 copy moveto
	2 copy vpt 0 90 arc closepath fill
	vpt 0 360 arc closepath} bind def
/C2 {BL [] 0 setdash 2 copy moveto
	2 copy vpt 90 180 arc closepath fill
	vpt 0 360 arc closepath} bind def
/C3 {BL [] 0 setdash 2 copy moveto
	2 copy vpt 0 180 arc closepath fill
	vpt 0 360 arc closepath} bind def
/C4 {BL [] 0 setdash 2 copy moveto
	2 copy vpt 180 270 arc closepath fill
	vpt 0 360 arc closepath} bind def
/C5 {BL [] 0 setdash 2 copy moveto
	2 copy vpt 0 90 arc
	2 copy moveto
	2 copy vpt 180 270 arc closepath fill
	vpt 0 360 arc} bind def
/C6 {BL [] 0 setdash 2 copy moveto
	2 copy vpt 90 270 arc closepath fill
	vpt 0 360 arc closepath} bind def
/C7 {BL [] 0 setdash 2 copy moveto
	2 copy vpt 0 270 arc closepath fill
	vpt 0 360 arc closepath} bind def
/C8 {BL [] 0 setdash 2 copy moveto
	2 copy vpt 270 360 arc closepath fill
	vpt 0 360 arc closepath} bind def
/C9 {BL [] 0 setdash 2 copy moveto
	2 copy vpt 270 450 arc closepath fill
	vpt 0 360 arc closepath} bind def
/C10 {BL [] 0 setdash 2 copy 2 copy moveto vpt 270 360 arc closepath fill
	2 copy moveto
	2 copy vpt 90 180 arc closepath fill
	vpt 0 360 arc closepath} bind def
/C11 {BL [] 0 setdash 2 copy moveto
	2 copy vpt 0 180 arc closepath fill
	2 copy moveto
	2 copy vpt 270 360 arc closepath fill
	vpt 0 360 arc closepath} bind def
/C12 {BL [] 0 setdash 2 copy moveto
	2 copy vpt 180 360 arc closepath fill
	vpt 0 360 arc closepath} bind def
/C13 {BL [] 0 setdash 2 copy moveto
	2 copy vpt 0 90 arc closepath fill
	2 copy moveto
	2 copy vpt 180 360 arc closepath fill
	vpt 0 360 arc closepath} bind def
/C14 {BL [] 0 setdash 2 copy moveto
	2 copy vpt 90 360 arc closepath fill
	vpt 0 360 arc} bind def
/C15 {BL [] 0 setdash 2 copy vpt 0 360 arc closepath fill
	vpt 0 360 arc closepath} bind def
/Rec {newpath 4 2 roll moveto 1 index 0 rlineto 0 exch rlineto
	neg 0 rlineto closepath} bind def
/Square {dup Rec} bind def
/Bsquare {vpt sub exch vpt sub exch vpt2 Square} bind def
/S0 {BL [] 0 setdash 2 copy moveto 0 vpt rlineto BL Bsquare} bind def
/S1 {BL [] 0 setdash 2 copy vpt Square fill Bsquare} bind def
/S2 {BL [] 0 setdash 2 copy exch vpt sub exch vpt Square fill Bsquare} bind def
/S3 {BL [] 0 setdash 2 copy exch vpt sub exch vpt2 vpt Rec fill Bsquare} bind def
/S4 {BL [] 0 setdash 2 copy exch vpt sub exch vpt sub vpt Square fill Bsquare} bind def
/S5 {BL [] 0 setdash 2 copy 2 copy vpt Square fill
	exch vpt sub exch vpt sub vpt Square fill Bsquare} bind def
/S6 {BL [] 0 setdash 2 copy exch vpt sub exch vpt sub vpt vpt2 Rec fill Bsquare} bind def
/S7 {BL [] 0 setdash 2 copy exch vpt sub exch vpt sub vpt vpt2 Rec fill
	2 copy vpt Square fill Bsquare} bind def
/S8 {BL [] 0 setdash 2 copy vpt sub vpt Square fill Bsquare} bind def
/S9 {BL [] 0 setdash 2 copy vpt sub vpt vpt2 Rec fill Bsquare} bind def
/S10 {BL [] 0 setdash 2 copy vpt sub vpt Square fill 2 copy exch vpt sub exch vpt Square fill
	Bsquare} bind def
/S11 {BL [] 0 setdash 2 copy vpt sub vpt Square fill 2 copy exch vpt sub exch vpt2 vpt Rec fill
	Bsquare} bind def
/S12 {BL [] 0 setdash 2 copy exch vpt sub exch vpt sub vpt2 vpt Rec fill Bsquare} bind def
/S13 {BL [] 0 setdash 2 copy exch vpt sub exch vpt sub vpt2 vpt Rec fill
	2 copy vpt Square fill Bsquare} bind def
/S14 {BL [] 0 setdash 2 copy exch vpt sub exch vpt sub vpt2 vpt Rec fill
	2 copy exch vpt sub exch vpt Square fill Bsquare} bind def
/S15 {BL [] 0 setdash 2 copy Bsquare fill Bsquare} bind def
/D0 {gsave translate 45 rotate 0 0 S0 stroke grestore} bind def
/D1 {gsave translate 45 rotate 0 0 S1 stroke grestore} bind def
/D2 {gsave translate 45 rotate 0 0 S2 stroke grestore} bind def
/D3 {gsave translate 45 rotate 0 0 S3 stroke grestore} bind def
/D4 {gsave translate 45 rotate 0 0 S4 stroke grestore} bind def
/D5 {gsave translate 45 rotate 0 0 S5 stroke grestore} bind def
/D6 {gsave translate 45 rotate 0 0 S6 stroke grestore} bind def
/D7 {gsave translate 45 rotate 0 0 S7 stroke grestore} bind def
/D8 {gsave translate 45 rotate 0 0 S8 stroke grestore} bind def
/D9 {gsave translate 45 rotate 0 0 S9 stroke grestore} bind def
/D10 {gsave translate 45 rotate 0 0 S10 stroke grestore} bind def
/D11 {gsave translate 45 rotate 0 0 S11 stroke grestore} bind def
/D12 {gsave translate 45 rotate 0 0 S12 stroke grestore} bind def
/D13 {gsave translate 45 rotate 0 0 S13 stroke grestore} bind def
/D14 {gsave translate 45 rotate 0 0 S14 stroke grestore} bind def
/D15 {gsave translate 45 rotate 0 0 S15 stroke grestore} bind def
/DiaE {stroke [] 0 setdash vpt add M
  hpt neg vpt neg V hpt vpt neg V
  hpt vpt V hpt neg vpt V closepath stroke} def
/BoxE {stroke [] 0 setdash exch hpt sub exch vpt add M
  0 vpt2 neg V hpt2 0 V 0 vpt2 V
  hpt2 neg 0 V closepath stroke} def
/TriUE {stroke [] 0 setdash vpt 1.12 mul add M
  hpt neg vpt -1.62 mul V
  hpt 2 mul 0 V
  hpt neg vpt 1.62 mul V closepath stroke} def
/TriDE {stroke [] 0 setdash vpt 1.12 mul sub M
  hpt neg vpt 1.62 mul V
  hpt 2 mul 0 V
  hpt neg vpt -1.62 mul V closepath stroke} def
/PentE {stroke [] 0 setdash gsave
  translate 0 hpt M 4 {72 rotate 0 hpt L} repeat
  closepath stroke grestore} def
/CircE {stroke [] 0 setdash 
  hpt 0 360 arc stroke} def
/Opaque {gsave closepath 1 setgray fill grestore 0 setgray closepath} def
/DiaW {stroke [] 0 setdash vpt add M
  hpt neg vpt neg V hpt vpt neg V
  hpt vpt V hpt neg vpt V Opaque stroke} def
/BoxW {stroke [] 0 setdash exch hpt sub exch vpt add M
  0 vpt2 neg V hpt2 0 V 0 vpt2 V
  hpt2 neg 0 V Opaque stroke} def
/TriUW {stroke [] 0 setdash vpt 1.12 mul add M
  hpt neg vpt -1.62 mul V
  hpt 2 mul 0 V
  hpt neg vpt 1.62 mul V Opaque stroke} def
/TriDW {stroke [] 0 setdash vpt 1.12 mul sub M
  hpt neg vpt 1.62 mul V
  hpt 2 mul 0 V
  hpt neg vpt -1.62 mul V Opaque stroke} def
/PentW {stroke [] 0 setdash gsave
  translate 0 hpt M 4 {72 rotate 0 hpt L} repeat
  Opaque stroke grestore} def
/CircW {stroke [] 0 setdash 
  hpt 0 360 arc Opaque stroke} def
/BoxFill {gsave Rec 1 setgray fill grestore} def
/Density {
  /Fillden exch def
  currentrgbcolor
  /ColB exch def /ColG exch def /ColR exch def
  /ColR ColR Fillden mul Fillden sub 1 add def
  /ColG ColG Fillden mul Fillden sub 1 add def
  /ColB ColB Fillden mul Fillden sub 1 add def
  ColR ColG ColB setrgbcolor} def
/BoxColFill {gsave Rec PolyFill} def
/PolyFill {gsave Density fill grestore grestore} def
/h {rlineto rlineto rlineto gsave fill grestore} bind def
%
% PostScript Level 1 Pattern Fill routine for rectangles
% Usage: x y w h s a XX PatternFill
%	x,y = lower left corner of box to be filled
%	w,h = width and height of box
%	  a = angle in degrees between lines and x-axis
%	 XX = 0/1 for no/yes cross-hatch
%
/PatternFill {gsave /PFa [ 9 2 roll ] def
  PFa 0 get PFa 2 get 2 div add PFa 1 get PFa 3 get 2 div add translate
  PFa 2 get -2 div PFa 3 get -2 div PFa 2 get PFa 3 get Rec
  gsave 1 setgray fill grestore clip
  currentlinewidth 0.5 mul setlinewidth
  /PFs PFa 2 get dup mul PFa 3 get dup mul add sqrt def
  0 0 M PFa 5 get rotate PFs -2 div dup translate
  0 1 PFs PFa 4 get div 1 add floor cvi
	{PFa 4 get mul 0 M 0 PFs V} for
  0 PFa 6 get ne {
	0 1 PFs PFa 4 get div 1 add floor cvi
	{PFa 4 get mul 0 2 1 roll M PFs 0 V} for
 } if
  stroke grestore} def
/languagelevel where
 {pop languagelevel} {1} ifelse
 2 lt
	{/InterpretLevel1 true def}
	{/InterpretLevel1 Level1 def}
 ifelse
%
% PostScript level 2 pattern fill definitions
%
/Level2PatternFill {
/Tile8x8 {/PaintType 2 /PatternType 1 /TilingType 1 /BBox [0 0 8 8] /XStep 8 /YStep 8}
	bind def
/KeepColor {currentrgbcolor [/Pattern /DeviceRGB] setcolorspace} bind def
<< Tile8x8
 /PaintProc {0.5 setlinewidth pop 0 0 M 8 8 L 0 8 M 8 0 L stroke} 
>> matrix makepattern
/Pat1 exch def
<< Tile8x8
 /PaintProc {0.5 setlinewidth pop 0 0 M 8 8 L 0 8 M 8 0 L stroke
	0 4 M 4 8 L 8 4 L 4 0 L 0 4 L stroke}
>> matrix makepattern
/Pat2 exch def
<< Tile8x8
 /PaintProc {0.5 setlinewidth pop 0 0 M 0 8 L
	8 8 L 8 0 L 0 0 L fill}
>> matrix makepattern
/Pat3 exch def
<< Tile8x8
 /PaintProc {0.5 setlinewidth pop -4 8 M 8 -4 L
	0 12 M 12 0 L stroke}
>> matrix makepattern
/Pat4 exch def
<< Tile8x8
 /PaintProc {0.5 setlinewidth pop -4 0 M 8 12 L
	0 -4 M 12 8 L stroke}
>> matrix makepattern
/Pat5 exch def
<< Tile8x8
 /PaintProc {0.5 setlinewidth pop -2 8 M 4 -4 L
	0 12 M 8 -4 L 4 12 M 10 0 L stroke}
>> matrix makepattern
/Pat6 exch def
<< Tile8x8
 /PaintProc {0.5 setlinewidth pop -2 0 M 4 12 L
	0 -4 M 8 12 L 4 -4 M 10 8 L stroke}
>> matrix makepattern
/Pat7 exch def
<< Tile8x8
 /PaintProc {0.5 setlinewidth pop 8 -2 M -4 4 L
	12 0 M -4 8 L 12 4 M 0 10 L stroke}
>> matrix makepattern
/Pat8 exch def
<< Tile8x8
 /PaintProc {0.5 setlinewidth pop 0 -2 M 12 4 L
	-4 0 M 12 8 L -4 4 M 8 10 L stroke}
>> matrix makepattern
/Pat9 exch def
/Pattern1 {PatternBgnd KeepColor Pat1 setpattern} bind def
/Pattern2 {PatternBgnd KeepColor Pat2 setpattern} bind def
/Pattern3 {PatternBgnd KeepColor Pat3 setpattern} bind def
/Pattern4 {PatternBgnd KeepColor Landscape {Pat5} {Pat4} ifelse setpattern} bind def
/Pattern5 {PatternBgnd KeepColor Landscape {Pat4} {Pat5} ifelse setpattern} bind def
/Pattern6 {PatternBgnd KeepColor Landscape {Pat9} {Pat6} ifelse setpattern} bind def
/Pattern7 {PatternBgnd KeepColor Landscape {Pat8} {Pat7} ifelse setpattern} bind def
} def
%
%
%End of PostScript Level 2 code
%
/PatternBgnd {
  TransparentPatterns {} {gsave 1 setgray fill grestore} ifelse
} def
%
% Substitute for Level 2 pattern fill codes with
% grayscale if Level 2 support is not selected.
%
/Level1PatternFill {
/Pattern1 {0.250 Density} bind def
/Pattern2 {0.500 Density} bind def
/Pattern3 {0.750 Density} bind def
/Pattern4 {0.125 Density} bind def
/Pattern5 {0.375 Density} bind def
/Pattern6 {0.625 Density} bind def
/Pattern7 {0.875 Density} bind def
} def
%
% Now test for support of Level 2 code
%
Level1 {Level1PatternFill} {Level2PatternFill} ifelse
/Symbol-Oblique /Symbol findfont [1 0 .167 1 0 0] makefont
dup length dict begin {1 index /FID eq {pop pop} {def} ifelse} forall
currentdict end definefont pop
end
gnudict begin
gsave
0 0 translate
0.050 0.050 scale
0 setgray
newpath
1.000 UL
LTb
700 400 M
63 0 V
6097 0 R
-63 0 V
700 889 M
63 0 V
6097 0 R
-63 0 V
700 1378 M
63 0 V
6097 0 R
-63 0 V
700 1867 M
63 0 V
6097 0 R
-63 0 V
700 2356 M
63 0 V
6097 0 R
-63 0 V
700 2844 M
63 0 V
6097 0 R
-63 0 V
700 3333 M
63 0 V
6097 0 R
-63 0 V
700 3822 M
63 0 V
6097 0 R
-63 0 V
700 4311 M
63 0 V
6097 0 R
-63 0 V
700 4800 M
63 0 V
6097 0 R
-63 0 V
700 400 M
0 63 V
0 4337 R
0 -63 V
1932 400 M
0 63 V
0 4337 R
0 -63 V
3164 400 M
0 63 V
0 4337 R
0 -63 V
4396 400 M
0 63 V
0 4337 R
0 -63 V
5628 400 M
0 63 V
0 4337 R
0 -63 V
6860 400 M
0 63 V
0 4337 R
0 -63 V
700 4800 M
700 400 L
6160 0 V
0 4400 V
-6160 0 V
1.000 UP
stroke
LT0
700 400 M
6 98 V
6 98 V
6 97 V
7 -195 V
6 98 V
6 97 V
6 98 V
6 -195 V
6 97 V
7 98 V
6 98 V
6 -196 V
6 98 V
6 98 V
6 98 V
7 97 V
6 -195 V
6 -196 V
6 98 V
6 98 V
6 98 V
7 97 V
6 98 V
6 98 V
6 98 V
6 98 V
6 97 V
6 98 V
7 98 V
6 98 V
6 -196 V
6 -195 V
6 97 V
6 98 V
7 -195 V
6 97 V
6 -195 V
6 98 V
6 -196 V
6 98 V
7 -196 V
6 98 V
6 -196 V
6 -195 V
6 98 V
6 97 V
7 98 V
6 98 V
6 -196 V
6 -195 V
6 98 V
6 97 V
6 98 V
7 98 V
6 98 V
6 98 V
6 97 V
6 -195 V
6 98 V
7 97 V
6 98 V
6 -195 V
6 -196 V
6 98 V
6 98 V
7 97 V
6 98 V
6 98 V
6 98 V
6 97 V
6 -195 V
7 98 V
6 -196 V
6 -195 V
6 97 V
6 98 V
6 98 V
6 -196 V
7 -195 V
6 98 V
6 -196 V
6 98 V
6 98 V
6 -196 V
7 98 V
6 98 V
6 97 V
6 98 V
6 -195 V
6 97 V
7 98 V
6 98 V
6 98 V
6 97 V
6 98 V
6 -195 V
7 -196 V
6 98 V
6 -196 V
6 -195 V
6 -196 V
6 98 V
6 98 V
7 -196 V
1347 1280 L
6 98 V
6 -196 V
6 -195 V
6 -196 V
7 98 V
6 -196 V
6 98 V
6 98 V
6 -196 V
6 98 V
7 98 V
6 98 V
6 97 V
6 98 V
6 98 V
6 -196 V
7 -195 V
6 -196 V
6 98 V
6 98 V
6 98 V
6 97 V
6 98 V
7 98 V
6 98 V
6 98 V
6 97 V
6 98 V
6 98 V
7 98 V
6 97 V
6 -195 V
6 98 V
6 97 V
6 98 V
7 98 V
6 98 V
6 -196 V
6 98 V
6 -196 V
6 98 V
7 98 V
6 98 V
6 98 V
6 97 V
6 98 V
6 98 V
6 98 V
7 97 V
6 98 V
6 -195 V
6 97 V
6 -195 V
6 98 V
7 -196 V
6 98 V
6 98 V
6 -196 V
6 98 V
6 98 V
7 -196 V
6 98 V
6 -196 V
6 98 V
6 98 V
6 98 V
7 -196 V
6 98 V
6 -196 V
6 -195 V
6 98 V
6 97 V
6 -195 V
7 98 V
6 97 V
6 98 V
6 98 V
6 -196 V
6 98 V
7 -195 V
6 97 V
6 -195 V
6 98 V
6 97 V
6 98 V
7 98 V
6 -196 V
6 98 V
6 98 V
6 98 V
6 97 V
7 -195 V
6 -196 V
6 -195 V
6 -196 V
6 98 V
6 98 V
6 98 V
7 -196 V
6 -196 V
6 98 V
6 98 V
6 98 V
6 -196 V
1994 1867 L
6 97 V
6 -195 V
6 98 V
6 97 V
6 -195 V
7 98 V
6 97 V
6 98 V
6 98 V
6 98 V
6 98 V
7 97 V
6 98 V
6 -195 V
6 -196 V
6 98 V
6 -196 V
6 98 V
7 -196 V
6 98 V
6 -195 V
6 -196 V
6 98 V
6 98 V
7 97 V
6 98 V
6 98 V
6 -196 V
6 98 V
6 98 V
7 -196 V
6 98 V
6 -195 V
6 97 V
6 98 V
6 98 V
7 98 V
6 98 V
6 97 V
6 98 V
6 98 V
6 -196 V
6 -195 V
7 98 V
6 97 V
6 98 V
6 98 V
6 -196 V
6 98 V
7 -195 V
6 97 V
6 98 V
6 98 V
6 -196 V
6 98 V
7 98 V
6 98 V
6 -196 V
6 -195 V
6 -196 V
6 98 V
7 -196 V
6 98 V
6 -196 V
6 98 V
6 98 V
6 -196 V
6 -195 V
7 98 V
6 97 V
6 98 V
6 98 V
6 98 V
6 98 V
7 97 V
6 98 V
6 98 V
6 -196 V
6 -195 V
6 -196 V
7 -195 V
6 97 V
6 -195 V
6 -196 V
6 -195 V
6 98 V
7 97 V
6 98 V
6 98 V
6 98 V
6 97 V
6 -195 V
6 98 V
7 97 V
6 -195 V
6 98 V
6 97 V
6 98 V
6 98 V
7 98 V
6 -196 V
6 98 V
6 98 V
6 98 V
2640 2160 L
7 98 V
6 98 V
6 -196 V
6 98 V
6 -196 V
6 98 V
7 -196 V
6 98 V
6 98 V
6 98 V
6 98 V
6 97 V
6 98 V
7 98 V
6 -196 V
6 -195 V
6 98 V
6 -196 V
6 98 V
7 98 V
6 97 V
6 98 V
6 98 V
6 -196 V
6 98 V
7 98 V
6 98 V
6 97 V
6 98 V
6 98 V
6 98 V
7 -196 V
6 -195 V
6 -196 V
6 98 V
6 98 V
6 97 V
6 98 V
7 98 V
6 98 V
6 -196 V
6 98 V
6 98 V
6 98 V
7 97 V
6 -195 V
6 98 V
6 97 V
6 98 V
6 -195 V
7 -196 V
6 98 V
6 -196 V
6 98 V
6 -196 V
6 98 V
7 98 V
6 98 V
6 -196 V
6 -195 V
6 97 V
6 -195 V
6 -196 V
7 98 V
6 98 V
6 -196 V
6 -195 V
6 98 V
6 97 V
7 98 V
6 -195 V
6 97 V
6 -195 V
6 -196 V
6 98 V
7 -196 V
6 98 V
6 98 V
6 98 V
6 98 V
6 97 V
7 98 V
6 98 V
6 98 V
6 97 V
6 98 V
6 98 V
6 98 V
7 -196 V
6 -195 V
6 -196 V
6 98 V
6 98 V
6 -196 V
7 98 V
6 -196 V
6 98 V
6 -195 V
6 -196 V
6 98 V
7 98 V
6 97 V
6 98 V
6 98 V
3287 2747 L
6 97 V
7 98 V
6 98 V
6 98 V
6 -196 V
6 98 V
6 98 V
6 -196 V
7 98 V
6 98 V
6 98 V
6 97 V
6 -195 V
6 98 V
7 97 V
6 98 V
6 98 V
6 98 V
6 97 V
6 -195 V
7 98 V
6 -196 V
6 98 V
6 98 V
6 97 V
6 -195 V
7 -196 V
6 98 V
6 -195 V
6 -196 V
6 98 V
6 -196 V
6 98 V
7 98 V
6 -196 V
6 98 V
6 98 V
6 -196 V
6 98 V
7 98 V
6 98 V
6 97 V
6 98 V
6 -195 V
6 97 V
7 98 V
6 -195 V
6 97 V
6 -195 V
6 98 V
6 97 V
7 98 V
6 98 V
6 98 V
6 97 V
6 98 V
6 -195 V
6 -196 V
7 98 V
6 98 V
6 97 V
6 98 V
6 -195 V
6 97 V
7 98 V
6 -195 V
6 97 V
6 -195 V
6 -196 V
6 -195 V
7 -196 V
6 -195 V
6 97 V
6 -195 V
6 -196 V
6 98 V
7 -195 V
6 -196 V
6 98 V
6 98 V
6 97 V
6 98 V
6 -195 V
7 -196 V
6 98 V
6 98 V
6 -196 V
6 98 V
6 98 V
7 -196 V
6 -196 V
6 98 V
6 98 V
6 98 V
6 98 V
7 97 V
6 98 V
6 -195 V
6 97 V
6 98 V
6 -195 V
7 97 V
6 98 V
6 -195 V
3934 2453 L
6 98 V
6 98 V
6 -196 V
7 98 V
6 98 V
6 98 V
6 97 V
6 98 V
6 -195 V
7 97 V
6 -195 V
6 -196 V
6 98 V
6 98 V
6 98 V
7 97 V
6 98 V
6 -195 V
6 97 V
6 -195 V
6 98 V
7 97 V
6 98 V
6 98 V
6 -196 V
6 98 V
6 98 V
6 98 V
7 -196 V
6 -195 V
6 97 V
6 -195 V
6 98 V
6 97 V
7 -195 V
6 98 V
6 97 V
6 -195 V
6 -196 V
6 98 V
7 98 V
6 -196 V
6 98 V
6 98 V
6 98 V
6 -196 V
7 98 V
6 98 V
6 97 V
6 98 V
6 98 V
6 98 V
6 98 V
7 97 V
6 98 V
6 98 V
6 98 V
6 -196 V
6 98 V
7 -196 V
6 98 V
6 -195 V
6 97 V
6 98 V
6 -195 V
7 97 V
6 98 V
6 -195 V
6 97 V
6 -195 V
6 98 V
7 97 V
6 98 V
6 98 V
6 98 V
6 97 V
6 -195 V
6 98 V
7 -196 V
6 -195 V
6 97 V
6 98 V
6 98 V
6 98 V
7 -196 V
6 -195 V
6 97 V
6 98 V
6 98 V
6 -196 V
7 -195 V
6 -196 V
6 -195 V
6 -196 V
6 -195 V
6 97 V
7 98 V
6 -195 V
6 -196 V
6 -196 V
6 -195 V
6 98 V
6 -196 V
7 98 V
4581 1573 L
6 98 V
6 -195 V
6 97 V
6 98 V
7 98 V
6 -196 V
6 -195 V
6 98 V
6 -196 V
6 -196 V
7 -195 V
6 98 V
6 97 V
6 98 V
6 98 V
6 98 V
7 98 V
6 -196 V
6 98 V
6 -196 V
6 98 V
6 98 V
6 -196 V
7 98 V
6 -196 V
6 98 V
6 98 V
6 98 V
6 98 V
7 97 V
6 98 V
6 98 V
6 98 V
6 97 V
6 -195 V
7 -196 V
6 -195 V
6 98 V
6 97 V
6 98 V
6 98 V
7 -196 V
6 98 V
6 -195 V
6 97 V
6 98 V
6 98 V
6 98 V
7 97 V
6 98 V
6 98 V
6 98 V
6 -196 V
6 -195 V
7 -196 V
6 98 V
6 98 V
6 97 V
6 -195 V
6 98 V
7 -196 V
6 -195 V
6 97 V
6 98 V
6 98 V
6 -196 V
7 -195 V
6 98 V
6 -196 V
6 98 V
6 98 V
6 97 V
6 -195 V
7 98 V
6 97 V
6 98 V
6 -195 V
6 97 V
6 -195 V
7 98 V
6 -196 V
6 98 V
6 98 V
6 97 V
6 -195 V
7 -196 V
6 98 V
6 98 V
6 -196 V
6 98 V
6 98 V
7 98 V
6 -196 V
6 98 V
6 -196 V
6 -195 V
6 97 V
6 98 V
7 98 V
6 98 V
6 98 V
6 97 V
6 98 V
6 98 V
5228 1867 L
6 -196 V
6 98 V
6 98 V
6 97 V
6 98 V
7 -195 V
6 97 V
6 98 V
6 98 V
6 98 V
6 -196 V
7 -195 V
6 97 V
6 98 V
6 98 V
6 98 V
6 98 V
6 -196 V
7 98 V
6 98 V
6 97 V
6 98 V
6 98 V
6 98 V
7 97 V
6 98 V
6 98 V
6 98 V
6 -196 V
6 -195 V
7 97 V
6 -195 V
6 98 V
6 97 V
6 -195 V
6 98 V
7 -196 V
6 98 V
6 98 V
6 97 V
6 98 V
6 98 V
6 98 V
7 -196 V
6 98 V
6 98 V
6 98 V
6 -196 V
6 98 V
7 98 V
6 97 V
6 98 V
6 98 V
6 -196 V
6 98 V
7 98 V
6 98 V
6 97 V
6 98 V
6 98 V
6 -196 V
7 98 V
6 98 V
6 98 V
6 98 V
6 97 V
6 -195 V
6 98 V
7 97 V
6 98 V
6 -195 V
6 -196 V
6 98 V
6 98 V
7 97 V
6 -195 V
6 98 V
6 97 V
6 -195 V
6 -196 V
7 -195 V
6 97 V
6 -195 V
6 98 V
6 97 V
6 98 V
7 98 V
6 98 V
6 98 V
6 97 V
6 98 V
6 98 V
6 -196 V
7 98 V
6 98 V
6 -196 V
6 -195 V
6 98 V
6 97 V
7 -195 V
6 -196 V
6 98 V
6 -196 V
6 98 V
5874 3920 L
7 98 V
6 -196 V
6 98 V
6 98 V
6 98 V
6 97 V
7 -195 V
6 -196 V
6 98 V
6 98 V
6 98 V
6 -196 V
6 -196 V
7 98 V
6 98 V
6 98 V
6 -196 V
6 -195 V
6 97 V
7 98 V
6 98 V
6 98 V
6 98 V
6 -196 V
6 98 V
7 98 V
6 -196 V
6 98 V
6 -196 V
6 -195 V
6 -196 V
7 98 V
6 98 V
6 97 V
6 98 V
6 98 V
6 98 V
6 98 V
7 97 V
6 98 V
6 98 V
6 -196 V
6 -195 V
6 98 V
7 97 V
6 98 V
6 -195 V
6 97 V
6 -195 V
6 98 V
7 97 V
6 98 V
6 98 V
6 98 V
6 -196 V
6 -195 V
7 -196 V
6 98 V
6 98 V
6 -196 V
6 -196 V
6 98 V
6 -195 V
7 97 V
6 -195 V
6 -196 V
6 98 V
6 98 V
6 -196 V
7 98 V
6 98 V
6 98 V
6 97 V
6 98 V
6 98 V
7 -196 V
6 98 V
6 98 V
6 98 V
6 98 V
6 97 V
7 98 V
6 -195 V
6 -196 V
6 98 V
6 98 V
6 97 V
6 98 V
7 -195 V
6 97 V
6 98 V
6 98 V
6 98 V
6 -196 V
7 98 V
6 -196 V
6 98 V
6 -195 V
6 97 V
6 98 V
7 98 V
6 98 V
6 -196 V
6 98 V
6521 4507 L
6 -196 V
7 98 V
6 -196 V
6 -195 V
6 98 V
6 -196 V
6 -196 V
6 98 V
7 -195 V
6 97 V
6 98 V
6 98 V
6 98 V
6 -196 V
7 98 V
6 -196 V
6 -195 V
6 98 V
6 -196 V
6 98 V
7 98 V
6 97 V
6 -195 V
6 98 V
6 97 V
6 98 V
7 98 V
6 -196 V
6 98 V
6 98 V
6 98 V
6 98 V
6 97 V
7 98 V
6 98 V
6 98 V
6 -196 V
6 -195 V
6 97 V
7 -195 V
6 -196 V
6 98 V
6 98 V
6 98 V
6 97 V
7 98 V
6 98 V
6 -196 V
6 98 V
6 -195 V
6 97 V
7 98 V
6 98 V
6 98 V
6 97 V
stroke
LTb
700 4800 M
700 400 L
6160 0 V
0 4400 V
-6160 0 V
1.000 UP
stroke
grestore
end
showpage
  }}%
  \put(6860,200){\makebox(0,0){\strut{} 1000}}%
  \put(5628,200){\makebox(0,0){\strut{} 800}}%
  \put(4396,200){\makebox(0,0){\strut{} 600}}%
  \put(3164,200){\makebox(0,0){\strut{} 400}}%
  \put(1932,200){\makebox(0,0){\strut{} 200}}%
  \put(700,200){\makebox(0,0){\strut{} 0}}%
  \put(580,4800){\makebox(0,0)[r]{\strut{} 45}}%
  \put(580,4311){\makebox(0,0)[r]{\strut{} 40}}%
  \put(580,3822){\makebox(0,0)[r]{\strut{} 35}}%
  \put(580,3333){\makebox(0,0)[r]{\strut{} 30}}%
  \put(580,2844){\makebox(0,0)[r]{\strut{} 25}}%
  \put(580,2356){\makebox(0,0)[r]{\strut{} 20}}%
  \put(580,1867){\makebox(0,0)[r]{\strut{} 15}}%
  \put(580,1378){\makebox(0,0)[r]{\strut{} 10}}%
  \put(580,889){\makebox(0,0)[r]{\strut{} 5}}%
  \put(580,400){\makebox(0,0)[r]{\strut{} 0}}%
\end{picture}%
\endgroup
 

%% file: fewpaths.pstex_t
\begin{picture}(0,0)%
\includegraphics{fewpaths.pstex}%
\end{picture}%
\setlength{\unitlength}{3947sp}%
\begingroup\makeatletter\ifx\SetFigFont\undefined%
\gdef\SetFigFont#1#2#3#4#5{%
  \reset@font\fontsize{#1}{#2pt}%
  \fontfamily{#3}\fontseries{#4}\fontshape{#5}%
  \selectfont}%
\fi\endgroup%
\begin{picture}(3012,1566)(901,-2215)
\end{picture}%

%% file: cycles.pstex_t
\begin{picture}(0,0)%
\includegraphics{cycles.pstex}%
\end{picture}%
\setlength{\unitlength}{3947sp}%
\begingroup\makeatletter\ifx\SetFigFont\undefined%
\gdef\SetFigFont#1#2#3#4#5{%
  \reset@font\fontsize{#1}{#2pt}%
  \fontfamily{#3}\fontseries{#4}\fontshape{#5}%
  \selectfont}%
\fi\endgroup%
\begin{picture}(6699,648)(739,-1360)
\end{picture}%